%% file: BKPsqpinvit.tex
\documentclass[a4paper,11pt,oneside]{amsart}
\usepackage[top=2.2cm,left=2.8cm,right=2.8cm,bottom=2.8cm]{geometry}
\usepackage[foot]{amsaddr}

\usepackage[table]{xcolor}
\usepackage{booktabs}
\usepackage{tabularx}
\usepackage{tabu}

\usepackage{graphicx}

\include{rwth-colors}
\usepackage{hyperref} %
\hypersetup{
    colorlinks,
    citecolor=grun,
    filecolor=black,
    linkcolor=rwth,
    urlcolor=blue
}
\hypersetup{linktocpage}

\usepackage{cite}

\usepackage[ruled]{algorithm2e} %

\SetCommentSty{mycommfont}

\usepackage{cleveref}
\crefname{equation}{}{}

\usepackage{amsbsy,bbm}
\usepackage{latexsym}
\usepackage{amsfonts}
\usepackage{amssymb}
\usepackage{amsmath,amsthm,stackengine}
\usepackage{enumerate}

\usepackage{tikz}
 \usetikzlibrary{arrows.meta}
\usetikzlibrary{arrows}
\usetikzlibrary{matrix}

\usepackage{mathdots}
\usepackage{mathtools}
\usepackage{mathdots}
\usepackage{fancybox}

\usepackage{scalerel}

\usepackage{stmaryrd}
\usepackage{dsfont}
\usepackage{bm}

\usepackage[scr=euler,scrscaled=1.00]{mathalfa}
\def\TODO{{\color{red}\mbox{TODO}}}

\usepackage{pgfplots}
\usepackage{import} %
\newif\ifuseprecompiled
\useprecompiledtrue

\usepackage{placeins} %

\newcommand{\tikzfigure}[4]%
 {%
 \begin{figure}[#1]%
  \begin{center}%
   \scalebox{0.75}{ %
    \ifuseprecompiled%
    \includegraphics{pdf_files/#3.pdf}%
    \else
    \tikzsetnextfilename{#3}
    \IfFileExists{tikz_base_files/#3.tikz}{
    \subimport{tikz_base_files/}{#3.tikz}
    }{\TODO{}}
    \fi
    }
  \end{center}%
  \vspace{-0.5cm}
  \caption{#4\label{fig:#3}}%
\end{figure}\noindent
}

\newcommand{\nolabeltikzfigure}[4]%
 {%
 \begin{figure}[#1]%
  \begin{center}%
   \scalebox{0.75}{ %
    \ifuseprecompiled%
    \includegraphics{pdf_files/#3.pdf}%
    \else
    \tikzsetnextfilename{#3}
    \IfFileExists{tikz_base_files/#3.tikz}{
    \subimport{tikz_base_files/}{#3.tikz}
    }{\TODO{}}
    \fi
    }
  \end{center}%
  \vspace{-0.5cm}
  \caption{#4}%
\end{figure}\noindent
}

\usetikzlibrary{arrows.meta}
\usetikzlibrary{backgrounds}
\usepgfplotslibrary{patchplots}
\usepgfplotslibrary{fillbetween}
\pgfplotsset{%
    layers/standard/.define layer set={%
        background,axis background,axis grid,axis ticks,axis lines,axis tick labels,pre main,main,axis descriptions,axis foreground%
    }{
        grid style={/pgfplots/on layer=axis grid},%
        tick style={/pgfplots/on layer=axis ticks},%
        axis line style={/pgfplots/on layer=axis lines},%
        label style={/pgfplots/on layer=axis descriptions},%
        legend style={/pgfplots/on layer=axis descriptions},%
        title style={/pgfplots/on layer=axis descriptions},%
        colorbar style={/pgfplots/on layer=axis descriptions},%
        ticklabel style={/pgfplots/on layer=axis tick labels},%
        axis background@ style={/pgfplots/on layer=axis background},%
        3d box foreground style={/pgfplots/on layer=axis foreground},%
    },
}
\pgfplotsset{compat=1.17} %

\pgfplotsset{
colormap={plots1}{rgb(0.00000000)=(1.00000000,1.00000000,1.00000000)
rgb(0.20000000)=(0.79215686,0.90588235,0.90588235)
rgb(0.40000000)=(0.67064087,0.79775559,0.84868949)
rgb(0.60000000)=(0.51423027,0.66829363,0.78644869)
rgb(0.80000000)=(0.25098039,0.49803922,0.71764706)
rgb(1.00000000)=(0.38039216,0.12941176,0.34509804)},
}

\usepackage[colorinlistoftodos, textwidth=23mm]{todonotes} %
\let\oldtodo\todo
\makeatletter
\newcommand\mytodo{%
\if@endpe\par\fi
\oldtodo}
\makeatother

\DeclareMathOperator{\arccot}{arccot}
\DeclareMathOperator{\rank}{rank}
\DeclareMathOperator{\ranks}{ranks}

\DeclareMathOperator{\Span}{span}

\DeclareMathOperator{\tr}{tr}

\newcommand\diag{\mathop{\rm diag}}
\newcommand\bidiag{\mathop{\rm bidiag}}

\newcommand\supp{\mathop{\rm supp}}

\providecommand{\abs}[1]{\lvert#1\rvert}

\providecommand{\biggabs}[1]{\biggl\lvert#1\biggr\rvert}

\providecommand{\norm}[1]{\lVert#1\rVert}

\newtheorem{theorem}{Theorem}
\newtheorem{lemma}[theorem]{Lemma}
\newtheorem{prop}[theorem]{Proposition}
\newtheorem{cor}[theorem]{Corollary}

\theoremstyle{definition}

\newtheorem{example}[theorem]{Example}

\theoremstyle{remark}
\newtheorem{remark}[theorem]{Remark}

\crefname{algorithm}{Algorithm}{Algorithms}
\crefname{remark}{Remark}{Remarks}
\crefname{lemma}{Lemma}{Lemmas}
\crefname{prop}{Proposition}{Propositions}
\crefname{theorem}{Theorem}{Theorems}
\crefname{cor}{Corollary}{Corollaries}
\crefname{section}{Section}{Sections}
\crefname{subsection}{Section}{Sections}
\crefname{subsubsection}{Section}{Sections}
\crefname{appendix}{Appendix}{Appendices}
\crefname{figure}{Figure}{Figures}

\newcommand{\cK}{{\mathcal{K}}}

\newcommand{\cF}{{\mathcal{F}}}

\newcommand{\A}{\ten A}
\newcommand{\E}{\ten E}

\newcommand{\sdd}{\,\mathrm{d}}

\newcommand{\Chi}{\raise .3ex
\hbox{\large $\chi$}}

\newcommand{\R}{\mathbb{R}}
\newcommand{\N}{\mathbb{N}}

\newcommand{\ten}[1]{\boldsymbol{#1}}
\newcommand{\rep}[1]{\mathsf{#1}}

\newcommand{\sst}[1]{{\scaleto{#1}{4.5pt}}}

\newcommand{\rmap}[1]{\tau(#1)}

\newcommand{\rmapless}[3]{\tau^\sst{<}_{#1,#2} (#3)}

\newcommand{\rmapgtr}[3]{\tau^\sst{>}_{#1,#2} (#3)}

\DeclareMathOperator*{\SKP}{\Join}

\newcommand{\unocc}[1]{\colorbox{gray!30}{\makebox(8,8){$\mathsf{#1}$}}}
\newcommand{\occ}[1]{\colorbox{black}{\textcolor{white}{\makebox(8,8){$\mathsf{#1}$}}}}
\newcommand{\neut}[2]{
\text{\raisebox{-.2cm}{
\begin{tikzpicture}
\pgfmathsetmacro{\cubex}{.48}
\pgfmathsetmacro{\cubey}{.48}
\draw[black,fill=white] (0,0) -- ++(-\cubex,0,0) -- ++(0,-\cubey,0) -- ++(\cubex,0,0) -- cycle;
\node at (-.22,-.25) {$\mathsf{#1_#2}$};
\end{tikzpicture}
}}}
\newcommand{\neutlift}[1]{
\text{\raisebox{-.2cm}{
\begin{tikzpicture}
\pgfmathsetmacro{\cubex}{.45}
\pgfmathsetmacro{\cubey}{.45}
\pgfmathsetmacro{\cubez}{.2}
\draw[black,fill=white] (0,0,0) -- ++(-\cubex,0,0) -- ++(0,-\cubey,0) -- ++(\cubex,0,0) -- cycle;
\draw[black,fill=gray!30] (0,0,0) -- ++(0,0,-\cubez) -- ++(0,-\cubey,0) -- ++(0,0,\cubez) -- cycle;
\draw[black,fill=gray!30] (0,0,0) -- ++(-\cubex,0,0) -- ++(0,0,-\cubez) -- ++(\cubex,0,0) -- cycle;
\node at (-.37,-.37,-.37) {$\mathsf{#1}$};
\end{tikzpicture}
}}}

\newcommand{\enum}{\ensuremath{\varepsilon_{\mathrm{num}}}}
\newcommand{\eiter}{\ensuremath{\varepsilon_{\mathrm{iter}}}}
\newcommand{\PINVIT}{\ensuremath{\operatorname{PINVIT}}}
\newcommand{\trp}{\top}
\newcommand{\etain}{\eta_{\mathrm{inner}}}
\newcommand{\etares}{\eta_{\mathrm{res}}}
\newcommand{\etaout}{\eta_{\mathrm{outer}}}
\newcommand{\teig}{t_{\mathrm{eig}}}
\newcommand{\zetares}{\zeta_{\mathrm{res}}}
\newcommand{\bfun}{b}

\input{numbers/exp11_KlowD1_data.txt}
\input{numbers/exp16_KevenhigherD1_data.txt}
\input{numbers/exp41_KlowD4_data.txt}
\input{numbers/exp46_KevenhigherD4_data.txt}

\SetArgSty{textup}

\numberwithin{equation}{section}
\numberwithin{theorem}{section}

\title[Low-Rank Eigenvalue Solvers for Block-Sparse Matrix Product States]{Low-Rank Eigenvalue Solvers for Block-Sparse Matrix Product States}
\author{Markus Bachmayr$^1$}
\address{\rm $^1$ Institut f\"ur Geometrie und Praktische Mathematik, RWTH Aachen University, Templergraben 55, 52062 Aachen, Germany.}
\email[Markus Bachmayr]{bachmayr@igpm.rwth-aachen.de}
\author{Sebastian Kr\"amer$^1$}
\email[Sebastian Kr\"amer]{kraemer@igpm.rwth-aachen.de}
\author{Max Pfeffer$^2$}
\address{\rm $^2$ Institut f\"ur Numerische und Angewandte Mathematik, Georg-August Universität G\"ottingen, Wilhelmsplatz 1, 37073 G\"ottingen, Germany.}
\email[Max Pfeffer]{m.pfeffer@math.uni-goettingen.de}
\thanks{M.B.\ acknowledges funding by
 Deutsche Forschungsgemeinschaft (DFG, German Research Foundation) -- Projektnummern 233630050, 211504053 -- TRR 146, SFB 1060, and by the European Union (ERC, COCOA, 101170147). 
M.P. was partially funded by DFG -- Projektnummer 448293816.}

\date{\today}

\begin{document}

\maketitle

\vspace{-18pt}
\begin{abstract}
We consider an iterative eigensolver for Schr\"odinger equations that constructs low-rank approximations of eigenfunctions with accuracy-adapted ranks, with particular focus on fermionic Schr\"odinger equations in second-quantized form and on matrix product state approximations enforcing particle number conservation.
We provide a complete analysis of a solver based on preconditioned inverse iteration combined with rank truncation and propose a generalization to subspace iteration for the joint approximation of several eigenspaces.
The practical performance of the method is illustrated by numerical tests for several model problems.
 

\end{abstract}

\section{Introduction}

In this paper, we consider eigenvalue solvers operating with convergence guarantees  on low-rank tensor representations, with particular focus on Hamiltonians associated to fermionic quantum systems in second-quantized representation. 
The fermionic nature of the particles entails antisymmetry of the sought eigenfunctions under exchange of particle coordinates. Such functions are commonly approximated by antisymmetrized tensor products of single-particle basis functions, which in this context are called \emph{orbitals}.
Elementary antisymmetrized tensor products, called \emph{Slater determinants}, are determined by selections of distinct orbitals. 

In second quantization, linear combinations of Slater determinants are parameterized by \emph{occupation numbers:} each selection of orbitals corresponds to a set of binary indices into a tensor holding the coefficients of the linear combination.
In many cases of interest, these high-order occupation number tensors are observed to have efficient low-rank approximations in terms of matrix product states (MPS, also known as tensor trains) or more general tree tensor networks. 
Constrainting the problem to a fixed particle number leads to certain block-sparse structures in the components of matrix product states.
 
The eigensolvers in low-rank format that we consider here are geared towards the particular setting of second-quantized formulations and block sparsity with particular attention to structure of operator representations, but the basic elements of the analysis given here may be of independent interest.
We obtain a method that is guaranteed to converge in terms of the $H^1$-norm angle with respect to the true eigenspace. At the same time, the achieved error bounds are balanced with the ranks of MPS approximations, where we obtain a near-optimal relation to ranks of best approximations of comparable accuracy.
This is achieved by a combination of error reduction based on preconditioned inverse iteration with adaptive rank truncation.

\subsection{Eigenvalue problems for Hamiltonians in second quantization}

Our considerations are motivated by applications in quantum chemistry, where the central object of study are eigenvalue problems of the electronic Schr\"odinger Hamiltonian, which for a system with $N$ electrons takes the form
\begin{equation}\label{eq:hamilt}
   H = - \frac12 \Delta + \sum_{i=1}^N V(x_i) + \frac12 \sum_{i \neq j} \frac{1}{\norm{x_i - x_j}_2}
\end{equation}
with electron coordinates $x_1, \ldots, x_N \in \R^3$,
where $V$ is an external potential, such as a Coulomb potential corresponding to the electrostatic attraction of atomic nuclei.
We make the standard assumption (which holds true in the absence of magnetic fields) of real-valued $V$, which implies that the sought eigenvalues and eigenfunctions are also real-valued.

Due to the fermionic nature of electrons, the physically relevant eigenfunctions, called wave functions, additionally need to satisfy certain antisymmetry conditions. There we need to take into account the additional spin degree of freedom of each particle, which can take the values $\pm \frac12$; the complete wave functions are determined by their spin components that are functions on $\R^{3N}$. Each such spin component $u$ needs to satisfy the antisymmetry property
\[
   u(\ldots, x_i , \ldots, x_j, \ldots) = - u(\ldots, x_j, \ldots, x_i, \ldots)
\]
whenever $i \neq j$ correspond to particles of equal spin.
For simplicity, in what follows we will always assume spin configurations where all spins are equal, so that the sought eigenfunctions are antisymmetric under exchange of any pair of particle coordinates in $\R^3$.
With this restriction, the main aim is then to find $u \in H^1(\R^{3N}) \cap \bigl(\bigwedge_{i=1}^N L^2(\R^3)\bigr) \subset H^1(\R^{3N})$, $u \neq 0$, such that $H u = \lambda u$ with minimal $\lambda \in \R$.

We assume a family of orbitals $\phi_i \in H^1(\R^3)$, $i \in \N$, that are an orthonormal basis of $L^2(\R^3)$. The Slater determinants $\phi_{i_1} \wedge \phi_{i_2} \wedge \ldots \wedge \phi_{i_N}$ for $i_1 < i_2 < \ldots < i_N$ are then an orthonormal basis of $\bigwedge_{i=1}^N L^2(\R^3)$.
In what follows, we consider approximations using the $K$ orbitals $\{ \phi_{1},\ldots, \phi_{K} \}$.
Any linear combination of Slater determinants formed from these $K$ orbitals, that is,
\[
   \sum_{1 \leq i_1 < \ldots < i_N \leq K } c_i\, \phi_{i_1} \wedge  \ldots \varphi_{i_N}
\]
with $c_{i} \in \R$, can be represented as a tensor of occupation numbers $\ten x \in \mathcal{F}^K$, where
\[  \mathcal{F}^K = (\R^2)^{\otimes K}  \simeq \R^{(2^K)}.  \]
This is achieved by identifying $i \in \{ 1, \ldots, K\}^N$, $i_1 < \ldots < i_N$, with precisely one $\alpha \in \{ 0, 1 \}^K$ by $\alpha_k = 1$ if $k \in \{ i_1, \ldots, i_N \}$, and $\alpha_k= 0$ otherwise for $k=1,\ldots,K$, and defining the nonzero entries of $\ten x$ by $\ten x_\alpha = c_i$ for such $\alpha$. 

We have the corresponding second-quantized representation of the Hamiltonian acting on occupation number tensors,
\begin{equation}\label{eq:hamil}
    \ten{H} = \ten T + \ten V = \sum_{i,j = 1}^K t_{ij} \ten{a}_{i}^\ast \ten{a}_{j}  +  \sum_{i,j,k,l = 1}^K v_{ijkl} \ten{a}^\ast_{i} \ten{a}^\ast_{j} \ten{a}_{k} \ten{a}_{l},
\end{equation}
 in terms of the \emph{annihilation operator} $\ten{a}_i$
\begin{equation}\label{def:anniloperator}
  \ten{a}_i =  \biggl( \bigotimes_{k=1}^{i-1} S \biggr) \otimes A \otimes \biggl( \bigotimes_{k =  i + 1}^K I \biggr)
\end{equation}
and the corresponding \emph{creation operator} $\ten{a}_i^\ast$. These are given in terms of the elementary components
 \begin{equation}\label{eq:elementarycomponents}
  S = \begin{pmatrix} 1 & 0 \\ 0 & -1 \end{pmatrix}, \quad
   A = \begin{pmatrix} 0 & 1 \\ 0 & 0 \end{pmatrix}, \quad 
   I = \begin{pmatrix} 1 & 0 \\ 0 & 1 \end{pmatrix}.
\end{equation}
The operators $\ten{a}^\ast_i$ and $\ten{a}_i$ can be thought of as transferring particles from the unoccupied to the occupied state of orbital $i$ or back, respectively. The antisymmetry of wavefunctions corresponds to the anticommutation relations
 \begin{equation}\label{eq:anticomm}
     \ten{a}_{i} \ten{a}^\ast_{j} + \ten{a}^\ast_{j} \ten{a}_{i} = \delta_{ij}, \quad   \ten{a}^\ast_{i} \ten{a}^\ast_{j} + \ten{a}^\ast_{j} \ten{a}^\ast_{i} =   \ten{a}_{i} \ten{a}_{j} + \ten{a}_{j} \ten{a}_{i} = 0\,.
 \end{equation}

 Neglecting electron spin for the moment, the coefficient tensors in \eqref{eq:hamil} are given by
 \begin{equation}\label{eq:tvdef}
 \begin{aligned}
   t_{ij} &=  \frac12  \int_{\R^3}  \nabla\phi_i \cdot\nabla \phi_j \sdd x + \int_{\R^3} V \phi_i \phi_j \sdd x, \\ v_{ijkl} &= \frac12 \int_{\R^3} \int_{\R^3}  \phi_i(x_1) \phi_j(x_2) \frac{1}{\norm{x_1 - x_2}_2} \phi_k(x_2) \phi_l(x_1) \sdd x_1 \sdd x_2\,.
   \end{aligned}
 \end{equation}
 Although in this paper, we test our methods on model problems with particles living in a one-dimensional space, they can be generalized immediately to the above general case.
 
A common strategy for handling large $K$ is to approximate elements of $\mathcal{F}^K$ in low-rank tensor formats. One instance are MPS \cite{white,vidal03}, which especially in the mathematical literature are also known as the tensor train (TT) format \cite{Oseledets:2011:TT}, which in turn are a variant of hierarchical tensors \cite{Hackbusch:09,Hackbusch:12tensorspaces}.  
Using tensor formats in this manner has the advantage that the occupation number tensors do not need to satisfy antisymmetry requirements, which would be difficult to incorporate directly \cite{Hackbusch:18}.

However, the second-quantized representation \eqref{eq:hamil} of the Hamiltonian is valid for any number of particles, and thus this number generally needs to be explicitly enforced.
Prescribing a number of $N$ particles amounts to restricting the eigenvalue problem for $\ten{H}$ to the subspace $\mathcal{F}_N^K$ of those occupation numbers that are also eigenvectors of the
\emph{particle number operator}
 \[
  \ten{P} = \sum_{i=1}^K \ten{a}_i^\ast \ten{a}_i
\]
with eigenvalue $N$. The particle number constraint does not need to be implemented explicitly: as considered in detail in \cite{BGP:22}, every particle number eigenspace corresponds to a certain \emph{block-sparse structure} in the cores of MPS. This fact has a long history in the physical literature (see, e.g., \cite{OR95,Daley:04,McCulloch:07,Schollwoeck:11,SPV11}). The block sparsity is associated to gauge symmetries of the Hamiltonian, such as $U(1)$ symmetry corresponding to particle number conservation.

An approximate solution for the lowest eigenvalue (or eigenvalues)  with $N$ particles of the Hamiltonian $\ten H$ is then sought within the subspace $\mathcal{F}_N^K \subset \mathcal{F}^K$ of $N$-particle states,
\[
 \mathcal{F}_N^K = \bigl\{ \ten{x} \in \mathcal{F}^K\colon   \ten P\ten x = N \ten x \bigr\} ,
\]
which leads to the eigenvalue problem
\begin{equation}\label{eq:Hevp}
 \langle \ten H \ten x, \ten y \rangle = \lambda \langle \ten x, \ten y \rangle \quad\text{for all $\ten y \in \mathcal{F}_N^K$}
\end{equation}
for $\ten x \in \mathcal{F}_N^K$, $\ten x \neq 0$. Here and in what follows, we use $\langle \cdot , \cdot \rangle$ to denote the standard Euclidean inner product and $\norm{\cdot}$ for the induced vector norm as well as the corresponding matrix norm.

\subsection{Novelty and relation to previous work}

We build on \cite{BGP:22} in using the block-sparsity pattern of MPS representations of element of $\mathcal{F}_N^K$. In our present setting, all involved operations preserve the particle number and thus again yield tensors that have a representation with the same block structure. Thus eigenvalue solvers can be implemented to operate only on the separate blocks in the tensor representation, which may reduce the computational costs. 

The most commonly used numerical method for solving \eqref{eq:Hevp} with $\ten x$ in MPS representation is the \emph{density matrix renormalization group} (DMRG) algorithm \cite{white,vidal03}. Although this alternating minimization method is often observed to perform well in practice, it can fail to converge. This can be seen from simple counterexamples for low-order tensors as, for example, in \cite[Ex.~5.12]{Pfeffer:18}, but is also observed in larger-scale problems due to effects that have to do with the global structure of eigenvectors, as noted in \cite{Dolfietal:12}. Beyond these basic difficulties, as outlined in \cite{BGP:22}, the restriction to the $N$-particle subspace $\mathcal{F}_N^K$ may impact the convergence of eigensolvers. In particular, the choice of rank parameters in the respective starting values can have a crucial influence.

While DMRG is an instance of a method operating on separate components in a tensor representation, one can also construct methods that approximately perform steps of a solver (such as a Krylov subspace method) on vectors represented in a low-rank tensor format. When the error tolerances for rank truncation are sufficiently small, the convergence properties of the underlying schemes can be preserved. As in the case of DMRG, however, one faces the problem of controlling the resulting ranks of approximations. 
In methods that operate on manifolds of tensors of fixed ranks, such as Riemannian optimization methods, conversely one faces the problem that it is difficult to ensure convergence.

Our central aim in this work is the construction of an eigenvalue solver operating on matrix product states that interacts well with the block structures associated to fixed particle numbers, that is ensured to converge (at least for good enough initial guesses) and comes with bounds on the generated ranks of approximations.

The strategy that we follow here for balancing approximation errors and tensor ranks is based on the one developed for linear elliptic operator equations in \cite{BachmayrDahmen:15,BachmayrDahmen:16}, as presented with some adjustments in the review article \cite{B23}. 
While in these works, the target error measure is the total norm error with respect to the exact solution of a PDE, in the present work some major adaptations are required, since we aim for controlling angles between approximate and exact eigenspaces.
For this reason, we do not consider here another strategy for controlling ranks based on soft thresholding of tensors as in \cite{BachmayrSchneider:17}, which is more strongly tied to norm errors.

We achieve control of approximation ranks by an interplay of error reduction, by an iterative solver providing a guaranteed improvement of approximation errors, with an approximation complexity reduction, by tensor rank truncation up to carefully chosen tolerances. Such combinations of iterative schemes with rank truncation have been considered previously for eigenvalue problems, for example, in \cite{Kressner:11,Dolgov:14,KSU:14}, but without bounds on the tensor ranks in terms of approximation errors. 

As our basic iterative solver, we use \emph{preconditioned inverse iteration} (PINVIT). This method is well suited for implementation in low-rank format and for the combination with preconditioning in our particular setting, but in principle it is not the only possible choice. The crucial point, however, is that for suitable starting values, it ensures a reduction of the eigenspace approximation error.
In this regard, we draw on the classical results on PINVIT for matrix eigenvalue problems by Knyazev and Neymeyr \cite{KN:03,KN:09} and on their adaptation to problems on function spaces with approximate residual evaluation \cite{Rohwedder:11}. 

\subsection{Outline}

We recall basic notation for tensors in MPS format and the relevant results of \cite{BGP:22} in Section \ref{sec:prelim}.
In Section \ref{sec:overview}, we give a summary of our main results that are shown in the subsequent sections.
Section \ref{sec:precinsc} provides an analysis of the low-rank preconditioner that we subsequently employ in our method.
In Section \ref{sec:conv}, we give a convergence analysis of inexact preconditioned inverse iteration for one-dimensional eigenspaces, which may be of independent interest in its sharpening of results of \cite{Rohwedder:11}.
We describe a generalization to a subspace iteration for simultaneous approximation of several eigenspaces in Section \ref{sec:multiple}.
In Section \ref{sec:numer}, we give results of numerical experiments for model systems.

\section{Preliminaries}\label{sec:prelim}

In our notation, we follow \cite{kazeev_low-rank_2012,BK:20} with some adaptations. We briefly summarize the findings of \cite{BGP:22}, where a more detailed introduction and discussion can be found.

\subsection{Matrix product states and operators}\label{sec:MPS}

The \emph{matrix product state} (or \emph{tensor train}) representation of $\ten{x} \in \cF^K$ reads
\begin{equation} \label{eq:mps}
\ten{x}_\alpha = \ten{x}_{\alpha_1, \ldots, \alpha_K} =
		\sum_{j_1=1}^{r_1}
		\cdots
		\sum_{j_{K-1}=1}^{r_{K-1}}
		X_1(j_0, \alpha_1, j_1)
		X_2(j_1, \alpha_2, j_2)
		\, \cdots \,
		X_K(j_{K-1}, \alpha_K, j_K)
\end{equation}
for $\alpha \in \{0,1\}^K$, where for notational reasons we set $j_0 = j_K = 1$, $r_0 = r_K = 1$.

For the third-order component tensors $X_k$ in such a representation, called \emph{cores}, we write
\begin{equation}\label{eq:cores}
 \rep{X} = (X_1,\ldots, X_K),\quad    X_k = \bigl(  X_k(j_{k-1}, \alpha_k, j_k)  \bigr)_{\substack{j_{k-1}  = 1,\ldots, r_{k-1}, \\ \alpha_k = 0,1,  \\ j_k = 1,\ldots,r_k}} \,.
\end{equation}
For specifying the components of an MPS explicitly, we use the notation
\begin{equation*}\label{eq:corenotation1}
   X^{ [j_{k-1}, j_k] }_k = \bigl( X_k(j_{k-1}, \alpha_k, j_k) \bigr)_{\alpha_k=0,1},    
\end{equation*}
In terms of the vectors $X_k^{[j_{k-1}, j_k]} \in \R^{\{0,1\}}$, a core $X_k$ is then given by the rankwise block representation
\begin{equation}\label{eq:corenotation2}
   X_k = \begin{bmatrix} X_k^{[1,1]} & \cdots &  X_k^{[1,r_k]}  \\ \vdots & \ddots & \vdots \\ X_k^{[r_{k-1},1]} & \cdots & X_k^{[r_{k-1},r_k]}  \end{bmatrix}.
\end{equation}
Complementing this, we also introduce the matrices
\begin{equation*}
   X^{ \{ \alpha_k \} }_k  = \bigl( X_k(j_{k-1}, \alpha_k, j_k) \bigr)_{j_{k-1}=1,\ldots,r_{k-1}}^{j_k=1,\ldots,r_k} \, ,
\end{equation*}
here and in what follows, we write row indices as subscripts and column indices as superscripts.

For a compact way of writing \eqref{eq:mps} in terms of the cores \eqref{eq:cores}, we use the \emph{strong Kronecker product}, 
\[
   ( X_1 \SKP X_2 )^{ \{ \alpha_1 \alpha_2 \} } = X_1^{ \{ \alpha_1 \}  }  X_2^{ \{ \alpha_2 \}  }.
\]
For example, for two cores $X, Y$ of ranks $2\times 2$, we obtain
\[
\begin{multlined}
  	\begin{bmatrix}
	X^{[1,1]} & X^{[1,2]}\\
	X^{[2,1]} & X^{[2,2]}\\
	\end{bmatrix}
	\SKP
	\begin{bmatrix}
	Y^{[1,1]} & Y^{[1,2]}\\
	Y^{[2,1]} & Y^{[2,2]}\\
	\end{bmatrix}  \qquad \\
 \qquad	=
	\begin{bmatrix}
	X^{[1,1]} \otimes Y^{[1,1]} + X^{[1,2]} \otimes Y^{[2,1]} & X^{[1,1]} \otimes Y^{[1,2]} + X^{[1,2]} \otimes Y^{[2,2]}\\
	X^{[2,1]} \otimes Y^{[1,1]} + X^{[2,2]} \otimes Y^{[2,1]} & X^{[2,1]} \otimes Y^{[1,2]} + X^{[2,2]} \otimes Y^{[2,2]}\\
	\end{bmatrix}.
	\end{multlined}
\]
With this notation, we have
\[
     [\ten{x}] = X_1 \SKP X_2 \SKP \cdots \SKP X_K\,,
\]
where the block $[\ten{x}] \in \R^{ 1 \times \{0,1\}^K \times 1}$ has leading and trailing dimensions of mode size 1.
For simplicity, we ignore such singleton dimensions, that is,
we identify $[\ten{x}]$ with the tensor $\ten{x} \in \R^{\{0,1\}^K}$ of order $K$.

\subsection{Singular value decomposition}
For $k=1,\ldots,K-1$, using $(\alpha_1,\ldots,\alpha_k)$ as row index and $(\alpha_{k+1},\ldots, \alpha_K)$ as column index, one obtains the $k$-th \emph{matricization} (or \emph{unfolding}) of a tensor $\ten x \in \cF^K$. Any representation $\rep X$ of $\ten x$ can be transformed by operations on its cores such that these matricizations are in variants of an SVD form, which also yield the singular values $\sigma_{k,j}$ for $k = 1,\ldots,K-1, j = 1,\ldots,r_k$. The \emph{Vidal decomposition} \cite{vidal03} is an instance of such a representation. A very similar form results immediately from the \emph{tensor train SVD (TT-SVD)} algorithm \cite{Oseledets:2011:TT}, which in turn can be regarded as a special case of the \emph{hierarchical SVD} \cite{Grasedyck:2010:HierarchicalSVD} of more general tree tensor networks.

We write $\rank_i(\ten x)$ for the rank of the $i$-th matricization of $\ten x$ and
\[  \ranks(\ten x) = \bigl( \rank_1(\ten x), \ldots, \rank_{K-1} (\ten x) \bigr)  \in \N_0^{K-1} \]
for the vector of these ranks. These correspond to the entrywise minimal tupel $(r_1,\ldots,r_{K-1})$ of any representation $\rep X$ of $\ten x$ and they are therefore called {\em ranks} of the tensor $\ten x$.
The rank truncation of each matricization yields quasi-optimal approximations of lower ranks \cite{Oseledets:2011:TT,Grasedyck:2010:HierarchicalSVD}: let $\ten x$ be given with ranks $r_1,\ldots,r_{K-1}$, and denote by $\operatorname{trunc}_{s_1,\ldots,s_{K-1}}(\ten x)$ its truncation to ranks at most $s_k\leq r_k$, $k=1,\ldots,K-1$, then
\begin{equation}\label{eq:ttsvdquasiopt}
\begin{multlined}
  \norm{ \ten x - \operatorname{trunc}_{s_1,\ldots,s_{K-1}}(\ten x)}
   \leq \biggl(\sum_{k=1}^{K-1} \sum_{j=s_k+1}^{r_k} \sigma_{k,j}^2 \biggr)^{\frac12} \\
  \qquad\qquad\qquad \leq \sqrt{K-1} \min\bigl\{ \norm{\ten x - \ten y} \colon \ten y \text{ of ranks at most $s_1,\ldots, s_{K-1}$} \bigr\}.
   \end{multlined}
\end{equation}
Using the above error bound in terms of the matricization singular values, one obtains an approximation $\operatorname{trunc}_\varepsilon(\ten x)$ with $\norm{\ten x - \operatorname{trunc}_\varepsilon(\ten x)}_2 \leq \varepsilon$ for any $\varepsilon>0$ by truncating ranks according to the smallest singular values.

\subsection{Block structure of matrix product states}\label{sec:blockstructure}

We represent the tensors with fixed particle number $N \leq K$, $\ten x \in \cF^K_N$ in block-sparse MPS format, see~\cite{BGP:22} for more details. The matrices $X^{\{0\}}_k$ and $X^{\{1\}}_k$ are either block-diagonal or they have blocks only just above or just below the diagonal. We denote the blocks representing an unoccupied $k$-th orbital by 
\begin{equation*}
\unocc{X}_{k,n} \in \R^{\rho_{k-1,n} \times \rho_{k,n}} \qquad \text{for $n \in \cK_{k-1}\cap \cK_{k}$},
\end{equation*}
and those representing an occupied orbital by 
\begin{equation*}
\occ{X}_{k,n}  \in \R^{\rho_{k-1,n} \times \rho_{k,n+1}} \qquad \text{for $n \in \cK_{k-1} \cap (\cK_{k}-1)$},
\end{equation*}
where $\cK_k := \bigl\{ \max\{ 0 , N - K + k \} ,\ldots,  \min\{N, k\} \bigr\}$.
Note that for the block sizes, it holds $\sum_{n \in \cK_k} \rho_{k,n} = r_k$ and $\rho_{0,0} = \rho_{K,N} = 1$.

For $k$ such that $N < k < K-N+1$, which we refer to as the \emph{generic case}, we have $\cK_{k-1} = \cK_k = \{0,\ldots,N\}$; otherwise, the number of particles to the right and to the left of orbital $k$, and hence the elements of $\cK_{k-1}$ and $\cK_k$, are restricted. The corresponding block structure  has the form
\begin{equation}\label{eq:blockstructure}
\setlength\arraycolsep{3pt}
X^{\{0\}}_k = \begin{pmatrix}
\unocc{X}_{k,0} & 0 & \cdots & 0 \\
0 & \unocc{X}_{k,1} & \cdots & 0 \\
\vdots & \vdots & \ddots & \vdots \\
0 & 0 & \cdots & \unocc{X}_{k,N}
\end{pmatrix}	
\qquad \text{and} \qquad
X^{\{1\}}_k = \begin{pmatrix}
0 & \occ{X}_{k,0} & \cdots & 0 \\
\vdots & \vdots & \ddots & \vdots \\
0 & 0 & \cdots & \occ{X}_{k,N-1} \\
0 & 0 & \cdots & 0 
\end{pmatrix}.	
\end{equation}
Since nonzero blocks for the unoccupied orbital never occur in the same position as the ones for the occupied orbital, the two layers $\alpha = 0$ and $\alpha = 1$ can be summarized in the core representation
\begin{equation*}
\setlength\arraycolsep{3pt}
X_k = \begin{bmatrix}
\unocc{X}_{k,0}^{\uparrow} & \occ{X}_{k,0}^{\uparrow} & \cdots & 0 \\[-3pt] 
0 & \unocc{X}_{k,1}^{\uparrow} & \ddots & \vdots \\[-2pt]
\vdots & \vdots & \ddots & \hspace{4pt} \occ{X}_{k,N-1}^{\uparrow} \\[3pt]
0 & 0 & \cdots & \hspace{-7pt}\unocc{X}_{k,N}^{\uparrow}
\end{bmatrix} =: \bidiag\left(\begin{bmatrix} \unocc{X}_{k,n}^{\uparrow} & \occ{X}_{k,n}^{\uparrow} \end{bmatrix}_{n \in \cK_{k-1}} \right),
\end{equation*}
where each block is composed of vectors, and where we define the lift
\begin{equation*}
\unocc{X}_{k,n}^{\uparrow} = \unocc{X}_{k,n} \uparrow e^0 \in \R^{\rho_{k-1,n} \times 2 \times \rho_{k,n}} \qquad \text{for $n \in \cK_{k-1}\cap \cK_{k}$}
\end{equation*}
and
\begin{equation*}
\occ{X}_{k,n}^{\uparrow} = \occ{X}_{k,n} \uparrow e^1 \in \R^{\rho_{k-1,n} \times 2 \times \rho_{k,n+1}} \qquad \text{for $n \in \cK_{k-1} \cap (\cK_{k}-1)$}.
\end{equation*}
Here, the \emph{lift product} $\uparrow$ is to be understood as a Kronecker product with reordered indices.
The cases where either $k \leq N$ or $k \geq K-N+1$ have the last rows or first columns (and zero rows and columns) removed, respectively, as illustrated in the following example.
\begin{example}
A tensor $\ten x \in \cF^5_2$ of order $K=5$ and particle number $N=2$, representing $5$ orbitals and $2$ particles, has the form\small
\begin{equation*}
\setlength\arraycolsep{3pt}
\ten x = \begin{bmatrix}
\unocc{X}_{1,0}^{\uparrow} & \occ{X}_{1,0}^{\uparrow}
\end{bmatrix}
\SKP
\begin{bmatrix}
\unocc{X}_{2,0}^{\uparrow} & \occ{X}_{2,0}^{\uparrow} & 0 \\[3pt]
0 & \unocc{X}_{2,1}^{\uparrow} & \occ{X}_{2,1}^{\uparrow}
\end{bmatrix}
\SKP
\begin{bmatrix}
\unocc{X}_{3,0}^{\uparrow} & \occ{X}_{3,0}^{\uparrow} & 0 \\[3pt]
0 & \unocc{X}_{3,1}^{\uparrow} & \occ{X}_{3,1}^{\uparrow} \\[3pt]
0 & 0 & \unocc{X}_{3,2}^{\uparrow} \\
\end{bmatrix}
\SKP
\begin{bmatrix}
\occ{X}_{4,0}^{\uparrow} & 0 \\[3pt]
\unocc{X}_{4,1}^{\uparrow} & \occ{X}_{4,1}^{\uparrow} \\[3pt]
0 & \unocc{X}_{4,2}^{\uparrow}
\end{bmatrix}
\SKP
\begin{bmatrix}
\occ{X}_{5,1}^{\uparrow} \\[3pt]
\unocc{X}_{5,2}^{\uparrow} 
\end{bmatrix}.
\end{equation*}\normalsize
\end{example}

As an eigenspace of a linear operator, $\cF^K_N$ is a linear subspace of $\cF^K$. Addition and scalar multiplication of MPS in this subspace correspondingly  work equivalently to those of regular MPS: Addition of two tensors in block-sparse MPS format is the concatenation of corresponding blocks, scalar multiplication is the multiplication of all blocks in one of its components. Furthermore, left- and right-orthogonalization as well as rank truncation procedures can be performed explicitly in this format, cf.~\cite[Section 4]{BGP:22}.

\subsection{Operator representations}

All operators used in this article are of the form~\eqref{eq:hamil}: They are the scaled sum of {\em one-particle interactions} $\ten{a}_{i}^\ast \ten{a}_j$, described by the one-particle operator $\ten T$, and {\em two-particle interactions} $\ten{a}_{i}^\ast \ten{a}_{j}^\ast \ten{a}_k \ten{a}_\ell$, stored in the two-particle operator $\ten V$. As such, they are also particle-number preserving and thus map $\cF^K_N$ to itself. 

These operators can be cast in a \emph{matrix product operator} (MPO) representation that is compatible with MPS representations of vectors. 
The corresponding MPO rank is an upper bound for the factor by which the action of the operator can increase MPS ranks.
The following rank bounds for $\ten T$ and $\ten V$ are shown in~\cite[Theorems~5.4,~5.5]{BGP:22} (see also \cite{dolgov2013two,kazeev2013low} and \cite{crosswhite2008finite,keller2015efficient}).

\begin{theorem}\label{thm:opranks}
For $\ten T$, $\ten V$ as in \eqref{eq:hamil}, the following rank bounds hold true:
\begin{enumerate}[{\rm\bf(i)}]
\item $\ten{T}$ has an MPO representation of ranks bounded by $K+2$, and $\ten V$ has an MPO representation of ranks bounded by $\frac12 {K^2}+ \frac32 {K}+2$.
\item If for some $d \in \N_0$, we have $ t_{ij} = 0$ whenever $\abs{i-j} > d$, then the MPO ranks of $\ten T$ are bounded by $2d+2$,
and if there exists $\tilde d\in\N_0$ such that $v_{i j k l} = 0$ whenever  \begin{equation*}
 \max\{ \abs{i-j},\abs{i-k},\abs{i-l},\abs{j-k},\abs{j-l},\abs{k-l}  \} > \tilde d,\end{equation*} 
 then the MPO ranks of $\ten{V}$ are bounded by $\tilde d^2+3\tilde d-1$ if $\tilde d$ is odd and by $\tilde d^2+3\tilde d-2$ if $\tilde d$ is even.
\end{enumerate}
\end{theorem}

In \cite{BGP:22}, it was shown that the action of such operators on an MPS in block-sparse representation can be performed by a \emph{matrix-free} scheme: Applying a single term in the sum~\eqref{eq:hamil} consists only of reordering and possibly rescaling the blocks in the MPS. The resulting tensors then have to be summed up by concatenating the blocks. This preserves the block structure but increases the ranks. We follow this paradigm and think of all operators that appear in this article as matrix-free reorderings, scalings, and concatenations of blocks, so that in particular, no additional memory is required for MPO representations.

\section{Overview of the Method and Main Results}\label{sec:overview}

In order to approximate the lowest eigenvalue of the Hamiltonian $\ten H$, we consider an inexact preconditioned inverse iteration that is compatible with the block-sparse MPS format described in \cref{sec:blockstructure}. We prove convergence of the proposed algorithm and obtain bounds on the ranks of the iterates and thus on the overall computational complexity. 
Although we consider the number $K$ of orbitals (and thus the order of the tensors that we operate on) to be fixed, we are in particular interested in how the behaviour of the method depends on $K$.

\subsection{Construction of the iterative method}\label{sec:construction}

In what follows, we assume $\gamma>0$ to be a shift parameter such that 
\[ \ten H_\gamma = \ten H + \gamma \ten I \] is symmetric positive definite. 
While the choice of $\gamma$ is problem-dependent, in the case of \eqref{eq:hamilt} it can be chosen to be independent of $K$, as considered in more detail in \cref{sec:shifts}.
In what follows, we take the positive eigenvalues of $\ten H_\gamma$ to be ordered as 
\begin{equation}\label{eq:eigenvalueordering}
0 < \lambda_1 \leq \lambda_2 \leq \ldots\leq \lambda_n, 
\end{equation}
 where typically the lowest eigenvalue $\lambda_1$ is of primary interest.

For a suitable self-adjoint, positive definite preconditioner $\ten S: \mathcal{F}_N^K \rightarrow \mathcal{F}_N^K$, we then replace the eigenvalue problem \cref{eq:Hevp} by the
equivalent symmetrically preconditioned generalized eigenvalue problem
\begin{equation}
\label{eq:Hevpprecond}
 \langle \ten S \ten H_\gamma \ten S \ten x, \ten y \rangle = \lambda \langle \ten S^2 \ten x, \ten y \rangle \quad\text{for all $\ten y \in \mathcal{F}_N^K$.}
\end{equation}
With the self-adjoint matrices
\[
   \A = \ten S \ten H_\gamma \ten S, \quad \E = \ten S^2, 
\]
the sought eigenvalue $\lambda_1$ is the minimum of the Rayleigh quotient 
\begin{equation}\label{eq:rayleigh1}
  \lambda (\ten x) := \frac{\langle \A \ten x, \ten x\rangle }{\langle \E \ten x ,\ten x \rangle }
\end{equation}
for \cref{eq:Hevpprecond}, that is,
\[
    \lambda_1  
   = \min_{\ten x \in \mathcal{F}_N^K} \frac{\langle \ten{H}_\gamma \ten x , \ten x\rangle }{\langle \ten x,\ten x \rangle}= \min_{\ten x \in \mathcal{F}_N^K} \lambda (\ten x) .
\]

The low-rank eigenvalue solvers that we consider in what follows are based on preconditioned inverse iteration, which with the above symmetric preconditioning and with an initial guess $\ten x^0$ takes the basic form
\begin{equation}\label{eq:basicpinvit}
   \ten x^{n+1}  = \ten x^n - \omega  \bigl(   \A \ten x^n - \lambda(\ten x) \E \ten x^n  \bigr) 
\end{equation} 
with a step size parameter $\omega> 0$ such that 
\begin{equation}\label{eq:precondassumption}
 \norm{ \ten I - \omega  \A  }_{\A} =  \norm{ \ten I - \omega  \A  } \leq \eiter < 1
\end{equation}
with some $\eiter  \in (0,1)$; here and in what follows, we write $\norm{\ten x}_{\ten A} = \sqrt{\langle \ten A \ten x, \ten x\rangle}$ for $\ten x \in \mathcal F^K$ and use the same notation for the induced operator norm.

Our further analysis is based on assumptions that are appropriate for representations of iterates and approximate eigenvectors in block-sparse MPS format and for corresponding representations of matrices. In this setting, achieving \eqref{eq:precondassumption} with $\eiter$ independent of $K$ leads to certain requirements on $\ten S$; in particular, the MPO ranks of $\ten S$ generally need to be larger than one, so that its action increases the ranks of MPS representations. The preconditioners that we construct in \cref{sec:precinsc} are of the form
\[ \ten S = \sum_{k=M^-}^{M^+} \ten S_k ,  \]
where each summand $\ten S_k: \mathcal{F}_N^K \rightarrow \mathcal{F}_N^K$ for $k = M^-,\ldots,M^+$ is a positive definite, diagonal matrix of MPO rank one (that is, a $K$-fold Kronecker product of matrices in $\R^{2\times 2}$).

Due to the combined MPO ranks of $\ten S$ and $\ten H_\gamma$, in our analysis we allow for inexact application of $\E$ and $\A$ with tensor rank truncations up to prescribed error tolerances in \eqref{eq:basicpinvit}.
However, the Rayleigh quotient $\lambda$ in \cref{eq:rayleigh1}, which takes the form
\begin{align}\label{exactrq}
 		\lambda(\ten x) = \frac{\displaystyle\sum_{k,k' = M^-}^{M^+} \langle \ten{H}_\gamma  \ten S_k \ten x, \ten S_{k'} \ten x\rangle}{\displaystyle\sum_{k,k' = M^-}^{M^+} \langle \ten S_k \ten x, \ten S_{k'} \ten x \rangle},
\end{align}
can be evaluated term by term without additional rank truncations. Due to this exact evaluation of $\lambda$ we can also take advantage of the faster convergence of Rayleigh quotients.

\subsection{Main results on convergence, rank estimates and computational complexity}\label{sec:eocc}

Our aim is to approximate the eigenpair $(\lambda_1, \ten u_1)$ of the preconditioned problem \eqref{eq:Hevpprecond} with $\norm{\ten u_1}_{\ten A} = 1$, under the assumption $\lambda_1 < \lambda_2$. This amounts to finding $\ten y \in \cF^K_N$ in low-rank representation such that  $\sin \angle_{\ten A} (\ten u_1 , \ten y)  \leq \norm{ \ten u_1 - \ten y }_{\ten A} \leq \varepsilon$ for each desired error tolerance $\varepsilon$. However, the ranks of such $\ten y$ should ideally not be larger than necessary for achieving an approximation of this quality. A natural point of reference is thus the least value of the maximal MPS rank $\norm{\ranks(\ten y)}_\infty$ that is required for $\norm{ \ten u_1 -\ten y}_{\ten A}\leq \varepsilon$, in other words,
\begin{equation}\label{eq:defrbest}
  r_\mathrm{best}(\ten u_1, \varepsilon) = \min \bigl\{ r \in \N_0 \colon \exists \ten y \in \cF^K_N , \norm{\ranks(\ten y)}_\infty \leq r \colon  \norm{\ten u_1 - \ten y }_{\ten A} \leq \varepsilon   \bigr\}\,.
\end{equation}
Our method has the following basic structure, where all operations on approximate eigenvectors are performed in MPS low-rank format: with a starting value $\ten y^0$, fixed $c_K \in (0,1)$ and $C_K>0$ sufficiently large, an initial error bound $\etaout^0>0$, and $\etaout^m = 2^{-m} \etaout^0$, 
\begin{tabbing}
 \hspace{9pt}\=\hspace{18pt}\=\hspace{24pt}\=\hspace{18pt}\=\hspace{18pt}\=\kill 
 \> for $m = 0, 1, 2,\ldots$, \\[3pt]
 \> \> (I1)\> with $\ten x^0 = \ten y^m$, determine $\ten x^{N_m}$ with sufficiently large $N_m$ by applying \cref{eq:basicpinvit}\\
 \> \>  \> with inexact residual evaluation and rank truncation, that is,   \\[6pt]
 \> \> \> \> $\ten{\tilde x}^{n+1} = \operatorname{trunc}_{\etain^n}(\ten{x}^n - \omega_n  \operatorname{res}_{\etares^n}(\ten{x}^n)) $ \\
 \> \> \> \> \>with suitably chosen $\etain^n,\etares^n>0$ and\\
 \> \> \> \> \> $\operatorname{res}_{\etares^n}(\ten{x}^n)$ such that $\norm{\operatorname{res}_{\etares^n}(\ten{x}^n) - (\ten A \ten x^n - \lambda(\ten x^n) \ten E \ten x^n)}_{\ten A} \leq \etares^n$,  \\[3pt]
 \> \> \> \> $\ten x^{n+1} = \ten{\tilde x}^{n+1}  / \norm{\ten{\tilde x}^{n+1} }_{\ten A}$, \\[3pt]
 \>\>\> such that $\norm{ \ten u_1 - \ten x^{N_m}} \leq c_K \etaout^m$.\\[6pt]
 \> \>  (I2)\> determine $\ten y^{m+1}$ by rank truncation of $\ten x^{N_m}$,\\[3pt]
 \>\>\>\> $\ten{\tilde y}^{m+1} = \operatorname{trunc}_{C_K \etaout^m} ( \ten x^{N_m} ),\qquad$ $\ten y^{m+1} = \ten{\tilde y}^{m+1} / \norm{ \ten{\tilde y}^{m+1} }_{\ten A}.$
\end{tabbing}
Whereas the inner iteration in step (I1) yields an error reduction, step (I2) performs a complexity reduction that, as shown below, ensures near-optimal ranks of $\ten y^m$, $m \in \N$.
The complete scheme with the concrete choices of algorithmic parameters is given in \Cref{alg:outerpinvit} in \cref{sec:conv}.

\begin{theorem}\label{thm:main}
For the results of steps {\rm (I1)} and {\rm (I2)} above, with $\gamma$ and $\eiter$ in \eqref{eq:precondassumption} independent of $K$ and  parameters chosen as in \Cref{alg:outerpinvit}, 
we have
\[
  \norm{ \ten u_1 - \ten y^m }_{\ten A} \leq \etaout^m,
\]
while the number of inner iterations performed for each $m$ is bounded, uniformly in $m$, by $C \log K$ with $C$ independent of $K$, and
\[
   \norm{\ranks(\ten y^m)}_\infty \leq r_\mathrm{best} \left(\ten u_1, \frac{c_1 \etaout^m }{1 + c_2 \sqrt{K-1}}  \right)
\]
with $c_1, c_2>0$ independent of $K$ and $m$. 
\end{theorem}

The proof is given in \Cref{sec:outeriter}.
We thus establish bounds for the ranks of iterates in the outer iterations indexed by $m$ in terms of best approximation ranks.
However, ranks may in general grow beyond these values in inner iterations, and with our present approach we can only give bounds on the intermediate ranks that can arise in the inner iterations that also depend on the ranks of the involved operators. 
While for this reason, we do not arrive at final assertions on the total computational complexity of methods, but we summarize some first results in this direction in \cref{rem:intermediateranks}.

\section{Preconditioning in Second Quantization}\label{sec:precinsc}

\subsection{Spectral shifts} \label{sec:shifts}
As discussed in \Cref{sec:construction}, the method that we study here requires a shift $\gamma > 0$ such that $\ten H_\gamma = \ten H + \gamma \ten I$ is positive definite, as well as a suitable preconditioner for $\ten H_\gamma$. The choice of such a shift is problem-dependent and can thus be expected to depend in particular on the number of particles $N$. However, it is desirable to avoid a dependence of $\gamma$ on the number of basis orbitals $K$, which plays the role of a discretization parameter.

For the Hamiltonian $H$ as in \eqref{eq:hamilt}, we obtain shifts with this latter property of $K$-independence by finding $\gamma$ such that the shifted Hamiltonian $H + \gamma I$ becomes elliptic on the appropriate energy space. In this case, under appropriate assumptions on the external potential $V$, there  exists $\gamma>0$ such that
\begin{equation}\label{eq:Hellipt}
  \langle H v, v \rangle + \gamma \norm{v}_{L_2}^2 \geq \textstyle\frac14\displaystyle \norm{v}_{H^1}^2, \quad \forall v \in H^1\,.
\end{equation}
If $V$ in \eqref{eq:hamilt} is the molecular potential of the standard electronic Schr\"odinger equation,
\[
   V(x) =  - \sum_{\nu} \frac{Z_\nu}{\norm{x - R_\nu}_2} , \quad x \in \R^3,
\]
then the minimal $\gamma$ such that \eqref{eq:Hellipt} holds depends only (algebraically) on $N$ and on $\sum_{\nu} Z_\nu$; see \cite[eq.~(3.16)]{Y:10}.

With $\gamma$ as in \eqref{eq:Hellipt}, and with the second-quantized representation $\ten L$ of the operator $-\Delta + I$, which reads 
\begin{equation}\label{eq:Ldef}
   \ten L =  \sum_{i,j = 1}^K \ell_{ij}\, \ten a_i^\ast \ten a_j, \quad \ell_{ij} = \int_{\R^3} \nabla \phi_i \cdot \nabla \phi_j + \phi_i \phi_j \sdd x\,,
\end{equation}
for each $N$ and $K$ we obtain
\begin{equation}
   \textstyle \frac14 \displaystyle \langle  \ten L  \ten  x , \ten x\rangle  \leq  \langle \ten H_\gamma \ten x , \ten x \rangle \leq \Gamma \langle \ten L \ten x,\ten x\rangle, \quad
    \forall \ten x \in \cF_N^K, 
\end{equation}
where $\gamma, \Gamma > 0$ depend only $N$ and $Z$, but neither on $K$ nor on the orbitals. This implies in particular
\begin{equation}\label{eq:Lcond}
  \operatorname{cond} \bigl( \ten L^{-1/2}  \ten H_\gamma \ten L^{-1/2}  \bigr) \leq 4 \Gamma.
\end{equation}
Note that $\ten x$ is a solution of the original eigenvalue problem \cref{eq:Hevp} for $\ten H$ if and only if $\ten z = \ten L^{1/2} \ten x$ solves
\[
   \langle  \ten L^{-1/2}  \ten H_\gamma \ten L^{-1/2}  \ten z, \ten y \rangle = \lambda \langle \ten L^{-2} \ten z, \ten y \rangle \quad\text{for all $\ten y \in \mathcal{F}_N^K$,}
\]
where we use that $\ten L$ and its powers preserve the particle number, that is, leave $\mathcal{F}_N^K$ invariant. 

While this two-sided preconditioning thus leads to a formulation that is suitable for preconditioned inverse iteration, the action of $ \ten L^{-1/2}$ is difficult to realize numerically.
We thus instead use a substitute for this mapping that is compatible with the low-rank representations used for elements of $\mathcal{F}_N^K$.

\subsection{Exponential sums}\label{sec:expsum}

We make use of exponential sum approximations of $t^{-\lambda}$, $\lambda > 0$, that are based on sinc quadrature with step size parameter $h > 0$. By \cite[Lemma 5.3]{Scholz_2017}, we have
\begin{equation}
 \left| \frac{1}{t^\lambda} - \frac{h}{\Gamma(\lambda)} \sum_{m = -\infty}^\infty e^{\lambda m h} \exp(-e^{m h} t) \right| \leq \frac{C(\lambda,h)}{t^\lambda} \quad \text{for all $t > 0$}
\end{equation}
with a $C(\lambda, h)>0$ independent of $t$. 
For $\lambda = \frac12$, with an appropriate truncation of the summation, this yields an approximation of $t\mapsto t^{-1/2}$ on any desired subinterval of $(0, \infty)$.

\begin{theorem}\label{thm:expsum}
Let $\ten W \in \R^{n \times n}$ be symmetric with spectrum $\sigma(\ten W) \subset [1,T]$. Let $c_0 \in (0,1)$ and
\begin{equation}\label{eq:expsum}
  S(t) = \frac{h}{\sqrt{\pi}} \sum_{m = M^-}^{M^+} e^{m h / 2} \exp(-e^{m h} t),
\end{equation}
where
\begin{equation}\label{eq:expsumparams}
  h \!=\! \frac{\pi^2}{\ln(\frac{8\sqrt{2}}{c_0} + 1)}, \;\; M^+ = \left\lceil h^{-1} \ln \biggabs{ \ln \frac{\sqrt{\pi} c_0}{4} } \right\rceil, \;\; M^- = -\left\lceil h^{-1}\left( 2 \biggabs{ \ln \frac{\sqrt{\pi} c_0 }4 } + \ln T \right) \right\rceil.
\end{equation}
Then 
\begin{align}\label{Shatdeltabound}
  (1-c_0)^2 \langle   \ten x,  \ten x \rangle
   \leq \langle   S(\ten W)  \ten W  S(\ten W) \ten x , \ten x \rangle 	\leq 
   (1 + c_0)^2 \langle  \ten x,  \ten x \rangle, \quad \forall \ten x \in \R^n.
\end{align}
\end{theorem}

\begin{proof}
We combine the truncation estimates in \cite[Corollary 4.3]{B23} with the upper bound $C(\frac12, h) \leq  \frac{2\sqrt{2} }{e^{\pi^2 / h}-1}$ shown in \cite{Scholz_2017}.
For \eqref{eq:expsum}, the choices \eqref{eq:expsumparams}
ensure in particular that $C(\frac{1}{2},h) \leq c_0 / 4$ and thus the relative error bound
\begin{equation}\label{Srelerror}
  \left| \frac{1}{\sqrt{t}} - S(t) \right| \leq \frac{c_0}{\sqrt{t}}, \quad \forall t \in [1,T].
\end{equation}
This immediately implies \eqref{Shatdeltabound}.
\end{proof}

Note that \eqref{Shatdeltabound} implies 
\[  
    \operatorname{cond}\bigl(S(\ten W) \ten W S(\ten W) \bigr) \leq \frac{(1+c_0)^2}{(1-c_0)^2}.
\]   
By appropriate rescaling, as shown in \Cref{cor:lrprecond} below, the above applies to any matrix $\ten W$ with $\mathrm{cond}(\ten W) \leq T$.

\subsection{Expression for second quantization}\label{sec:expfsq}

For constructing a preconditioner of $\ten H_\gamma$, we use an auxiliary matrix $ \ten D$ of diagonal one-particle interactions,
\[ \ten D = \sum_{i = 1}^K d_{i}\, \ten a_i^\ast \ten a_i,\]
where we assume $\gamma$ and $\ten D$ to be chosen such that $\ten D^{-1/2} \ten H_\gamma \ten D^{-1/2}$ is well-conditioned, with a condition number that is in particular independent of $K$. Provided that we have an estimate of the form \eqref{eq:Lcond} for the given problem, this can be achieved by choosing $\ten D$ to be spectrally equivalent to the representation of the Laplacian $\ten L$ as in \eqref{eq:Ldef}. That this can be achieved uniformly in $K$ amounts to requiring that for the underlying $L^2$-orthonormal orbitals $\{ \phi_i \}$, the rescaling $\{ \phi_i  / \norm{\phi_i}_{H^1} \}$ yields a Riesz basis of $H^1$ with uniform constants.
We thus assume that with $C_{\mathrm{lower}} , C_{\mathrm{upper}} > 0$, 
\begin{equation}\label{eq:lowerupper}
  C_{\mathrm{lower}} \langle   \ten x,  \ten x \rangle
   \leq \langle \ten D^{-1/2} \ten H_\gamma \ten D^{-1/2}  \ten x , \ten x \rangle 	\leq 
   C_{\mathrm{upper}} \langle  \ten x,  \ten x \rangle, \quad \forall \ten x \in \cF^{K}_N.
\end{equation}

We now construct low-rank approximations of $\ten D^{-1/2}$ based on exponential sums.

\begin{cor}\label{cor:lrprecond}
Let $t_{\min}, t_{\max}$ be chosen such that $[  t_{\min},  t_{\max} ] \supset [\min(\sigma(\ten D)) ,  \max(\sigma(\ten D))]$, and with $c_0 \in (0,1)$ and $S$ as in \Cref{thm:expsum}, let
\begin{equation}\label{eq:Sdef}
  \ten S =  \alpha_0 S \bigl( t_{\min}^{-1} \ten D \bigr) , \quad \alpha_0 = \sqrt{\frac{2}{ t_{\min} (C_{\mathrm{lower}} ( 1 - c_0)^2 + C_{\mathrm{upper}}( 1 + c_0)^2)}}\,.
\end{equation}
Then for $\A = \ten S \ten H_\gamma \ten S$, we have
\begin{align}\label{specequivA}
  (1-c) \langle   \ten x,  \ten x \rangle
   \leq \langle \A  \ten x , \ten x \rangle 	\leq 
   (1 + c) \langle  \ten x,  \ten x \rangle, \quad \forall \ten x \in \mathcal{F}_N^K ,
\end{align}
where 
\begin{equation}\label{eq:precondcdef}
   c = \frac{C_{\mathrm{upper}}(1+c_0)^2 - C_{\mathrm{lower}}(1-c_0)^2}{C_{\mathrm{upper}}(1+c_0)^2 + C_{\mathrm{lower}}(1-c_0)^2}  \in (0,1) \,,
\end{equation}
and as a particular consequence,
\begin{equation}\label{rhockappa}
\norm{ \ten I - \A }_2 \leq c.
\end{equation}
\end{cor}

\begin{proof}
In a first step, we adapt the approximation on the normalized interval $[1,T]$ provided by \Cref{thm:expsum} by setting
\[
 \ten S_0 = t_{\min}^{-1/2}  S \bigl( t_{\min}^{-1} \ten D \bigr)  \,.
\]
By \Cref{thm:expsum}, with $c_0$ as in \eqref{eq:lowerupper}, for all $\ten x$ we have
\[   (1-c_0)^2 \langle   \ten x,  \ten x \rangle
   \leq \langle   \ten S_0  \ten D  \ten S_0 \ten x , \ten x \rangle 	
   \leq    (1 + c_0)^2 \langle  \ten x,  \ten x \rangle,
\]
which combined with \eqref{eq:lowerupper} yields
\begin{align*}
  (1-c_0)^2 C_{\mathrm{lower}} \langle   \ten x,  \ten x \rangle
   \leq \langle \ten S_0 \ten H_\gamma \ten S_0  \ten x , \ten x \rangle 	\leq 
   (1+c_0)^2 C_{\mathrm{upper}} \langle  \ten x,  \ten x \rangle, \quad \forall \ten x \in \cF^K_N,
\end{align*}
or in other words,
\[
  \sigma(\ten S_0 \ten H_\gamma  \ten S_0)\subseteq [  C_{\mathrm{lower}}(1-c_0)^2,   C_{\mathrm{upper}}(1+c_0)^2 ] \,.
\] 
We now use that the $\omega > 0$ such that $\norm{ \ten I - \omega \ten S_0 \ten H_\gamma \ten S_0 }_2$ is minimal subject to these spectral bounds is given by 
$ \omega = 2 / (C_{\mathrm{lower}}(1-c_0)^2 +  C_{\mathrm{upper}}(1+c_0)^2) $ to arrive at the scaling factor in \eqref{eq:Sdef} and the bounds \eqref{specequivA} and \eqref{rhockappa}.
\end{proof}

Since the summands of ${\ten D}$ commute, recalling the definition of $\ten a_i$ for $i=1,\ldots,K$ in terms of $A$ as in \eqref{eq:elementarycomponents}, we can simplify $\ten S$ to
\begin{align*}
\ten S & =
 \sum_{m = M^-}^{M^+} \alpha_m \exp(-\beta_m (d_{1} A^\trp A \otimes I \otimes \cdots \otimes  I + \ldots +  I \otimes \cdots \otimes  I \otimes d_K A^\trp A))   \\
& =  \sum_{m = M^-}^{M^+} \alpha_m  \begin{pmatrix}
1 & 0 \\ 0 & e^{-\beta_m   d_1}
\end{pmatrix}  \otimes \begin{pmatrix}
1 & 0 \\ 0 & e^{-\beta_m   d_2}
\end{pmatrix}  \otimes \cdots \otimes  \begin{pmatrix}
1 & 0 \\ 0 & e^{-\beta_m   d_K}
\end{pmatrix} 
\end{align*}
with $\alpha_m = \frac{h \alpha_0}{\sqrt{ \pi}} e^{m h / 2}$, $\beta_m = \frac{1}{\sqrt{t_{\min}}} e^{m h}$.
Note that we can also rewrite this as
\[ 
    \ten S = \sum_{m = M^-}^{M^+} \prod_{i = 1 }^K ( \alpha_m^{1/K } \ten a_i \ten a_i^\ast + \alpha_k^{1/K} e^{-\beta_m d_i} \ten a_i^\ast \ten a_i ) = \sum_{m = M^-}^{M^+} 
\bigotimes_{i = 1}^K \begin{pmatrix} \alpha_m^{1/K} & 0 \\ 0 & \alpha_m^{1/K} e^{-\beta_m d_i} \end{pmatrix}. 
\]
In particular, the action of $\ten S$ increases the ranks of an MPS representation by at most a factor of $M^+ - M^- + 1$, with $M^+$ and $M^-$ as in \eqref{eq:expsumparams}.

\section{Convergence of Inexact Preconditioned Inverse Iteration}\label{sec:conv}

In this section, we consider the case of a ground state eigenvalue $\lambda_1$ of multiplicity one, so that $\lambda_1 < \lambda_2$. We return to the case of potentially higher multiplicities and joint approximations of several eigenspaces in \Cref{sec:multiple}.

For the mapping applied in every step of the iterative scheme \eqref{eq:basicpinvit}, we introduce the notation
\begin{equation}\label{eq:unpertiter}
   \Psi (\ten x, \omega) = \ten x - \omega \ten r(\ten x), \qquad \ten r(\ten x) = \A \ten x - \lambda(\ten x) \E \ten x.	
\end{equation}
We assume the step size parameter $\omega>0$ to be chosen such that $ \| \ten I - \omega  \A \|_{\A}  \leq \eiter < 1$ with $\eiter \in (0,1)$.
For $\ten A$ as in \Cref{cor:lrprecond}, for the particular choice $\omega_{\mathrm c} = 1$, by \eqref{rhockappa} we have $\eiter = c$ with $c$ as in \eqref{eq:precondcdef}.

\subsection{Relation to forward iteration}\label{corfor}
For the convergence analysis, following \cite{KN:09,Rohwedder:11}, we rewrite the iterative scheme defined by \eqref{eq:unpertiter} as an equivalent forward iteration 
for the maximization of the Rayleigh quotient
\begin{align*}
 \mu(\ten x) = \frac{\langle  \ten B\ten x, \ten x \rangle_{\ten  A}}{\langle \ten x,\ten x \rangle_{\ten A}} = \lambda(\ten x)^{-1}
\end{align*}
with the notation
\begin{align*}
  \ten B = \ten A^{-1} \E, \quad \ten T_\omega = \omega \ten A.
\end{align*}
Note that $\ten B$ is self-adjoint with respect to the $\ten A$-inner product. The iteration \eqref{eq:unpertiter} can be rewritten, with $\ten r_{\ten B}(\ten x) = \ten B \ten x - \mu(\ten x) \ten x$, as
\begin{align*}
 \Psi (\ten x, \omega) = \ten x + \frac{1}{\mu(\ten x) } T_\omega \ten r_{\ten B}(\ten x),
\end{align*}
corresponding to a maximization of $\mu$. Furthermore, we introduce
\begin{align*}
 \rho(\ten x)
 = \frac{1}{\mu(\ten x) } \frac{\| \ten r_{\ten B}(\ten x) \|_{\ten A}}{\|\ten x\|_{\ten A}} = \frac{\|\ten r(\ten x)\|_{\ten A^{-1}}}{\|\ten x\|_{\ten A}}\,.
\end{align*}
A sufficient condition for convergence then reads $\| \ten I - \ten T_\omega \|_{\ten A} \leq \eiter < 1$ for the unperturbed forward iteration.

\subsection{Perturbed iteration}

As the perturbed iteration is performed on the block-sparse tensor format (see \Cref{sec:blockstructure}), rank truncations are required to maintain numerical viability. For a perturbed residual $\operatorname{res}_{\etares}(\ten x)$, we consider the fixed point map
\[
 \tilde \Psi (\ten x, \omega, \etain, \etares) = \operatorname{trunc}_{\etain} \big( \ten x - \omega  \operatorname{res}_{\etares}(\ten x) \big).
\]
Corresponding to one iteration of \cref{alg:scheme}, it defines the perturbed inverse iteration as approximation to $\Psi$ under the use of certain truncations (see \Cref{truncdetail}) that satisfy the bounds
\begin{align*}
 \norm{\operatorname{trunc}_{\eta}(\ten x) - \ten x } \leq \eta \quad \text{and} \quad \norm{  \operatorname{res}_{\eta}(\ten x) - \ten r(\ten x) }  \leq \eta,
\end{align*}
for all $\eta > 0$ and $\ten x \in \mathcal{F}_N^K$. The details on the implementation of \cref{alg:scheme} including the choice of all parameters are given in \cref{alg:innerpinvit} below, based on the following analysis of the scheme. %
\begin{algorithm}[t!]
\DontPrintSemicolon
\SetKwInOut{Input}{input}\SetKwInOut{Output}{output}
\Input{tensor $\ten{x}^0$, constants $\teig, \enum$}
\Output{approximation to generalized eigenpair $(\lambda_1,\ten{u}_1)$ of $(\ten A, \ten E)$}
\BlankLine
\For{$n = 0,1,\ldots$}{
    set $\etares > 0$ such that $\norm{\operatorname{res}_{\etares}(\ten{x}^n) - \ten r(\ten{x}^n)}_{\ten A} \leq \enum \| \ten r(\ten{x}^n) \|_{\A^{-1}}$\tcp*[r]{\normalfont{} \Cref{thm:conv}(i)}
    set $\omega_\ast$ as minimizer of 
    $ \lambda_\ast := \min_{\omega \in \R}\ \lambda (\ten{x}^n - \omega  \operatorname{res}_{\etares}(\ten{x}^n))$\tcp*[r]{\normalfont{} \Cref{omegamin}}
    set $\etain > 0$ such that $\lambda (\operatorname{trunc}_{\etain}(\ten{x}^n - \omega_\ast  \operatorname{res}_{\etares}(\ten{x}^n))) \leq \lambda_\ast + \teig \, (\lambda(\ten{x}^n) - \lambda_\ast)$\;  
    $\ten{x}^{n+1} \gets \operatorname{trunc}_{\etain}(\ten{x}^n - \omega_\ast  \operatorname{res}_{\etares}(\ten{x}^n))$\tcp*[r]{\normalfont{} \Cref{thm:conv}(ii)}

}
\Return{$(\lambda(\ten{x}^n), \ten{x}^n)$}
\caption{Scheme for perturbed inverse (inner) iteration}
\label{alg:scheme}
\end{algorithm}%
We first restate a result for fixed $\omega_c = 1$ and $\etain = 0$ that corresponds to \cite[Theorem 3]{Rohwedder:11}. In essence, it states that provided small enough perturbations, the iteration converges just as the unperturbed version, but at a modified rate.

\begin{theorem}\label{thm:conv}
 Let $\ten x \in \mathcal{F}_N^K$, $\ten x \neq \ten 0$, and let $k \in \N_0$ be such that $\lambda_k \leq \lambda(\ten x) < \lambda_{k+1}$. Let $\etares$ be chosen such that for $\enum$ with $\varepsilon = \eiter + \enum < 1$, 
 \begin{align}\label{eq:enumcond}
  \|\tilde \Psi(\ten x, \omega_{\mathrm{c}}, 0, \etares) - \Psi(\ten x, \omega_{\mathrm{c}})\|_{\A} \leq \enum \| \ten r(\ten x) \|_{\A^{-1}}.
 \end{align}
 Then we have the following.
 \begin{enumerate}[{\rm\bf(i)}]
\item For $\ten x' := \tilde \Psi(\ten x, \omega_{\mathrm{c}},  0, \etares)$, either $\lambda(\ten x') < \lambda_k$ or $\lambda_k \leq \lambda(\ten x') < \lambda_{k+1}$, and in the latter case, 
 \begin{align}\label{eq:qred}
  \frac{\lambda(\ten x') - \lambda_k}{\lambda_{k+1} - \lambda(\ten x')} \leq q^2(\varepsilon, \lambda_k, \lambda_{k+1}) \frac{\lambda(\ten x) - \lambda_k}{\lambda_{k+1} - \lambda(\ten x)},
 \end{align}
where
\[ 
 q(\varepsilon, \lambda_k, \lambda_{k+1}) = 1 - (1-\varepsilon)(1 - \lambda_k / \lambda_{k+1}).
\]
\item In addition, let $\omega_\ast = \omega_\ast(\ten x, \etares)$ be the minimizer of
 \begin{equation*}
  \lambda_\ast := \min_{\omega \in \R}\ \lambda \bigl(\tilde \Psi(\ten x, \omega, 0, \etares)\bigr),
 \end{equation*}
and for some $\teig \in (0,1)$, let $\ten x' := \tilde \Psi(\ten x, \omega_\ast, \etain, \etares)$ with $\etain$ be chosen such that
\begin{align}\label{eq:modqcond}
 \lambda (\ten x') \leq \lambda_\ast + \teig \bigl(\lambda(\ten x) - \lambda_\ast \bigr).
\end{align}
Then with $ \tilde q^2 = (1-\teig)q^2(\varepsilon, \lambda_k, \lambda_{k+1}) +\teig \in (q^2(\varepsilon, \lambda_k, \lambda_{k+1}),1)$,
\begin{align*}
    \frac{\lambda(\ten x') - \lambda_k}{\lambda_{k+1} - \lambda(\ten x')} \leq \tilde q^2 \frac{\lambda(\ten x) - \lambda_k}{\lambda_{k+1} - \lambda(\ten x)}\,.
\end{align*}

\end{enumerate}
\end{theorem}
\begin{proof}
 The statement \textbf{(i)} follows directly from \cite[Theorem 3]{Rohwedder:11}, subject to the same adaptations described in \Cref{corfor}, for $\ten P^{-1} = \omega_c \ten I$.

 Concerning \textbf{(ii)}, we first observe that by optimality of $\omega_\ast$, we have $\lambda_\ast \leq \lambda (\tilde \Psi(\ten x, \omega_c, 0, \etares))$, and thus \eqref{eq:qred} holds true for $\lambda_\ast$ as well. Assume that $\lambda_k \leq \lambda_\ast$. Then the bound with modified factor $\tilde q$ follows directly by \Cref{perturbedq} for $\ten y = \tilde \Psi(\ten x, \omega_\ast, 0, \etares)$ and $\ten z = \ten x'$. Conversely, if $\lambda_\ast < \lambda_k$, then either we still have $\lambda(\ten x') < \lambda_k$, or otherwise $\lambda(\ten x') \geq \lambda_k$. It the latter case, since $\lambda (\ten x') \leq \lambda_k + t (\lambda(\ten x) - \lambda_k)$, \Cref{perturbedq} can be applied with $\lambda(\ten y) = \lambda_k$.
\end{proof}

In order to determine the optimal step size $\omega_\ast$ in \Cref{thm:conv}(ii)
for given $\ten x$ and $\operatorname{res}_{\etares}(\ten x)$, we solve the corresponding reduced two-dimensional eigenvalue problem
\begin{equation}\label{eq:reducedevp}
 \langle \A \ten v, \ten u \rangle = \lambda_\ast \langle\E \ten v, \ten u \rangle \quad \ \forall \ten u \in \mathrm{span}\{\ten x, \operatorname{res}_{\etares}(\ten x)\}
\end{equation}
for the eigenvector $\ten v_* \in \mathrm{span}\{\ten x, \operatorname{res}_{\etares}(\ten x)\}$ corresponding to the lower eigenvalue $\lambda_\ast$. 

\begin{remark}\label{omegamin}
For numerical stability, we solve \eqref{eq:reducedevp} in terms of the normalized basis vectors $\ten z_1 = \ten S \ten x / \norm{\ten S  \ten x}$ and $\ten z_2 = \ten S  \operatorname{res}_{\etares}(\ten x) / \norm{\ten S  \operatorname{res}_{\etares}(\ten x)}$, that is, we determine $v \in \R^2$ that is the eigenvector corresponding to the lower eigenvalue of
\[   \begin{pmatrix}
\langle \ten H_\gamma \ten z_1, \ten z_1\rangle & \langle \ten H_\gamma \ten z_1, \ten z_2 \rangle \\
 \langle \ten H_\gamma \ten z_1, \ten z_2 \rangle & \langle \ten H_\gamma \ten z_2, \ten z_2 \rangle
\end{pmatrix}
v 
=
\lambda_\ast
\begin{pmatrix}
1 & \langle \ten z_1, \ten z_2 \rangle \\
 \langle \ten z_1, \ten z_2 \rangle & 1
\end{pmatrix}
 v ,
\]
where $\langle \ten H_\gamma \ten z_1, \ten z_1 \rangle =\lambda(\ten x)$, $\langle \ten H_\gamma \ten z_2, \ten z_2 \rangle=\lambda(\operatorname{res}_{\etares}(\ten x) )$ with $\lambda$ as in \eqref{eq:rayleigh1}. We then obtain the optimal step size as
\[
\omega_\ast = - \frac{\norm{\ten S  \operatorname{res}_{\etares}(\ten x)} v_2}{\norm{\ten S  \ten x} v_1}\,.
\]
\end{remark}

\subsection{Error-controlled parameter adaptation}\label{truncdetail}

We now turn to practically feasible ways of ensuring the conditions \eqref{eq:enumcond} and \eqref{eq:modqcond} of \Cref{thm:conv} for convergence of the truncated iterations.

\subsubsection{Adaptive residual rank truncation}\label{truncr}

The following result leads to a practical strategy for inexact residual evaluation with prescribed relative error. Here we adapt the common approach, as used also in \cite{Rohwedder:11}, of reducing directly controllable absolute errors below a certain threshold.

\begin{lemma}\label{lem:truncr}
 With the notation of \Cref{thm:conv}, let $\etares$ be chosen such that $\etares \leq  \zetares \| \ten r(\ten x) \|$ with
 $\zetares \leq (1+c)^{-1} \enum$ for $\enum > 0$ with $\eiter + \enum < 1$. Then the bound \eqref{eq:enumcond} holds true.
\end{lemma}
\begin{proof}
Since $\omega_{\mathrm{c}} = 1$, the introduced perturbation directly simplifies to
\begin{align*}
 & \ \| \tilde \Psi(\ten x, \omega_{\mathrm{c}}, 0, \etares) - \Psi(\ten x, \omega_{\mathrm{c}})\|
 = \norm{ \operatorname{res}_{\etares}(\ten x)  -  \ten r(\ten x)  }  \leq \etares \leq  \zetares \| \ten r(\ten x) \|.
\end{align*}
Since $\| \cdot \|_{\A} \leq \sqrt{1+c}  \| \cdot \| \leq (1+c)  \| \cdot \|_{\A^{-1}}$ by \eqref{specequivA}, assumption \eqref{eq:enumcond} of \Cref{thm:conv} is thus fulfilled if $\zetares \leq (1+c)^{-1} \enum$ is satisfied.
\end{proof}

An efficient approximation of $\operatorname{res}_{\etares}(\ten x)$ ensuring $\norm{ \operatorname{res}_{\etares}(\ten x) - \ten r(\ten x) } \leq \etares$ as well as $\etares \leq \zetares \| \ten r(\ten x) \|$ for given $\etares \geq 0$, without calculation of $\ten r(\ten x)$, can be obtained as follows. 
We assume $\etares$ to be given. With the exponential sum approximations obtained in \Cref{sec:expsum}, the residual takes the form
\begin{align*}
 \ten r(\ten x) = \sum_{k,k' = M^-}^{M^+} \ten S_k \bigl( \ten H_\gamma - \lambda(\ten x) \ten I \bigr) \ten S_{k'} \ten x.
\end{align*}
By first computing the norms of these single terms, the summation can be sorted as
\begin{align*}
 \ten r(\ten x) = \sum_{j = 1}^{J} \ten s_j, \quad J := (M^+ - M^- + 1)^2,
\end{align*}
with $\|\ten s_j\| \leq \|\ten s_{j+1}\|$ for $j = 1,\ldots,J-1$. In order to reduce the computational complexity, the first $J_0$ terms that sum up to at most $\etares/3$ are omitted. With $\ten y_{J_0} := \ten 0$ and
\[
		\ten y_{j} = \operatorname{trunc}_{\varepsilon_j} (\ten y_{j-1} + \ten s_j), \quad j = J_0 + 1, \ldots, J, \quad J_0 := \max \biggl\{ \ell \in \N_0 \colon \sum_{j = 1}^{\ell} \|\ten s_j\| \leq \frac{\etares}3 \biggr\},
		\]
we then set $\operatorname{res}_{\etares}(\ten x) = \operatorname{trunc}_{\etares/3} (\ten y_J)$. In order to maintain the bound prescribed by $\etares$, the single remaining tolerances are set as
\[
   \varepsilon_j = \frac{ \etares  \norm{ \ten s_j }}{ 3 \sum_{i = J_0 + 1}^J \norm{ \ten s_i } }.
\]
Second, in order to maintain the relative bound $\etares  \leq \zetares \| \ten r(\ten x) \|$ for a given $\zetares$, we apply the following strategy.
Starting from an estimate
\begin{align}\label{startetar}
 \etares = \frac{4}{5}  \zetares \sum_{j = 1}^{J} \|\ten s_j\|, \qquad \text{where $\frac{5}{4} \etares \geq \zetares \| \ten r(\ten x) \|$,}
\end{align}
which will avoid unnecessarily large choices, the approximate residual $\operatorname{res}_{\etares}(\ten x)$ is computed as above.
As the norm of the unperturbed residual can be estimated via
\begin{align}\label{rnormestimate}
  \| \ten r(\ten x) \| \geq  \norm{ \operatorname{res}_{\etares}(\ten x)}  - \norm{ \operatorname{res}_{\etares}(\ten x) - \ten{r}(\ten x) }  \geq   \norm{ \operatorname{res}_{\etares}(\ten x)} - \etares, 
\end{align}
the result is then accepted if $\zetares \big(\norm{ \operatorname{res}_{\etares}(\ten x)} - \etares \big) \geq \etares$, or equivalently, if $\norm{ \operatorname{res}_{\etares}(\ten x) }  \geq (1+\zetares^{-1}) \etares$. Otherwise, we replace $\etares$ by $\frac45 \etares$ and repeat the procedure until it necessarily terminates after finitely many (and empirically one or very few) steps. With the result, one further obtains the upper bound $\| \ten r(\ten x) \| \leq \norm{ \operatorname{res}_{\etares}(\ten x)}  + \etares$ as well as
\begin{align} \label{rhobound}
\rho(\ten x) = \frac{\|\ten r(\ten x)\|_{\ten A^{-1}}}{\|\ten x\|_{\ten A}} \leq 
(1-c)^{-1/2} \frac{\|\ten r(\ten x)\|}{\|\ten x\|_{\A}} \leq (1-c)^{-1/2} \frac{\norm{ \operatorname{res}_{\etares}(\ten x) } + \etares}{\|\ten x\|_{\A}}
\end{align}

\subsubsection{Adaptive truncation of iterates}\label{trunciter}

 We provide an a priori bound for the Rayleigh quotient under perturbation in the forward setting for the scalar product $\langle \cdot , \cdot \rangle_{\ten A}$ and operator $\ten B$ as defined in \Cref{corfor}.

\begin{lemma}
With the notation of \Cref{thm:conv}, for $\ten x_\ast := \tilde \Psi(\ten x, \omega_\ast, 0, \etares)$, let
$\teig \in (0,1)$ and $\Delta_{\ten A} > 0$ be chosen such that
\begin{align}\label{DeltaAcond}
\sqrt{1-\Delta_{\ten A}^2} - \rho(\ten x_\ast) \Delta_{\ten A} \geq \sqrt{ \frac{\lambda_\ast}{\lambda_\ast + \teig (\lambda(\ten x) - \lambda_\ast)} }
\end{align}
and let
\begin{equation}\label{eq:etaxdef}
 \etain \leq \frac{ \Delta_{\ten A}  \|\ten x_\ast\|_{\A} }{ \sqrt{1 + c}} .
\end{equation}
Then the condition \Cref{eq:modqcond} is fulfilled with this value of $\teig$.
\end{lemma}
\begin{proof}
By the properties of $\ten x' = \mathrm{trunc}_{\etain}(\ten x_\ast)$,
\begin{equation}\label{eq:anglebound}
\|\ten x'-\ten x_\ast\|_{\ten A} \leq \sqrt{1+c} \|\ten x'-\ten x_\ast\| \leq
\etain \sqrt{1+c} \leq \Delta_{\ten A} \|\ten x_\ast\|_{\ten A},
\end{equation}
and $\Delta_{\ten A}$ satisfies \Cref{DeltaAcond}. 
We now apply \Cref{thm:pertbound} with $L = 0$. 
 To this end, we note that \eqref{eq:anglebound} implies $\sin\angle_{\ten A}(\ten x',\ten x_\ast) \leq \Delta_{\ten A}$. Substitution considering 
 \[ \cos \angle_{\ten A}(\ten x',\ten x_\ast) \geq \cos(\arcsin(\Delta_{\ten A})) = \sqrt{1 - \Delta_{\ten A}^2}, \] 
 and taking into account that $\lambda(\ten x)^{-1} = \mu_{\ten B}(\ten x) \geq 0$, then yields 
\begin{align}\label{ribound}
 \frac{\lambda(\ten x_\ast)}{\lambda(\ten x')} \geq  \left( \max \left\{ \sqrt{1-\Delta_{\ten A}^2} - \rho(\ten x_\ast) \Delta_{\ten A}, 0 \right\} \right)^2 .
\end{align}
We obtain the desired bound \Cref{eq:modqcond} with \Cref{ribound},
 \begin{equation*}
  \lambda(\ten x') \leq  \left(  \sqrt{1-\Delta_{\ten A}^2} - \rho(\ten x_\ast) \Delta_{\ten A}   \right)^{-2} \lambda_\ast \leq \lambda_\ast + \teig (\lambda(\ten x) - \lambda_\ast). \qedhere
 \end{equation*}
\end{proof}

Note that since $\lambda(\ten x) \geq \lambda_*$, a choice of $\teig$ and $\Delta_{\ten A}$ as in \eqref{DeltaAcond} is always possible. 
As $\rho(\ten x_\ast)$ is not directly available but through the upper bound \Cref{rhobound}, one can ensure \Cref{DeltaAcond} via
\begin{align*}
  \sqrt{1-\Delta_{\ten A}^2} - \rho(\ten x_\ast) \Delta_{\ten A} \geq 
  \sqrt{1-\Delta_{\ten A}^2} -  \frac{\|\operatorname{res}_{\etares}(\ten x_\ast)\| + \etares}{ \sqrt{ 1 - c} \|\ten x_\ast\|_{\A}} \Delta_{\ten A} =\sqrt{ \frac{\lambda_\ast}{\lambda_\ast + \teig (\lambda(\ten x) - \lambda_\ast)} }.
\end{align*}
As the desired perturbation bound for the Rayleigh quotient \Cref{eq:modqcond}, however, can be verified a posteriori without computations other than those already required, less conservative tolerances may be applied, such as by assuming $\rho(\ten x_\ast) \leq \|\operatorname{res}_{\etares}(\ten x)\|$.

\subsection{A posteriori error control}\label{sec:aposerrcon}
In the following, for the eigenvector $\ten{u}_1$ with $\|\ten u_1\|_{\A} = 1$ corresponding to the lowest eigenvalue $\lambda_1$ (assumed here to be simple), we always assume without loss of generality that the orientation of $\ten{u}_1$ is the same as that of any arbitrary $\ten x$\ it is compared to, that is, $ \langle \ten u_1, \ten x \rangle_{\ten A} > 0$. In this section, we derive a posteriori error bounds under the assumption that a lower bound for the gap between the sought eigenvalue $\lambda_1$ and the second-lowest one $\lambda_2$ is known.
Specificially, we assume that we know a constant $\delta > 1$ with
\begin{align*}
 \delta \geq \delta_0 := \frac{\lambda_2}{\lambda_2 - \lambda_1} \in (1,\infty).
\end{align*}
The case that is most favorable in the following estimates is $\delta$ and $\delta_0$ both being close to one, which in view of \eqref{eq:eigenvalueordering} corresponds to the largest possible gap.

\begin{prop}\label{raylangleb} 
For all $\ten x \in \mathcal{F}_N^K$, we have
\begin{align}\label{eq:sinbound}
 \left( 1 - \frac{\lambda_1}{\lambda(\ten x)} \right) \leq \sin^2 \angle_{\ten A}(\ten{u}_1,\ten x) \leq 
   \delta \left( 1 - \frac{\lambda_1}{\lambda(\ten x)} \right),
\end{align}
as well as
\begin{align}\label{eq:2ndsinbound}
  \sin \angle_{\ten A}(\ten{u}_1,\ten x) \leq \left\|\ten{u}_1 - \frac{\ten x}{\|\ten x\|_{\A}} \right\|_{\A} = 2 \sin \left(\frac{1}{2} \angle_{\A}(\ten{u}_1,\ten x) \right)  \leq \sqrt{2} \sin \angle_{\ten A}(\ten{u}_1,\ten x).
\end{align}
\end{prop}

\begin{proof}
 The first inequality in \eqref{eq:sinbound} follows, for instance, from \cref{thm:pertbound} with $\ten p = \ten x$, $\ten v= \ten u_1$ and  $L = 0$; the second follows directly from \cref{sinbound} considering that $\gamma \geq \frac{\mu_2}{\mu_1} = \frac{\lambda_1}{\lambda_2}$ is equivalent to $\delta = (1 - \gamma)^{-1} \geq \delta_0$. 

The equality in \eqref{eq:2ndsinbound} is based on simple geometry. The lower bound in \eqref{eq:2ndsinbound} follows from $\sin \angle_{\ten A}(\ten{u}_1,\ten x) = \min_{\alpha \in \R} \| \ten u_1 - \alpha \ten x\|_{\A}$, whereas the upper bound is again elementary.
\end{proof}

\begin{lemma}\label{lambdaqbound}
 Assume that both
 \begin{align}\label{lambdaqassumptions}
  \rho^2(\ten x) \leq \frac{1}{4 \delta (\delta-1)}, \quad  \sin^2\angle_{\A}(\ten{u}_1,\ten x) \leq \frac{\delta}{2\delta - 1}. 
 \end{align}
 Then
 \begin{align}\label{eq:bbetarho}
 1 - \frac{\lambda_1}{\lambda(\ten x)} \leq \bfun_\delta\big(\rho(\ten x) \bigr) \rho^2(\ten x), \quad \bfun_\delta( \rho(\ten x)) := 
\frac{2 \delta }{1 + 2 \delta \rho^2(\ten x) + \sqrt{1 - 4 (\delta - 1) \delta \rho^2(\ten x)}}. 
 \end{align}
\end{lemma}

\begin{proof}
 Follows directly from \cref{qbound,raylangleb} considering $\rho(\ten x) = \tau(\ten x)$ and that $\gamma \geq \frac{\mu_2}{\mu_1} = \frac{\lambda_1}{\lambda_2}$ is equivalent to $\delta = (1 - \gamma)^{-1} \geq \delta_0$. 
\end{proof}
The first condition in \eqref{lambdaqassumptions} is trivial to check, and is required for $\bfun_\delta(\rho(\ten x))$ to give a real number. 
By \cref{eq:sinbound}, the second assumption in \cref{lambdaqassumptions} is implied by
\begin{equation}\label{eq:imply2ndeq}
 \frac{\lambda_1}{\lambda(\ten x)} \geq \frac{2\delta - 2}{2\delta - 1},
\end{equation}
which in turn corresponds to \cref{gammacoresp}.
This requires some a priori knowledge on the magnitude of $\lambda_1$. However, a rough estimate suffices, as it does not influence the accuracy of the presented error estimates.
From \cref{eq:bbetarho}, one can directly derive the relative eigenvalue bound
\begin{align}\label{eq:releigerr}
 \frac{\lambda(\ten x) - \lambda_1}{\lambda_1} \leq \frac{\bfun_\delta(\rho(\ten x))}{1-\bfun_\delta(\rho(\ten x)) \rho^2(\ten x)} \rho^2(\ten x).
\end{align}
Further, via \Cref{raylangleb}, one obtains
\begin{align}\label{eq:nicesin2bound}
 \sin^2\angle_{\A}(\ten u_1,\ten x) \leq \delta \bfun_\delta( \rho(\ten x)) \rho^2(\ten x) =
\frac{2 \delta^2 \rho^2(\ten x)}{1 + 2 \delta \rho^2(\ten x) + \sqrt{1 - 4 (\delta - 1) \delta \rho^2(\ten x)}}.
\end{align}
It is easy to see that this bound for $\rho \rightarrow 0$ quickly approaches from above its quadratic Taylor approximation $\rho^2 \delta^2$. For small $\rho$, this  yields
\begin{align}\label{eq:sinlinappr}
\frac{\lambda(\ten x) - \lambda_1}{\lambda_1}  \leq \delta \rho^2(\ten x) + \mathcal{O}(\rho^4(\ten x)), \quad
 \sin \angle_{\A}(\ten u_1,\ten x) \leq \delta \rho(\ten x)  + \mathcal{O}(\rho^3(\ten x)). 
\end{align}
With the following result, the eigenvector error can efficiently be controlled a posteriori in terms of $\rho(\ten x)$, which is vital to the outer iteration introduced in \cref{sec:outeriter}. 

\begin{theorem}\label{aposterbound}
 For all $\ten x$ that satisfy the assumptions in \Cref{lambdaqbound},
 \begin{align}\label{eq:everrbound}
 \left\| \ten{u}_1 - \frac{\ten x}{\|\ten x\|_{\A}} \right\|_{\A} \leq 2 \sin \left(\frac{1}{2} \arcsin \bigl( \sqrt{ \delta  \bfun_\delta( \rho(\ten x)) } \, \rho(\ten x) \bigr) \right).
 \end{align}
\end{theorem}

\begin{proof}
As in \Cref{raylangleb}, we have
\begin{align*}
 \left\| \ten{u}_1 - \frac{\ten x}{\|\ten x\|}_{\A} \right\|_{\A} = 2 \sin \left(\frac{1}{2} \angle_{\A}(\ten{u}_1,\ten x) \right) 
 = 2 \sin \left(\frac{1}{2} \arcsin ( \sin\angle_{\ten A} (\ten{u}_1,\ten x)) \right).
\end{align*}
We now combine this identity with the angle bound of \Cref{raylangleb,lambdaqbound}.
\end{proof}

\begin{remark}
With regard to the assumed knowledge of only $\rho(\ten x)$, $\lambda(\ten x)$ and $\delta$, the estimates \cref{eq:sinbound,eq:bbetarho,eq:everrbound} are sharp, as they are restatements of the sharp bounds derived in \cref{sec:rayqup} for the corresponding forward iteration.
\end{remark}

The complete numerical scheme for the inverse iteration, together with the error control according to \Cref{aposterbound}, are summarized in \cref{alg:innerpinvit}. 

\begin{algorithm}[t!]
\DontPrintSemicolon
\SetKwFunction{pinvit}{PINVIT}
\SetKwInOut{Input}{input}\SetKwInOut{Output}{output}
\SetKwProg{KwProg}{function}{}{end}
\SetKwInput{KwRequires}{further requires}
\Input{tensor $\ten{x}^0$, tolerance $\tau$}
\KwRequires{$\enum \in (0, 1-\eiter)$, $\teig \in (0,1)$, $c \in (0,1)$ as in \cref{specequivA}, $\delta \geq \frac{\lambda_2}{\lambda_2 - \lambda_1}$}
\Output{approximation $(\lambda(\ten{x}^n), \ten{x}^n)$ to eigenpair $(\lambda_1,\ten u_1)$ of $(\ten A, \ten E)$ with $\| \ten u_1 - \ten x^n \|_{\A} \leq  \tau$ and upper bound $\rho^n$ to residual}
\BlankLine
\KwProg{\pinvit{$\ten x^0, \tau$}}{
$\zetares \gets (1+c)^{-1} \enum$\tcp*[r]{\normalfont{} \Cref{lem:truncr}}
$\etares \gets \infty$
\BlankLine
\For{$n = 0,1,\ldots$}{
    \tcp{(exactly) determine Rayleigh quotient}
    $\lambda^n \gets \lambda(\ten{x}^n)$\tcp*[r]{\normalfont{} \Cref{exactrq}}
    $\ten{x}^n \gets \frac{\ten{x}^n}{\|\ten{x}^n\|_{\A}}$ \;
    \BlankLine
    \tcp{determine approximate residual and bound}
    $\etares \gets \min \{ 2 \etares, \frac{4}{5} \zetares \sum_{k,k'} \| \ten S_k \ten H_\gamma \ten S_{k'} \ten{x}^n - \lambda^n \ten S_k \ten S_{k'} \ten{x}^n\|\} $\tcp*[r]{\normalfont{} \Cref{startetar}} 
    \While {$\norm{\operatorname{res}_{\etares}(\ten{x}^n)} < (1+\zetares^{-1}) \etares$}{
        $\etares \gets \frac{4}{5} \etares$\tcp*[r]{\normalfont{} \Cref{truncr}}
        }
    $\ten r^n \gets \operatorname{res}_{\etares}(\ten{x}^n)$\;
    $\rho^n \gets (1-c)^{-1/2} \frac{\|\ten r^n\| + \etares}{\|\ten{x}^n\|_{\A}}$\tcp*[r]{\normalfont{} \Cref{rhobound}} 
    \BlankLine
    \tcp{termination criterion using a posteriori bound}
    \If{$2 \sin \left(\frac{1}{2} \arcsin \big( \sqrt{ \delta \bfun_\delta( \rho^{n}) } \,\rho^{n} \big) \right) \leq \tau$}{
      \Return{$(\lambda(\ten{x}^n), \ten{x}^n, \rho^n)$}  \tcp*[r]{\normalfont{} \Cref{aposterbound}}
    }
    \BlankLine
    \tcp{determine optimal step size}
    solve $ \ten{X}^\trp \A \ten X z = \lambda_\ast \ten{X}^\trp \E \ten{X} z$ for $\ten X = \begin{bmatrix} \ten{x}^n &  \ten r^n \end{bmatrix}$\tcp*[r]{\normalfont{} \Cref{omegamin}} 
   $\ten x_\ast \gets z_1 \ten{x}^n + z_2 \ten r^n$
    \BlankLine
    \tcp{truncate iterate}
    $\etain \gets (1+c)^{-1/2} \Delta_{\ten A}  \|\ten x_\ast\|_{\A}$ where $\sqrt{1-\Delta_{\ten A}^2} - \rho^n \Delta_{\ten A} = \sqrt{ \frac{\lambda_\ast}{\lambda_\ast + \teig (\lambda^n - \lambda_\ast)}  }$\;
    $\ten{x}^{n+1} \gets \operatorname{trunc}_{\etain}(\ten x_\ast)$\tcp*[r]{\normalfont{} \Cref{trunciter}} 
}  
}
\caption{Perturbed inverse (inner) iteration}
\label{alg:innerpinvit}
\end{algorithm}

\subsection{A priori bound on the required number of inner iteration steps}
As in \cref{sec:aposerrcon}, let
\begin{equation}\label{eq:defkappa}
 \delta_0 := \frac{\lambda_2}{\lambda_2 - \lambda_1} \in (1,\infty).
\end{equation}
as well as $\ten u_1$ be a suitably oriented eigenvalue with $\|\ten u_1\|_{\A} = 1$ corresponding to $\lambda_1$. For the result $\ten x^n$ returned by \Cref{alg:innerpinvit} with tolerance $\tau>0$, we write 
\[  \ten x^n =  \PINVIT_\tau(\ten x), \]
which guarantees that the returned vector $\ten x^n$ satisfies $\| \ten u_1 - \ten x^n \|_{\A} \leq \tau$, as well as $\norm{\ten x^n}_{\ten A}=1$.
We show in this section that the number of required steps to decrease the eigenvalue error by a fixed constant can, however, be bounded explicitly by a value independent of the starting vector $\ten x^0$. This number thus does not increase in the course of the outer iteration. We do so for an arbitrary factor $u \geq 1$, but we make a specific choice in \cref{thm:quasioptr} below.

In order to apply \Cref{thm:conv}, we first relate the improvement in eigenvalue approximations by Rayleigh quotients that it guarantees to improvements in the approximate eigenvector.

\begin{lemma}\label{rankoptaprior}
Let $\ten x$ be such that for a $\delta \geq \delta_0$, with $\delta_0$ as in \eqref{eq:defkappa},
\begin{equation}\label{assrankopt}
 \frac{\lambda_1}{\lambda(\ten x)} \geq 1 - \frac{1}{2 \delta^2}    \,.
\end{equation}
Then $\lambda(\ten x) < \lambda_2$. If in addition, $\ten y$ is such that $\lambda(\ten y) \leq \lambda_2$ and for some $u \geq 1$,
\begin{equation}\label{actualineq}
 \frac{\lambda(\ten y) - \lambda_1}{\lambda_2 - \lambda(\ten y)} \leq \nu  \frac{\lambda(\ten{x}) - \lambda_1}{\lambda_2 - \lambda(\ten{x})} , \quad \nu := \frac{1}{4 u^2 \delta  - 1},
\end{equation} 
 then we have
 \begin{equation}\label{eq:normubound}
  \left\| \ten u_1 - \frac{\ten y}{\|\ten y\|_{\A}} \right\|_{\A} \leq \frac{1}{u} \sin \angle_{\A}(\ten u_1,\ten x).
\end{equation}
\end{lemma}
\begin{proof}
For $\delta_0$ as defined in \eqref{eq:defkappa}, the inequality $\frac{\lambda_1}{\lambda(\ten x)} < 1 - \frac{1}{\delta_0^{2}}$ is equivalent to $\lambda(\ten x) < \lambda_2$, 
and the latter estimate thus also follows from \eqref{assrankopt} for any $\delta \geq \delta_0$.

 Assume now that \cref{actualineq} holds true, and let $s := \sin \angle_{\A}(\ten u_1,\ten x)$. Then since $\lambda(\ten{x}) \leq \frac{\lambda_1}{1 - s^2}$ as a consequence of  \cref{eq:sinbound}, 
 \begin{equation}\label{eq:rankopaprior1}
    \frac{\lambda(\ten{y}) - \lambda_1}{\lambda_2 - \lambda(\ten{y})}\leq \nu  \frac{\lambda(\ten{x}) - \lambda_1}{\lambda_2 - \lambda(\ten{x})}
     \leq \nu \frac{\lambda_1 s^2}{\lambda_2 (1 - s^2) - \lambda_1} .
 \end{equation}
 Next, note that by \cref{eq:sinbound}, the assumption \cref{assrankopt} implies that $s \leq \frac{1}{\sqrt{2 \delta_0}}$, and thus 
 \begin{equation}\label{eq:rankopaprior2}
\nu =   \frac{1}{4 u^2 \delta  - 1} \leq \frac{1}{4 u^2 \delta_0  - 1} =  \frac{1 - \delta_0 / (2 \delta_0)}{2u^2 \delta_0 - \delta_0 / (2 \delta_0)}  \leq  \frac{1 - s^2 \delta_0}{2u^2 \delta_0 - s^2  \delta_0},
 \end{equation}
 where we have used that the expression on the right is monotonically decreasing with respect to $s$. Inserting this into \eqref{eq:rankopaprior1} and using that $\delta_0 =\lambda_2 /(\lambda_2 - \lambda_1)$ gives
 \[
      \frac{\lambda(\ten{y}) - \lambda_1}{\lambda_2 - \lambda(\ten{y})} \leq  \frac{\lambda_1}{\lambda_2}\frac{s^2}{2u^2 - s^2} .
 \]
A direct calculation shows that the latter inequality is in turn equivalent to
\begin{equation}\label{eq:lambdayest}
 \lambda(\ten{y}) \leq \lambda_1 \frac{2 \delta_0 u^2}{2 \delta_0 u^2 - s^2}.
\end{equation}
Using \cref{eq:sinbound} and then inserting \eqref{eq:lambdayest} we finally obtain
\begin{equation*}
\sin \angle_{\A}(\ten u_1,\ten y)^2  \leq \delta_0 \left( 1 - \frac{\lambda_1}{\lambda(\ten y)} \right) 
  \leq   \frac{s^2}{2 u^2},
\end{equation*}
and thus \cref{eq:normubound} follows with \cref{eq:2ndsinbound}.
\end{proof}

\begin{theorem}\label{thm:reduction}
 Let $\ten x^0$ fulfill
\begin{align*}
 \frac{\lambda_1}{\lambda(\ten x^0)} \geq 1 - \frac{1}{2 \delta^2}.
\end{align*}
Let $\ten x^n$ denote the $n$-th iterate of \cref{alg:innerpinvit} without termination criterion, with parameters $\teig \in (0,1)$ and $\enum \in (0,1 - \eiter)$.
 Then for every $u \geq 1$ and $\varepsilon = \eiter + \enum$, if
 \begin{align*}
  n \geq  \log \left( \frac{1}{4u^2 \delta - 1}  \right) / \log \left( \left(1 - \teig \right) \left(1 - \frac{1 - \varepsilon}{\delta} \right)^2 + \teig \right),
 \end{align*}
then we have
 \begin{align*}
 \left\| \ten u_1 - \frac{\ten x^n}{\| \ten x^n \|_{\A}} \right\|_{\A} \leq \frac{1}{u} \left\| \ten u_1 - \ten x^0 \right\|_{\A}.
 \end{align*}
\end{theorem}
\begin{proof}
Since $\lambda(\ten x^0) < \lambda_2$ by our assumptions, we can apply \Cref{thm:conv}(ii). \Cref{rankoptaprior} thus yields that $n$ steps suffice for the reduction of the eigenvector residual by the factor $\frac{1}{u}$, taking into account that $\sin \angle_{\A}(\ten u_1,\ten x^0) \leq \| \ten u_1 - \ten x^0 \|_{\A}$.
\end{proof}

\subsection{Rank and complexity bounds}\label{sec:outeriter}

In order to guarantee quasi optimal ranks as layed out in \cite{B23}, the iterate is truncated to specific tolerances after each a certain number of inner iterations.  By \cref{aposterbound}, $\PINVIT_\tau(\ten x)$ performs sufficiently many inner iterations such that
\begin{align*}
 \| \ten u_1 - \PINVIT_\tau(\ten x) \|_{\A} \leq \tau.
\end{align*}
In the {\em outer iteration} summarized in \cref{alg:outerpinvit}, for a starting vector $\ten{y}^0$ and $\alpha > 0$, we iterate %
\begin{align*}
 \ten{y}^{m+1} = \operatorname{trunc}_{\theta \etaout^m}\left( \PINVIT_{\tau_m}(\ten{y}^m) \right), \quad \tau_m := \frac{\etaout^m}{2 (1+(1+\alpha)\kappa)}    , \quad \theta := \frac{(1+\alpha) \kappa}{2(1+(1+\alpha)\kappa)}
\end{align*}
for $\etaout^{m+1} := \frac{1}{2} \etaout^m$, $m \in \N$. 
From \cref{eq:ttsvdquasiopt}, using that $\sqrt{1-c} \norm{\ten x} \leq \norm{\ten x}_{\ten A} \leq \sqrt{1+c} \norm{\ten x} $ by \eqref{specequivA}, we obtain
\begin{equation*}
   \norm{ \ten x -  \operatorname{trunc}_{s_1,\ldots,s_{K-1}}(\ten x) }_{\ten A} \leq \kappa  \min\bigl\{ \norm{\ten x - \rmap{\rep Y}}_{\ten A} \colon \rep Y \text{ of ranks at most $s_1,\ldots, s_{K-1}$} \bigr\}
\end{equation*}
with the modified quasi-optimality constant
\begin{equation}\label{kappadef}
 \kappa = \sqrt{\frac{1 + c}{1-c}} \sqrt{K-1} \,.
\end{equation}
We obtain the following rank bound in terms of the best approximation ranks defined in \eqref{eq:defrbest}.

\begin{theorem}\label{thm:quasioptr}
 For the sequence $\ten y^m$ generated by \Cref{alg:outerpinvit}, we have
 \begin{align*}
  \|\ten{u}_1 - \ten y^m\|_{\A} \leq \etaout^m, \quad \norm{\ranks(\ten y^m)}_\infty \leq r_{\mathrm{best}}\bigl(\ten{u}_1, (1+(1+\alpha)\kappa)^{-1} \alpha \etaout^m \bigr).
 \end{align*}
\end{theorem}
\begin{proof}
 Let the statement hold for some $m \in \N_0$. Then we obtain
 \begin{align}\label{priortrunccond}
 \left\| \ten{u}_1 - \PINVIT_{\tau_m}(\ten y^m) \right\|_{\A} \leq \frac{\etaout^m}{2 (1+(1+\alpha)\kappa)}.
\end{align}
The remainder follows analogously to \cite[Theorem 5.5]{B23}, taking into account the use of the $\ten A$-norm and the equivalence constants with respect to the Euclidean norm.
\end{proof}

\begin{proof}[Proof of \Cref{thm:main}]
The bounds on $ \|\ten{u}_1 - \ten y^m\|_{\A}$ and $\norm{\ranks(\ten y^m)}_\infty$ follow directly from \Cref{thm:quasioptr}, taking into account that \eqref{kappadef} and that $\alpha$ and $c$ can be chosen to be independent of $K$. The statement on the number of inner iterations follows from \Cref{thm:reduction} by noting that this number depends logarithmically on the factor of error reduction $2(1 + (1+\alpha)\kappa) \leq \tilde C \sqrt{K-1}$, with $\tilde C>0$ independent of $K$.
\end{proof}

As noted in Section \ref{sec:eocc}, the rank bounds of \Cref{thm:main} are for the results of each outer iteration in \cref{alg:outerpinvit}.
Using our strategy, at this point we cannot obtain sharp bounds on the intermediate ranks arising in the inner iteration of \cref{alg:outerpinvit}, which is given by \cref{alg:innerpinvit}, but we can obtain upper bounds that  are more problem-specific than the result of \Cref{thm:main}.
To this end, we rely on the bounds on ranks of the involved operators from \cite{BGP:22} given in \Cref{thm:opranks}. 

\begin{remark}\label{rem:intermediateranks}
For all intermediate results of the inner iteration \Cref{alg:innerpinvit} performed in every step of \Cref{alg:outerpinvit}, we can also derive rank bounds by estimating the possible inflation of ranks over the course of the inner iteration. While such bounds cannot be expected to be sharp, they give some first bounds on the total computational costs of the method:
\begin{enumerate}[{\rm\bf(i)}]
\item For general $\ten H = \ten T + \ten V$ as in \cref{eq:hamil}, so that \Cref{thm:opranks}(i) applies, the ranks of all intermediate quantities produced by \Cref{alg:innerpinvit} in the inner iteration starting from $\ten y^m$ are with certain $C_1, C_2>0$ bounded by 
\[ C_1  K^{C_2 ( \log K + 2 (M^+ - M^- + 1) ) } r_{\mathrm{best}}\bigl(\ten{u}_1, (1+(1+\alpha)\kappa)^{-1} \alpha \etaout^m \bigr). \]
 Here $C_2$ depends on the ratio of $C_{\mathrm{lower}}$ and $C_{\mathrm{upper}}$ in \eqref{eq:lowerupper}, which can depend on $K$, but otherwise $C_1, C_2$ are independent of $K$. Finally, while $M^+$ can be chosen to be fixed, $M^-$ depends linearly on $\log \max(\sigma(\ten D)) /  \min(\sigma(\ten D))$.

\item For $\ten H = \ten T + \ten V$ such that the additional assumptions of \Cref{thm:opranks}(ii) apply with parameters $d, \tilde d$, the bound is modified to 
\[
\tilde C_1  K^{\tilde C_2 ( \log \max\{ d, \tilde{ d}^2\} + 2 (M^+ - M^- + 1) )  } r_{\mathrm{best}}\bigl(\ten{u}_1, (1+(1+\alpha)\kappa)^{-1} \alpha \etaout^m \bigr),
\]
with the same comments as above applying to $\tilde C_1, \tilde C_2$ and $M^+, M^-$.
\end{enumerate}
\end{remark}

The bounds in \Cref{rem:intermediateranks} turn out to be far from sharp in comparison to the numerically observed ranks for the intermediate results in inner iterations.
To keep these ranks low, in practice one can also use additional intermediate recompressions with small tolerances. However, quantifying their effect a priori would likely require strongly problem-adapted techniques.

\begin{algorithm}[t!]
\DontPrintSemicolon
\SetKwFunction{pinvit}{PINVIT}
\SetKwFunction{opinvit}{OPINVIT}
\SetKwInOut{Input}{input}
\SetKwInOut{Output}{output}
\SetKwProg{KwProg}{function}{}{end}
\SetKwInput{KwRequires}{further requires}
\Input{tensor $\ten{y}^0$, tolerance $\tau$}
\KwRequires{$\alpha > 0$, $c$ as in \cref{specequivA}, $\delta \geq \frac{\lambda_2}{\lambda_2 - \lambda_1}$}
\Output{approximation $(\lambda(\ten{y}^+),\ten{y}^+)$ to generalized eigenpair $(\lambda_1,\ten{u}_1)$ of $(\ten A, \ten E)$ with $| \lambda(\ten y^+) - \lambda_1| < \lambda_1  \tau$}
\BlankLine
\KwProg{\opinvit{$\ten{y}^0, \tau$}}{
$\theta \gets \frac{(1+\alpha) \kappa}{2(1+(1+\alpha)\kappa)}$\tcp*[r]{\normalfont{} \Cref{kappadef}, \Cref{thm:quasioptr}} 
$\eta_0 \gets \frac{1}{2\theta\sqrt{D}}$ \tcp*[r]{\normalfont{} with $D := 1$ for \Cref{alg:innerpinvit}}
\BlankLine  
\For{$m = 0,1,\ldots$}{
    \tcp{perform inner iteration}
    $(\lambda(\ten{y}^+), \ten{y}^+, \rho^+) \gets \pinvit \bigl(\ten y^m, \frac12 (1+(1+\alpha)\kappa)^{-1} \etaout^m \bigr)$ \tcp*[r]{\normalfont{} \Cref{alg:innerpinvit}} 
    
    \BlankLine
    \tcp{termination criterion}
    \If{$\frac{\bfun_\delta(\rho(\ten x))}{1-\bfun_\delta(\rho(\ten x)) \rho^2(\ten x)} \rho^2(\ten x) <  \tau$}{
    \Return{$(\lambda(\ten{y}^+), \ten{y}^+)$}  \tcp*[r]{\normalfont{} \Cref{eq:releigerr}}  
    }
    \BlankLine
    \tcp{ensure quasi-optimal ranks}
    $\ten{y}^{m+1} \gets \operatorname{trunc}_{\etaout^m \theta}(\ten{y}^+)$\tcp*[r]{\normalfont{} \Cref{thm:quasioptr}}
    $\etaout^{m+1} \gets \etaout^m / 2$\;
}
}
\caption{Outer perturbed inverse iteration}
\label{alg:outerpinvit}
\end{algorithm}

\section{Simultaneous Approximation of Multiple Eigenvalues}\label{sec:multiple}

In this section, we consider a generalization of \Cref{alg:innerpinvit,alg:outerpinvit} to the simultaneous approximation of multiple eigenvalues and eigenvectors. For this, let
\begin{equation}
 \ten{X} := (\ten{x}_1,\ldots,\ten{x}_D), \; \ten{Y} := (\ten{y}_1,\ldots,\ten{y}_D)\;\, \in\,\; \cF^K_N \otimes \R^D
\end{equation}
be collections of $D$ approximate eigenvectors, which we treat as matrices with $D$ columns. We define
\begin{align}\label{rpara}
 		\Lambda(\ten X) := 
 		\diag(\lambda(\ten{x}_1),\ldots,\lambda(\ten{x}_D))
\end{align}
as the matrix containing the Rayleigh quotients of $\ten{x}_1,\ldots,\ten{x}_D$ on the diagonal.
The residual is given by
\begin{equation}\label{respara}
 \ten R(\ten X) = \A \ten X - \E \ten X \Lambda(\ten X) = \sum_{k,k'} \bigl( \ten S_k \ten H_\gamma  \ten S_{k'} \ten X - \ten S_{k} \ten S_{k'} \ten X \Lambda(\ten X) \bigr).
\end{equation}

The inner iteration for $D > 1$ (summarized in \cref{alg:jointinnerpinvit}) is analogous to \cref{alg:innerpinvit} for $D = 1$, in which case both algorithms coincide. 
The joint representation of eigenspace bases described in \cref{ssc:stepsize} allows for an efficient implementation.
However, it also entails that we can apply only a single rank truncation tolerance to the entire tensor including all columns of the represented matrix iterate.
Such truncations are thus handled under the worst-case assumption that the entire error may occur in a single one of the $D$ approximate eigenvectors, but without knowing which one. With this restriction, the outer iteration requires only minor adaptation compared to the single-vector case $D=1$, such that the first approximate eigenvector still provides the same adaption criterion. Additionally, the residual of this first column is heuristically compared with those remaining for the termination of the outer iteration. The development of a more involved scheme that uses all $D$ Rayleigh quotients and in particular controls also the remaining $D-1$ eigenvalue errors sufficiently tightly, as well as as generalized convergence estimates, are left for future work.

\subsection{Subspace minimization}\label{ssc:stepsize}
We describe the choice of stepsize for multiple eigenvalues in detail, because it differs more substantially from the case $D = 1$ described in \cref{eq:reducedevp}. Instead of an update that corresponds to an optimal step size, we solve for the lowest $D$ Ritz values with respect to the subspaces spanned by $\ten X^n$, as well as $\ten R^n$ when computing the unperturbed update $\ten X_\ast$. The solution is then given by the reduced eigenvalue problem
\begin{align*}
  \ten{\mathcal{X}}^\trp \A  \ten{\mathcal{X}} V =  \ten{\mathcal{X}}^\trp \E  \ten{\mathcal{X}} V \Lambda_\ast , \quad
\text{where $\ten{\mathcal{X}} = (\ten{X}^n\ \ten{R}^n) \in \mathcal{F}_N^K \otimes \R^{2D}$,} \quad V \in \R^{2D \times D}. %
\end{align*}
Symmetrically scaling the eigenvalue problem via $W := \mathrm{diag}(\Xi_{\ten x},\Xi_{\ten r}) V$, where
\begin{align}\label{eq:subspacerescale}
 \Xi_{\ten x} := \mathrm{diag}(\|\ten{x}^n_1\|_{\E}, \ldots, \|\ten{x}^n_D\|_{\E}), \quad \Xi_{\ten r} := \mathrm{diag}(\|\ten{r}^n_1\|_{\E}, \ldots, \|\ten{r}^n_D\|_{\E}),
\end{align}
leads to a normalized formulation that improves the numerical stability of solvers similarly to \cref{omegamin}. Furthermore, based on the prior subspace minimization on the span of $\ten X^n$ only (which is the first step of \cref{alg:jointinnerpinvit}), the eigenvalue relation
\begin{equation}\label{eq:Xnnormal}
I_D = (\ten X^n)^\trp \A \ten X^n = (\ten X^n)^\trp \E \ten X^n \Lambda^n = \Xi_{\ten x}^2 \Lambda^n
\end{equation}
can be assumed to hold true. This then results in the formulation $B W = M W \Lambda_\ast$, with 
\begin{align*}
B =
\begin{pmatrix}
 \Lambda^n & \Xi_{\ten x}^{-1} (\ten{X}^n)^\trp \A \ten R^n \Xi_{\ten r}^{-1} \\
\Xi_{\ten r}^{-1} (\ten R^n)^\trp \A \ten{X}^n \Xi_{\ten x}^{-1} & \Xi_{\ten r}^{-1} (\ten R^n)^\trp \A \ten R^n \Xi_{\ten r}^{-1}
\end{pmatrix}
\end{align*}
as well as the likewise symmetric weight matrix
\begin{align*}
M =
\begin{pmatrix}
 I_D & \Xi_{\ten x}^{-1} (\ten{X}^n)^\trp \E \ten R^n \Xi_{\ten r}^{-1} \\
\Xi_{\ten r}^{-1} (\ten R^n)^\trp \E \ten{X}^n \Xi_{\ten x}^{-1} & \Xi_{\ten r}^{-1} (\ten R^n)^\trp \E \ten R^n \Xi_{\ten r}^{-1}
\end{pmatrix}.
\end{align*}
Both matrices are calculated exactly (subject to only common round-off errors) using the summation formulas
\begin{align*}
 \ten{Y^\trp} \A \ten{Z} = \sum_{k,k'} (\ten S_k \ten{Y})^\trp \ten H_\gamma (\ten S_{k'} \ten Z), \quad  \ten{Y^\trp} \E \ten{Z} = \sum_{k,k'} (\ten S_k \ten{Y})^\trp (\ten S_{k'} \ten Z)
\end{align*}
for any pairs of matrices $\ten Y, \ten Z \in \mathcal{F}_N^K \otimes \R^D$ represented by joint block MPS factorizations as described in \cref{sec:jointMPS} below.
The matrix of the $D$ eigenvectors $W \in \R^{2D \times D}$ corresponding to the lowest eigenvalues $\Lambda_\ast = \mathrm{diag}(\lambda_{\ast 1},\ldots,\lambda_{\ast D})$ then yields the (unperturbed) update
\begin{align*}
 \ten{X}_\ast = \ten{\mathcal{X}} \mathrm{diag}(\Xi_{\ten x},\Xi_{\ten r})^{-1} W
 = \ten{X}^n \Xi_{\ten x}^{-1} W_1 + \ten R^n \Xi_{\ten r}^{-1} W_2, \quad W =: \begin{pmatrix}
              W_1 \\                                                              
             W_2
\end{pmatrix}, \ W_1,W_2 \in \R^{D \times D}.
\end{align*}
Moreover, it follows that the columns of $\ten{X}_\ast$ are always both $\A$- and $\E$-orthogonal such that
$\ten X_\ast^\trp \A \ten X_\ast = W^\trp B W$ and $\ten X_\ast^\trp \E \ten X_\ast = W^\trp M W$ are diagonal. We therefore rescale the colums of $W$ and thus $\ten X_\ast$ such that the represented columns become in fact $\A$-orthonormal. This is likewise the basis for the prior assumption \cref{eq:Xnnormal}. 

\subsection{Joint tensor parameterization of several eigenvectors}\label{sec:jointMPS}
In the simultaneous approximation of several eigenvectors $\ten{X} = (\ten{x}_1,\ldots,\ten{x}_D)$ by low-rank tensor methods, instead of maintaining a separate tensor representation for each eigenvector, it is natural to instead consider a \emph{joint tensor parameterization} where the index numbering the eigenvectors is simply treated as an additional mode in the tensor decomposition. This strategy is also called a \emph{block tensor train} representation in the literature \cite{Dolgov:14,KSU:14}.
For the joint approximation of eigenvectors $\ten{x}_i \in \cF^K$, $i=1,\ldots,D$, for $D \in \N$, this amounts to considering the matrix product state representation of the tensor 
\[ 
   \ten{\check x} = ( \ten{x}_{i,\alpha_1,\ldots,\alpha_K} )_{i \in \{1,\ldots,D\}, \alpha \in \{0,1\}^K}
\]
in the form
\[
    \ten{\check x} = \check{X} \SKP X_1 \SKP \cdots \SKP X_K.
\]
In the same manner, the additional core can be added at any other position $1 \leq p \leq K+1$ (that is, before component $X_p$) to yield a representation 
\[
  \ten{\check x} = X_1\SKP \cdots\SKP X_{p-1} \SKP \check{X} \SKP X_{p} \SKP \cdots \SKP X_K,
\]
where $X_k \in \R^{r_{k-1} \times 2 \times r_k}$ for $k=1,\ldots,p-1$, $X_k \in \R^{r_k \times 2 \times r_{k+1}}$ for $k=p,\ldots,K$, and the additional core $\check{X} \in \R^{r_{p} \times D \times r_{p+1}}$ carries the eigenvector index. Note that now, $r_K > 1$ and the rank indices for the component on the right are shifted by one. The new ranks $(1,r_1,\ldots,r_K,1)$ required for representing $\ten{\check x}$ in this block representation can be at most the sum of the ranks of the collected eigenvectors.

This representation extends naturally to the block-sparse representation of $\ten x_i \in \cF_N^K$: For $\ten{x}_i$, $i=1,\ldots,D$, with $\rep{X}_i = (X_{i,1}, \ldots, X_{i,K})$, we set the components $X_k \in \R^{r_{k-1} \times 2 \times r_k}$ by concatenating the blocks
\begin{equation*}
\unocc{X}_{k,n} = \begin{pmatrix}
\unocc{X_1}_{k,n} & 0 & \cdots & 0 \\
0 & \unocc{X_2}_{k,n} & \cdots & 0 \\
\vdots & \vdots & \ddots & \vdots \\
0 & 0 & \cdots & \unocc{X_D}_{k,n}
\end{pmatrix} \in \R^{\rho_{k-1,n} \times \rho_{k,n}}
\end{equation*} 
for $n \in \cK_{k-1} \cap \cK_{k}$ and where $\rho_{k,n} = \sum_{i = 1}^D \rho_{i,k,n}$ is the sum of the block sizes. The occupied blocks $\occ{X}_{k,n}$ are defined analogously for $n \in \cK_{k-1} \cap (\cK_{k}-1)$. Note that for $k = 1$ and $k = K$, the blocks are stacked horizontally and vertically, respectively.
Let the additional block $\check X$ that carries the eigenvector-specific information be at position $p$ (that is, before component $X_p$). Then it can be chosen to have block structure
\[
  \check X = \begin{pmatrix}
\neutlift{X}_{\hspace*{-.1cm}0} & 0 & \cdots & 0 \\
0 & \neutlift{X}_{\hspace*{-.1cm}1} & \cdots & 0 \\
\vdots & \vdots & \ddots & \vdots \\
0 & 0 & \cdots & \neutlift{X}_{\hspace*{-.1cm}N}
\end{pmatrix}  
\]
where 
\begin{equation*}
\neutlift{X}_{\hspace*{-.1cm}n} = \begin{matrix}
& & \hspace*{-10pt} \begin{pmatrix}
\hspace*{-.1cm} \neut{X}{D}_{\hspace*{-.2cm} n}
\end{pmatrix} \\[-8pt]
& \hspace*{-10pt} \iddots & \\[-8pt]
\begin{pmatrix} \hspace*{-.1cm} \neut{X}{1}_{\hspace*{-.2cm} n} \end{pmatrix} & & 
\end{matrix} \in \R^{\rho_{p,n} \times D \times \rho_{p+1,n}}, \quad n\in \cK_{p-1}
\end{equation*} 
stores the eigenvector-specific blocks
\begin{equation}\label{eq:evspecific}
\neut{X}{i}_{\hspace*{-.15cm} n} \in \R^{\rho_{p,n} \times \rho_{p+1,n}}, \quad i=1,\ldots,D, n\in \cK_{p-1}.
\end{equation}

\begin{example}
A collection of $D$ tensors $\ten x_1, \ldots, \ten x_D \in \cF^4_2$ of order $K=4$ and particle number $N=2$ can be represented in joint block format as
\begin{equation*}
\setlength\arraycolsep{3pt}
\ten x \!=\! \begin{bmatrix}
\unocc{X}_{1,0}^{\uparrow} & \occ{X}_{1,0}^{\uparrow}
\end{bmatrix}
\!\SKP\!
\begin{bmatrix}
\unocc{X}_{2,0}^{\uparrow} & \occ{X}_{2,0}^{\uparrow} & 0 \\[3pt]
0 & \unocc{X}_{2,1}^{\uparrow} & \occ{X}_{2,1}^{\uparrow}
\end{bmatrix}
\!\SKP\!
\begin{bmatrix}
\neutlift{X}_{\hspace*{-.1cm}0}  & 0 & 0 \\[3pt]
0 & \neutlift{X}_{\hspace*{-.1cm}1} & 0 \\[3pt]
0 & 0 & \neutlift{X}_{\hspace*{-.1cm}2} \\
\end{bmatrix}
\!\SKP\!
\begin{bmatrix}
\occ{X}_{3,0}^{\uparrow} & 0 \\[3pt]
\unocc{X}_{3,1}^{\uparrow} & \occ{X}_{3,1}^{\uparrow} \\[3pt]
0 & \unocc{X}_{3,2}^{\uparrow}
\end{bmatrix}
\!\SKP\!
\begin{bmatrix}
\occ{X}_{4,1}^{\uparrow} \\[3pt]
\unocc{X}_{4,2}^{\uparrow} 
\end{bmatrix}.
\end{equation*}
\end{example}
All operations mentioned in~\cite{BGP:22}, including the application of matrix free-operators, can be performed also in this joint block format. In particular, the rank truncation and orthogonalization procedures can be carried out in an analogous way. Using this orthogonalization, we require the components $1,\ldots,p-1$ to be {\em left-orthogonal} and the components $p,\ldots,K$ to be {\em right-orthogonal}. Then the norm information is stored in the eigenvector-specific blocks as in \eqref{eq:evspecific} in $\check X$, where
\begin{equation*}
\norm{ \ten x_i }^2 = \sum_{n \in \cK_{p-1}} \tr\begin{pmatrix} \hspace*{-.1cm} \neut{X}{i}_{\hspace*{-.15cm} n}^{\hspace*{-.15cm} \trp} \hspace*{-.1cm}  \neut{X}{i}_{\hspace*{-.15cm} n}\end{pmatrix}.
\end{equation*}

In order to maintain orthonormality between eigenvectors, we only need to ensure orthonormality of the blocks in \eqref{eq:evspecific}. 
In order to perform the steps described in \cref{ssc:stepsize}, we apply this orthogonalization procedure on the joint block representation of $(\ten X^n, \ten R^n) = (\ten x_1^n, \ldots, \ten x_D^n, \ten r_1^n,\ldots, \ten r_D^n)$. This yields the rescalings $\Xi_{\ten x}$ and $\Xi_{\ten r}$ in \cref{eq:subspacerescale} containing the norms of the approximate eigenvectors and their residuals respectively. Finally, after solving the small generalized eigenvalue problem required by Algorithm~\ref{alg:jointinnerpinvit}, the columns of $(\ten X^n, \ten R^n)$ are transformed by the matrix $W \in \R^{2D \times D}$ in order to yield the updates
$\ten x_1^{n+1},\ldots,\ten x_D^{n+1}$. This transformation can be applied directly to the blocks in $\check X$ of $(\ten X^n, \ten R^n)$. 

\begin{algorithm}[t!]
\DontPrintSemicolon
\SetKwFunction{pinvit}{PINVIT}
\SetKwInOut{Input}{input}\SetKwInOut{Output}{output}
\SetKwProg{KwProg}{function}{}{end}
\SetKwInput{KwRequires}{further requires}
\Input{joint tensor representation of $\ten{X}^0 := (\ten{x}^0_1,\ldots,\ten{x}^0_D) \in \mathcal{V}^K_N \otimes \R^D $, tolerance $\hat \tau_{x_1}$}
\KwRequires{tolerance $\enum \in (0, 1-\eiter)$, tolerance $\teig \in (0,1)$, bound $c$ satisfying \cref{specequivA}, bound $\delta \leq 1 - \frac{\lambda_1}{\lambda_2}$}
\Output{approximations $\{(\lambda_i,\ten{x}^n_i)\}_{i = 1}^D$ to generalized eigenpairs $\{(\lambda_i,\ten{u}_i)\}_{i=1}^D$ of $(\ten A, \ten E)$ with $\| \ten u_1 - \ten x^n_1 \|_{\A} \leq \hat \tau_{x_1}$ and upper bound $\{\rho^n_i\}_{i=1}^D$ to residuals} 
\BlankLine
\KwProg{\pinvit{$\ten X^0, \hat \tau_{x_1}$}}{
$\zetares \gets (1+c)^{-1} \enum$\tcp*[r]{\normalfont{} \Cref{lem:truncr}}
$\etares \gets \infty$
\BlankLine
\For{$n = 0,1,\ldots$}{
    \tcp{(exactly) normalize and evaluate Rayleigh quotients}
    solve $({\ten{X}^n})^\trp \A \ten{X^n} V = ({\ten{X}^n})^\trp \E \ten{X}^n V \Lambda^n$ for $V \in \R^{D \times D}$\;
    $\ten{X}^n \gets \ten{X}^n V$\tcp*[r]{\normalfont{} such that $\ten{X}^n$ has $\A$-orthonormal columns and $\Lambda(\ten{X}^n) = \Lambda^n$} 
    \BlankLine
    \tcp{determine approximate residual and bound}
    $\etares \gets \min(2 \etares,  \frac{4}{5} \zetares \sum_{k,k'} \| \ten S_k \ten H_\gamma \ten S_{k'} \ten{X}^n -  \ten S_k \ten S_{k'} \ten{X}^n \Lambda^n\|)$\tcp*[r]{\normalfont{} \Cref{startetar,respara}} 
    \While {$\min_{i = 1,\ldots,D} \|\operatorname{res}_{\etares}(\ten{x}^n)_i\| < (1+\zetares^{-1}) \etares$}{
        $\etares \gets \frac{4}{5} \etares$\tcp*[r]{\normalfont{} \Cref{truncr} }
        }
    $\ten R^n \gets \operatorname{res}_{\etares}(\ten{x}^n)$\;
    $\rho^n_i \gets (1-c)^{-1/2} \frac{\|\ten{r}^n_i\| + \etares}{\|\ten{x}^n_i\|_{\A}}$, $i = 1,\ldots,D$ \tcp*[r]{\normalfont{} \Cref{rhobound}} 
    \BlankLine
    \tcp{bound eigenvector residual and check termination criterion}
    $\tau_{x_1} \gets 2 \sin \left(\frac{1}{2} \arcsin \bigl( \sqrt{ \delta  \bfun_\delta( \rho^n_1) } \, \rho^n_1 \bigr) \right)$ \tcp*[r]{\normalfont{} \Cref{aposterbound}}
    \If{$\tau_{x_1} \leq \hat \tau_{x_1}$}{
      \Return{$(\Lambda(\ten{X}^n), \ten{X}^n, \{\rho^n_i\}_{i=1}^D)$}
    }
    \BlankLine
    \tcp{determine Ritz vectors}
    solve $ \ten{\mathcal{X}}^\trp \A  \ten{\mathcal{X}} V = \ten{\mathcal{X}}^\trp \E  \ten{\mathcal{X}} V \Lambda_\ast$ for $ \ten{\mathcal{X}} = (\ten{X}^n\ \ten{R}^n) \in \mathcal{F}_N^K \otimes \R^{2D}$\; 
   $\ten X_\ast \gets \ten{\mathcal{X}} V$ \tcp*[r]{\normalfont{} $\ten{X_\ast}$ has $\A$-orthonormal columns and $\Lambda(\ten{X}_\ast) = \Lambda_\ast$} 
    \BlankLine
    \tcp{truncate iterate}
    $\etain \gets (1+c)^{-1/2}  \cdot \min_{i = 1,\ldots,D}  \|\ten X_{\ast i}\|_{\A} \Delta_{\ten A,i} $\\ $\quad$ where $\sqrt{1-\Delta_{\ten A,i}^2} - \rho^n_i \Delta_{\ten A,i} = q_{\lambda,i}^{-1/2}$, $q_{\lambda,i} = \frac{\lambda_{\ast i} + \teig (\lambda_i^n - \lambda_{\ast i})}{\lambda_{\ast i}}$\;
    ${\ten X^{n+1}} \gets \operatorname{trunc}_{\etain}(\ten X_\ast)$\tcp*[r]{\normalfont{} \Cref{trunciter}} 
}  
\Return{$(\Lambda(\ten{X}^n), \ten{X}^n)$}
}
\caption{Perturbed inverse (inner) subspace iteration}
\label{alg:jointinnerpinvit}
\end{algorithm}%

\section{Numerical Experiments}\label{sec:numer}

We consider mainly one setting with one-dimensional Coulomb-like single particle potentials with nonsmooth interaction in the basic form \eqref{eq:hamil}, as described in detail in \cref{sec:exp10}, but the approach described in the following is suitable for general cases. Results of the experiments are discussed in \cref{sec:expobs}.

\subsection{Preliminaries}

In this section, we first derive suitable orbitals for the considered class of problems, discuss the evaluation of the two-particle interaction operators and comment on the choice of algorithmic parameters.

\subsubsection{Approximate eigenfunctions as orbital functions}\label{sec:approxeigen}
To derive suitable (real-valued) orbital functions $\{\phi_i\}_{i = 1}^K$ on $\Omega = [-b,b] \subset \R$, we approximate the eigenfunctions corresponding to the lowest eigenvalues of the one-particle interaction on a larger basis,
$$\ten T^{\psi}: \mathcal{F}_N^B \rightarrow \mathcal{F}_N^B, \quad B \gg K.$$ %
We therefore apply an $L_2$-orthonormal basis $\{\psi_j\}_{j = 1}^B$ of piecewise degree $6$ polynomial, $C^1(\Omega)$-multiwavelets \cite{DGH:99} under Dirichlet boundary conditions on $[-b,b]$, refined around $0$.\footnote{A wavelet $w$ is included in this basis if and only if $\supp(w) \cap (-b_w,b_w) \neq \emptyset$, $b_w := 2^{m_{\min}-\mathrm{level}(w)} b$, as well as $\mathrm{level}(w) \leq m_{\max}$.}
The corresponding $(B \times B)$-matrix $T^\psi = (t^{\psi}_{ij})_{ij}$ with
\begin{align*}
 t^{\psi}_{ij} := \int_\Omega \psi_i(x)\big( - \textstyle \frac12 \displaystyle \Delta + V \big) \psi_j(x) \ \mathrm{d}x &= \frac{1}{2} \int_\Omega \psi_i'(x) \psi_j'(x) \ \mathrm{d}x + \int_\Omega \psi_i(x)\ V(x)\ \psi_j(x) \ \mathrm{d}x
\end{align*}
is computed via sufficiently accurate quadrature rules, each adapted to the intersection of supports of pairs $(\psi_i,\psi_j)$ for $i,j = 1,\ldots,B$. The first $K$ (discrete, orthonormal) eigenvectors $\{q_i\}_{i = 1}^K$ with eigenvalues $\{\omega_i\}_{i = 1}^K$ of $T^\psi$ are then used to construct the orbitals
\begin{align*}
 \phi_i(x) := \sum_{j = 1}^B q_{ij} \psi_j(x), \quad T^\psi \phi_i = \omega_i \phi_i, \ i = 1,\ldots,K.
\end{align*}
Consequently, these yield the first $K$ approximate, $L_2$-orthonormal eigenfunctions of $\ten T$. 
For $\ten T$, we thus obtain the diagonal coefficient matrix
\begin{align*}
 (t_{ij})_{i,j = 1}^K = \mathrm{diag}(\omega_1,\ldots,\omega_d).
\end{align*}
The coefficient matrix for the shifted operator $\ten T + \gamma I$ is accordingly \begin{align*}
   \mathrm{diag}(\omega_1,\ldots,\omega_K) + \frac{\gamma}{N} I_K.                                                                                        
\end{align*}
Due to the exponentially decreasing support sizes of higher-level wavelets, the matrix $(t^\psi_{ij})_{ij}$ is sparse, whereas the eigenvectors can be computed by a suitable solver based on inverse iteration. The above process may be improved by combining the single, consecutive steps into a basis-adaptive eigenfunctions solver, but the present form suffices for the subsequent experiments.
For the preconditioner, similarly to \cite{MG22}, we approximate the diagonal of $\ten H + \gamma \ten I$ in $\mathcal{F}_N^K$ with the auxiliary $\ten D = \sum_{i = 1}^K \theta_{i}\, \ten a_i^\ast \ten a_i$ as in \cref{sec:expfsq} via
\begin{equation*}
 \theta_i := \frac{\gamma}{N} + t_{ii} + \sum_{j = k_i + 1}^{k_i + N} (v_{ijji} - v_{ijij}), \quad k_i := \max(0,i-N).
\end{equation*}

\subsubsection{Evaluation of the two-particle interaction operator}\label{sec:evaltwopart}
In principle, every entry
\begin{equation*}
  v_{ijk\ell}^\psi := \int_\Omega \int_\Omega \psi_i(x) \psi_j(x') \ V_2(x,x') \ \psi_k(x') \psi_\ell(x) \sdd x \sdd x' 
\end{equation*}
could be computed via an adaptive quadrature rule as described for $t^\psi$ in \cref{sec:approxeigen}, where $V_2$ describes the two-particle interaction potential. However, as the number of non-zero entries of $v^\psi$ equals the squared number of those of $t^\psi$, it becomes impractical to proceed the required basis reduction to $\{q_i\}_{i = 1}^K$, even with sparse arithmetics. Instead, in order to evaluate the coefficient tensor $(v_{ijk\ell})_{i,j,k,\ell = 1}^K$, the same summed Gaussian quadrature rule 
$$\int_{\Omega} f(x) \mathrm{d}x\ \approx\ \sum_{u = 1}^U w_u f(x_u), \quad w_u > 0,\ x_u \in \Omega,\ u = 1,\ldots,U$$ of degree $n \geq 7$ is applied to all required integrals, where the quadrature is exact for piecewise polynomials up to degree $2n-1$. This includes $f(x) = \psi_i(x) \psi_j(x)$ for $i,j = 1,\ldots,B$ and $f(x) = \phi_i(x) \phi_j(x)$ for $i,j = 1,\ldots,K$.\footnote{The quadrature intervals are chosen such that the length of any interval $I$ equals $2b\, 2^{-(m(I)+3)}$, where
 $m(I) = \max \{ \mathrm{level}(\phi_j) \colon \supp(\phi_j) \cap I \setminus \delta I \neq \emptyset, \ j = 1,\ldots,B \}$. The values $m(I)$ and intervals lengths are well-defined due to properties of the wavelets. The additional factor $2b$ is due to the rescaling of wavelets originally defined on $[0,1]$ to $\Omega = [-b,b]$.}
In this manner, we can efficiently compute
\begin{align*}
 v_{ijk\ell} & = \int_\Omega \int_\Omega \phi_i(x) \phi_j(x') \ V_2(x,x') \ \phi_k(x') \phi_\ell(x) \sdd x \sdd x' \\
 & \approx\ \sum_{s = 1}^U \sum_{t = 1}^U w_s \phi_i(x_s) \phi_j(x_t) V_2(x_s,x_t) w_t \phi_k(x_t) \phi_\ell(x_s) \\
 &  =  \sum_{s = 1}^U \sum_{t = 1}^U \big( \phi_i(x_s) \phi_\ell(x_s) \big)
  \big( w_s V_2(x_s,x_t) w_t \big) \big( \phi_j(x_t) \phi_k(x_t) \big). 
  \end{align*}

The wavelets are first evaluated on all quadrature points, the result of which is then transformed with the basis $\{q_i\}_{i = 1}^K$. This evaluation only requires objects of (the usually much smaller) sizes $U \times U, U \times B$ and $U \times K^2$.

\subsubsection{Algorithmic parameters}\label{sec:extrapol}
In all experiments, we use the parameter choices $\delta = 0.1$ and $\delta_0 = \eta = \delta/2$ concerning the exponential sums for the preconditioner (cf.~\cref{sec:expsum}) as well as 
$\enum := \frac{1}{2} (1 - \eiter)$ concerning the truncation of the residual (cf.~\cref{truncr}), such that $\eiter + \enum \in (\frac{1}{2},1)$. Further, we set $\teig = \frac{1}{2}$ concerning the truncation of the iterate (cf.~\cref{trunciter}). The position $p$ of the additional core in the joint tensor parametrization for $D > 1$ (cf.~\cref{sec:jointMPS}) is determined by the initial iterate $\ten Y^0$ as described in \cref{sec:initvalues}.

With the construction as in \cref{sec:approxeigen,sec:evaltwopart}, the discretization of the infinite dimensional problem to find the lowest eigenvalue(s) of $\ten H + \gamma \ten I$ depends on the shift $\gamma \in \R$, the minimal and maximal level of wavelets $m_{\min},m_{\max} \in \N$ and the number of orbitals $K \in \N$. The quadrature rule can be assumed to be  sufficiently accurate such that its error is negligible. A suitable shift $\gamma$ is chosen based on observations for $K = 14$ concerning resulting values $\lambda_1$, $\delta$ and $c$, on which we comment further below. Values for $m_{\min}$, $m_{\max}$ and $K$ are in turn chosen as large as computational resources allow.

The two main parameters ultimately required by the algorithm are a value $c \in (0,1)$ bounding the relation between the $\A$- and $\ell_2$-norms as in \cref{specequivA} and a value $\delta > 1$ bounding the relation between the two lowest eigenvalues of $\ten H_\gamma$ as in \cref{sec:aposerrcon}. For larger numbers of orbitals $K > 14$, we heuristically extrapolate such bounds from discretized problems that sufficiently small to be handled by conventional matrix algorithms, which for\footnote{using the same number of preconditioner summands $M^-$ to $M^+$ for $\ten S_{K_{\mathrm{low}}}$ as for $K$} $K_{\mathrm{low}} \in \{12,14\}$ yield precise bounds
\begin{align*}
  C_{K_{\mathrm{low}},\mathrm{lower}} \langle   \ten x,  \ten x \rangle
   \leq \langle \A_{K_{\mathrm{low}}}  \ten x , \ten x \rangle 	\leq 
   C_{{K_{\mathrm{low}}},\mathrm{upper}} \langle  \ten x,  \ten x \rangle, \quad \forall \ten x \in \mathcal{F}_N^{K_{\mathrm{low}}},
\end{align*}
as well as $\delta_{K_{\mathrm{low}}} = \frac{\lambda_{{K_{\mathrm{low}}},2}}{\lambda_{{K_{\mathrm{low}}},2} - \lambda_{{K_{\mathrm{low}}},1}}$. We estimate $c$ via \cref{rhockappa} from the extrapolated values
\begin{align*}
 C_{K,\mathrm{lower}} = C_{14,\mathrm{lower}}^{1+k}  C_{12,\mathrm{lower}}^{-k}, \quad
 C_{K,\mathrm{upper}} = C_{14,\mathrm{upper}}^{1+k} C_{12,\mathrm{upper}}^{-k}, \quad k = \frac{K - 14}{2}.
\end{align*}
Similarly, we set $\delta := \delta_K := \max(\delta_{14},\delta_{14}^{1+k}  \delta_{12}^{-k})$. This extrapolation itself is generally not bounded with respect to $K$, but in the following experiments, $C_{K}$ and $\delta_K$ are sufficiently close to $C_{14}$ and $\delta_{14}$.

\subsubsection{Initial values}\label{sec:initvalues}

Due to the diagonalization of the one-particle interaction, the first unit vector $e_1$ within $\mathcal{F}_N^K$ is an exact eigenvector of $\ten T$ with lowest eigenvalue $\tilde{\lambda}_1 = \gamma + \sum_{i = 1}^N t_{ii}$. It therefore yields a canonical, rank-one initial value $\ten y^0$ for \cref{alg:outerpinvit} that does not account for the two-particle interaction. In fact, as tested for the following experiments, other initial values cause the first truncated iterate $\ten y^1$ to be very close to $e_1$, but it can then take significant effort to arrive at this first outer truncation. More generally,  we use the first $D$ unit vectors $\sum_{i = 1}^D e_i$ in $\mathcal{F}_N^K$ as initial value $\ten Y^0$ for the simultaneous approximation of multiple eigenvectors. While directly bounded by $D$, the true ranks of $\ten Y^0$ are slightly lower. Moreover, we choose the position $p$ of the joint core as described in \cref{sec:jointMPS} so that the ranks of $\ten Y^0$ are minimized.

\subsection{One-dimensional Coulomb-like single particle potentials with nonsmooth interaction}\label{sec:exp10}
We consider the Hamiltonian
\begin{equation*}
   H = - \frac12 \Delta + \sum_{i=1}^N V(x_i) + \frac12 \sum_{i \neq j} V_2(x_i,x_j), 
\end{equation*}
with one- and two-particle potentials given by
\begin{align*}
 V(x) := -N \delta_0(x), \quad V_2(x,y) := \exponeoneKlowDonecfV \exp(-|x - y|),
\end{align*}
for $N = \exponeoneKlowDoneN$ particles on the interval $\Omega = [-b,b]$, $b = \exponeoneKlowDoneb$. We choose the shift $\gamma = \exponeoneKlowDonegamma$. Additional experiments suggest that $(m_{\min},m_{\max}) = (\exponesixKevenhigherDonemmin,\exponesixKevenhigherDonemmax)$ and $K = \exponesixKevenhigherDoneK$ suffice to bound the discritization error by roughly $10^{-4}$ to $10^{-6}$. Nonetheless, for demonstrational purposes, we set a termination tolerance (cf. \cref{alg:outerpinvit}) of $\tau = \exponesixKevenhigherDonereltol$ for $D = \exponesixKevenhigherDoneD$ and $\tau = \expfoursixKevenhigherDfourreltol$ for $D = \expfoursixKevenhigherDfourD$.

In \cref{sec:exp10orbitals}, we first consider the applied orbital functions (cf. \cref{sec:evaltwopart}). In \cref{sec:exp11_KlowD1}, we investigate the convergence properties of the eigenvalue algorithm for a reduced number of orbitals, $K = \exponeoneKlowDoneK$. 
\Cref{sec:exp16_KevenhigherD1} then shows the results for $K = \exponesixKevenhigherDoneK$. For this case, with $D = 1$, we also compare with a version that only runs the inner iteration. In \cref{sec:exp1D3}, we apply the joint block MPS variant for the simultaneous approximation of $D = \expfouroneKlowDfourD$ eigenvalues. For simplicity, $n$ in these results counts the total number of inner iterations and is not reset after each outer step. In the following, a number continued by $\mathtt{(\ldots)}$ indicates that the shown beginning of the value is the result of rounding.

\subsubsection{Orbital functions and resulting operators}\label{sec:exp10orbitals}
A minimal to maximal level of wavelets $(m_{\min},m_{\max}) = (\exponeoneKlowDonemmin,\exponeoneKlowDonemmax)$ yields $B = \exponeoneKlowDoneB$ basis functions. The first $\exponeoneKlowDoneK$ resulting orbitals (cf. \cref{sec:approxeigen}) are displayed in \cref{fig:exp11_KlowD1_wavelets}. 
Whereas the MPO ranks of the resulting operator $\ten T$ are bounded by $4$ due to the diagonalization, the ranks for the resulting operator $\ten V$, for $K = \exponesixKevenhigherDoneK$, are
\begin{align*}
 (1,4, 16, 39, 78, 99, 120, 145, 174, 207, 244, 285, 330, 379, 330, \ldots, 4,1).
\end{align*}
In principle, rounding small entries in the underlying tensor $V$ to zero may decrease the rank, but within sensible tolerances, we observe a negligible rank reduction.

\tikzfigure{htb}{1}{exp11_KlowD1_wavelets}{Applied orbitals as layed out in \cref{sec:approxeigen} for the setting described in \cref{sec:exp10}.}

\subsubsection{Comparison with exact results for few orbitals}\label{sec:exp11_KlowD1}

The preconditioner according to \cref{sec:expfsq} for $K = \exponeoneKlowDoneK$ with $\exponeoneKlowDonelengthS$ summands is based on the auxiliary $\ten D$ with spectrum in $[t_{\min},t_{\max}] = [\exponeoneKlowDonetmin(\ldots),\exponeoneKlowDonetmax(\ldots)]$, hence with quotient $T = t_{\max} / t_{\min} =  \exponeoneKlowDoneT(\ldots)$.
The algorithmic parameters are here exactly determined as $\delta = \exponeoneKlowDonekappa(\ldots)$ and $c = \eiter = \exponeoneKlowDoneepsiter(\ldots)$ as well as consequently $\enum = \exponeoneKlowDoneepsnum(\ldots)$. 
The exact lowest eigenvalue of the discretized problem setting for $\A, \E \in \R^{2^K \times 2^K}$ is 
$
 \lambda_{1} = \exponeoneKlowDonelambdaonetrue
$.
\Cref{fig:exp11_KlowD1_bounds} shows the convergence of \cref{alg:outerpinvit} depending on $n \in \N$ as well as the comparison with the bounds of the relative error determined during the algorithm, which only uses $\delta$ and bounds $\rho^n$ of the residuals $\rho(\ten x^n)$ as described by \cref{rhobound,lambdaqbound}. 
\Cref{fig:exp11_KlowD1_ranks} shows the convergence of \cref{alg:outerpinvit} in relation to the maximal ranks of the iterates.

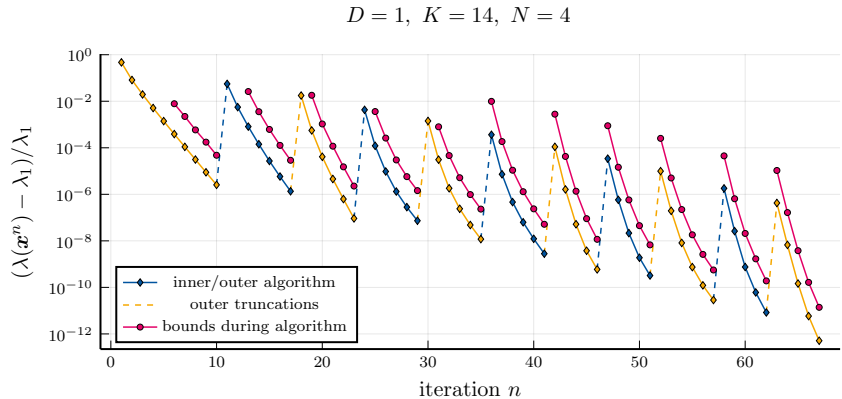
\tikzfigure{htb}{1}{exp11_KlowD1_bounds}{Exact eigenvalue error as well as the bound determined during the algorithm (cf. \cref{rhobound,lambdaqbound}). Bounds in iterations in which  \cref{lambdaqassumptions} is violated have been omitted.}
\tikzfigure{htb}{1}{exp11_KlowD1_ranks}{Maximal ranks of iterates in relation to convergence of \cref{alg:outerpinvit} as well as the variant with only inner iterations.}

\subsubsection{High accuracy convergence results}\label{sec:exp16_KevenhigherD1}

The preconditioner for $K = \exponesixKevenhigherDoneK$ with $\exponesixKevenhigherDonelengthS$ summands is based on the auxiliary $\ten D$ with  $[t_{\min},t_{\max}] = [\exponesixKevenhigherDonetmin(\ldots),\exponesixKevenhigherDonetmax(\ldots)]$, hence with quotient $T = t_{\max} / t_{\min} = \exponesixKevenhigherDoneT(\ldots)$. The algorithmic parameters are here extrapolated following \cref{sec:extrapol} as $\delta = \exponesixKevenhigherDonekappa(\ldots)$ and $c = \eiter = \exponesixKevenhigherDoneepsiter(\ldots)$ as well as consequently $\enum = \exponesixKevenhigherDoneepsnum(\ldots)$. 
The eigenvalues returned by the algorithms are
\begin{equation*}
 \lambda_1^{\mathrm{inner/outer}} = \exponesixKevenhigherDonelambdaoneIO, \quad \lambda_1^{\mathrm{inner}} = \exponesixKevenhigherDonelambdaoneIonly.
\end{equation*}
\Cref{fig:exp16_KevenhigherD1_ranks} shows the convergence of \cref{alg:outerpinvit} (where $n_{\max}$ is the number of iterations) in relation to the maximal ranks of the iterates. \Cref{fig:exp16_KevenhigherD1_ranksres} shows both the maximal ranks of the iterates and of the residuals in relation to the iteration number. \Cref{fig:exp16_KevenhigherD1_ranksheatmap,fig:exp16_KevenhigherD1_ranksheatmap_plain} visualize all entries of each rank vector in relation to the iteration number.

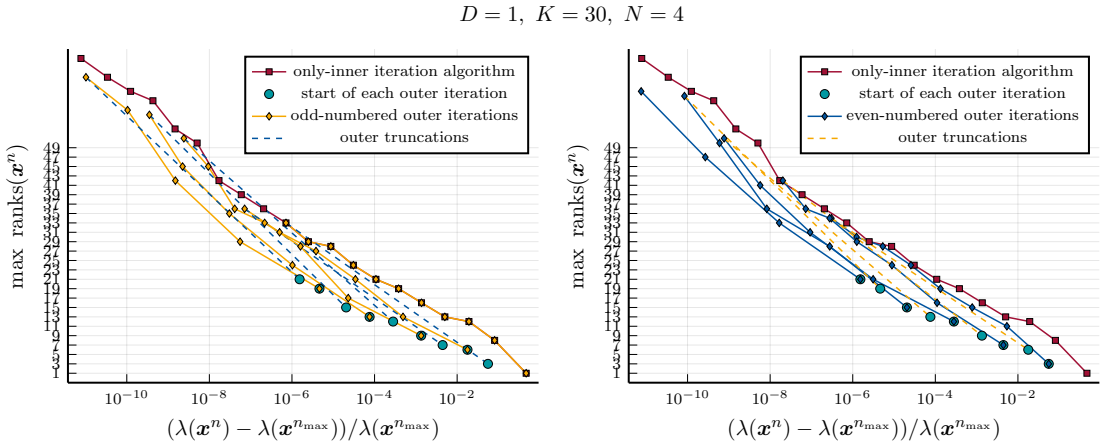
\tikzfigure{htb}{1}{exp16_KevenhigherD1_ranks}{Maximal ranks of iterates in relation to convergence of \cref{alg:outerpinvit} as well as the variant with only inner iterations, using the last determined Rayleigh quotient as reference quantity.}
\tikzfigure{htb}{1}{exp16_KevenhigherD1_ranksres}{Maximal ranks of iterates (left) and residuals (right) in relation to the iteration number. }

\tikzfigure{h}{1}{exp16_KevenhigherD1_ranksheatmap}{Ranks of iterates by iteration for the inner/outer iteration. The values next to the dashed lines show each the relative eigenvalue error (as specified in \cref{fig:exp16_KevenhigherD1_ranks}) after an outer truncation.}
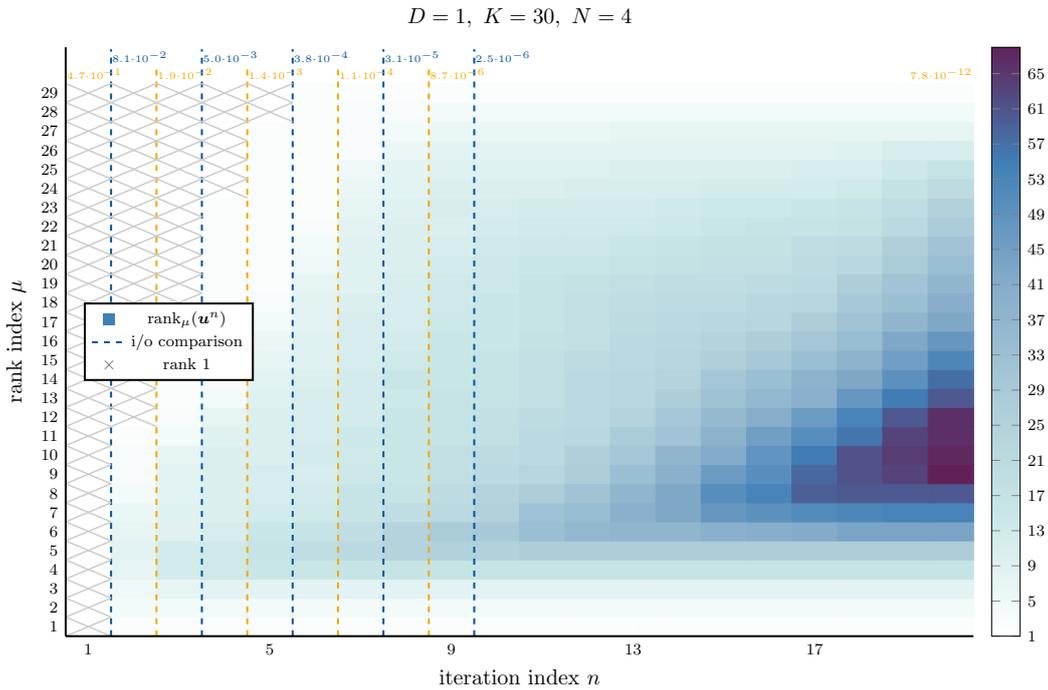
\tikzfigure{h}{1}{exp16_KevenhigherD1_ranksheatmap_plain}{Ranks of iterates by iteration. The dashed lines instead indicate each the last iteration with a larger relative eigenvalue compared to such after outer truncations of the inner/outer iteration \cref{fig:exp16_KevenhigherD1_ranksheatmap}.}

\subsubsection{Simultaneous approximation of multiple eigenvalues}\label{sec:exp1D3}
The same experiments as in \cref{sec:exp11_KlowD1,sec:exp16_KevenhigherD1} are performed for the simultanous approximation of the first $D = \expfouroneKlowDfourD$ eigenvalues with the joint inner iteration \cref{alg:jointinnerpinvit}. An additional, heuristic termination criterion is used here to not only ensure the accuracy of the first Rayleigh quotient, but to also take into account the typically lower acccuracy for subsequent ones. The first $D$ exact eigenvalues for $K = \expfouroneKlowDfourK$ are 
\begin{align*}
 (\lambda_i)_{i = 1}^{\expfouroneKlowDfourD} = \expfouroneKlowDfourlambdaDtrue.
\end{align*}
The returned eigenvalues for $K = \expfoursixKevenhigherDfourK$ are 
\begin{align*}
 (\lambda_i^{\mathrm{inner/outer}})_{i = 1}^{\expfouroneKlowDfourD} = \expfoursixKevenhigherDfourlambdaDIO.
\end{align*}
\Cref{fig:exp46_KevenhigherD4_ranks} shows the convergence of \cref{alg:jointinnerpinvit} in relation to the maximal ranks of the iterates for $K = \expfouroneKlowDfourK$ (left)  and  $K = \expfoursixKevenhigherDfourK$ (right). \Cref{fig:exp46_KevenhigherD4_ranksres} shows the maximal ranks of the iterates and of the residuals in relation to the iteration number for $K = \expfoursixKevenhigherDfourK$. \Cref{fig:exp46_KevenhigherD4_ranksheatmap} visualizes all ranks, with each rank vector having one more entry due to the joint core, in dependence of the iteration number.

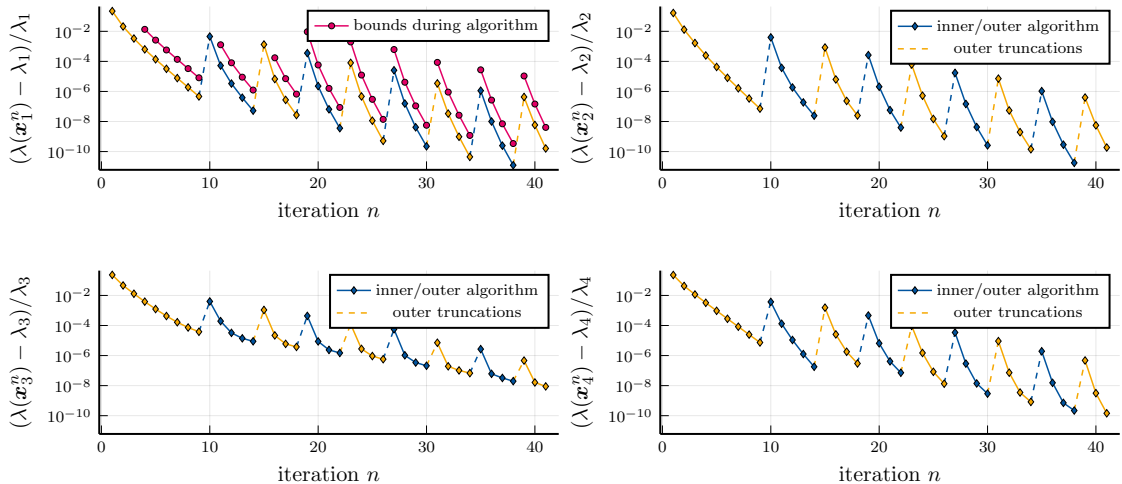
\tikzfigure{htb}{1}{exp41_KlowD4_singleerrs}{Individual convergence of the $D$ Rayleigh quotients determined by \cref{alg:jointinnerpinvit}.}
\tikzfigure{htb}{1}{exp46_KevenhigherD4_ranks}{Maximal ranks of iterates in relation to convergence of \cref{alg:jointinnerpinvit} as well as the variant with only inner iterations. For $K = \exponesixKevenhigherDoneK$, we compare with the last determined Rayleigh quotients.} %
\tikzfigure{htb}{1}{exp46_KevenhigherD4_ranksres}{Maximal ranks of iterates (left) and residuals (right) in relation to the iteration number.} 
\tikzfigure{h}{1}{exp46_KevenhigherD4_ranksheatmap}{Ranks of iterates by iteration for the inner/outer iteration. The values next to the dashed lines show each the relative eigenvalue error (as specified in \cref{fig:exp46_KevenhigherD4_ranks}) after an outer truncation.}

\subsection{Observations on numerical experiments}\label{sec:expobs}
In summary, we observe in particular a clear effect of the inner-outer iteration scheme, where the truncation in each step of the outer iteration produces approximations with substantially smaller ranks than obtained by a na\"ive use of the inner iteration for all steps. While the ranks of iterates are observed to depart from such near-best approximations during the inner iterations due to the action of operators, this rank increase is far less pronounced than the individual maximal ranks of iterates and operators would suggest. In particular, we observe favorable performance of the method and controlled convergence also for larger values of $K$, including in the case of simultaneous approximation of multiple eigenvalues in joint tensor representation.

\section{Conclusions}

We have constructed an eigenvalue solver operating on low-rank tensor representations based on inexact preconditioned inverse iteration. While the central convergence estimates, which sharpen those given in \cite{Rohwedder:11}, are generally applicable and may be of independent interest, in our construction we pay particular attention to Schr\"odinger eigenvalue problems in second-quantized form as common in quantum chemistry.
By a rigorous construction of a low-rank preconditioner for such problems, for the single-vector version of the iterative scheme, under standard assumptions on starting values we achieve systematic convergence with guaranteed error reduction in each step. Moreover, we show this method to provide approximations of near-optimal ranks.
We have also proposed a version of the scheme operating on higher-dimensional subspaces that can be combined with joint low-rank representations of several eigenvectors. While this method performs equally well in practice, its analysis will be a subject of future work.

\section*{Acknowledgements}

Co-funded by the European Union (ERC, COCOA, 101170147). Views and opinions expressed are however those of the authors only and do not necessarily reflect those of the European Union or the European Research Council. Neither the European Union nor the granting authority can be held responsible for them.

\bibliographystyle{amsplain}
\bibliography{BKPsqpinvit}

\begin{appendix}

\section{Rayleigh Quotient Bounds}

\subsection{Auxiliary result for \Cref{thm:conv}}

\begin{lemma}\label{perturbedq}
 Let $\ten x,\ten y,\ten z$ be vectors such that $\lambda(\ten z) \leq (1-t) \lambda(\ten y) + t \lambda(\ten x)$ as well as $\lambda(\ten x), \lambda(\ten y)\in [\lambda_k , \lambda_{k+1})$ and $\lambda(\ten x) \geq \lambda(\ten y)$.
 Then if for a $q \in [0,1)$,
 \begin{equation}\label{eq:perturbedeqass}
  \frac{\lambda(\ten y) - \lambda_k}{\lambda_{k+1} - \lambda(\ten y)} \leq q^2  \frac{\lambda(\ten x) - \lambda_k}{\lambda_{k+1} - \lambda(\ten x)} , 
  \end{equation}
  we also have
  \begin{equation*}
   \frac{\lambda(\ten z) - \lambda_k}{\lambda_{k+1} - \lambda(\ten z)} \leq ((1-t)q^2+t)  \frac{\lambda(\ten x) - \lambda_k}{\lambda_{k+1} - \lambda(\ten x)}.
 \end{equation*}
\end{lemma}
\begin{proof}
 By monoticity, 
 \begin{align*}
  \frac{\lambda(\ten z) - \lambda_k}{\lambda_{k+1} - \lambda(\ten z)} \leq \frac{(1-t) \lambda(\ten y) + t \lambda(\ten x) - \lambda_k}{\lambda_{k+1} - (1-t) \lambda(\ten y) -t  \lambda(\ten x)}
  = \frac{(1-t) (\lambda(\ten y) - \lambda_k) + t(\lambda(\ten x) - \lambda_k)}{(1-t)(\lambda_{k+1} - \lambda(\ten y)) + t( \lambda_{k+1} - \lambda(\ten x))}.
 \end{align*}
With $a \coloneqq  \lambda(\ten y) - \lambda_k$, $b \coloneqq \lambda(\ten x) - \lambda_k$, $c \coloneqq \lambda_{k+1} - \lambda(\ten y)$, $d\coloneqq\lambda_{k+1} - \lambda(\ten x)$ it thus suffices to show
\begin{equation}\label{eq:abcd}
   \frac{(1-t) a + t b}{(1-t) c + t d} \leq \bigl((1-t) q^2  + t\bigr)  \frac{b}{d}.
\end{equation}
By \eqref{eq:perturbedeqass}, we have $ad \leq q^2 bc$, and thus $((1-t) a + t b) d \leq (q^2 (1-t) c + t d) b$. Moreover, a simple direct calculation shows that whenever $c\geq d$, which in turn is equivalent to our assumption $\lambda(\ten x) \geq \lambda(\ten y)$, it holds that
\[
     q^2 (1-t) c + t d   \leq \bigl(q^2 (1 -t) + t \bigr) \bigl( (1 - t) c + t d \bigr)\,.
\]
 In summary, this gives \eqref{eq:abcd} and thus the assertion.
\end{proof}

\subsection{Rayleigh quotient under perturbation}\label{sec:rayqup}
To derive a lower bound for the Rayleigh quotient under perturbation, we here consider a general  linear selfadjoint operator $\ten B: V \rightarrow V$ on a Hilbert space\footnote{With our application in mind, we keep the explicit dependency of the inner product on the operator $\ten A$ but note that any inner product is admissible here} $(V,\langle \cdot, \cdot \rangle_{\ten A})$. We further assume that the Rayleigh quotient $\mu(\ten v)$ is bounded from below by some $L \in \R$ and introduce the normalized residual $\varrho(\ten v)$:
\begin{align*}
 \mu(\ten v) := \frac{\langle \ten B \ten v, \ten v \rangle_{\ten A}}{\langle \ten v, \ten v \rangle_{\ten A}} \geq L, \quad \varrho(\ten v) := \frac{\|\ten B \ten v - \mu(\ten v) \ten v\|_{\ten A}}{\|\ten v\|_{\ten A}}.
\end{align*}
 The main result derived in this section, used in \cref{trunciter}, is the following \cref{thm:pertbound}. From thereon, $\ten v$ and $\ten p$ are used to relate the unperturbed and perturbed vector, and $\ten x \in \mathcal{X}$ is used when there is only one vector.

\begin{theorem}\label{thm:pertbound}
For every $\ten v,\ten p \in V$ with angle $\angle_{\ten A}(\ten v,\ten p) \leq \alpha \in [0,\frac{\pi}{2}]$, 
\begin{equation}\label{eq:thm:pertbound}
 \frac{\mu(\ten p) - L}{\mu(\ten v) - L} \geq \left(  \max\left\{ \cos(\alpha) - \frac{\varrho(\ten v) }{\mu(\ten v) - L}\sin(\alpha) , 0 \right\} \right)^2.
\end{equation}%
\end{theorem}%

\begin{remark}
The bound is sharp when $L$ itself is sharp (and no further information about $\ten v$, $\ten p$ or $\ten B$ is provided). Assume for simplicity that $L$ is in fact attained as eigenvalue at an eigenvector $\ten v_L$.
When $\mu_\ast := \frac{\varrho(\ten v)^2}{\mu(\ten v) - L} + \mu(\ten v)$ happens to be an eigenvalue with eigenvector $\ten v_\ast$, then there is a pair $\ten{\tilde v}, \ten{\tilde p} \in \mathcal{X} := \mathrm{span}\{\ten v_\ast,\ten v_L\}$ that yields an equality in Theorem \ref{thm:pertbound}.
In general, equality is obtained if and only if the space $\mathcal{X} = \mathrm{span}\{\ten v,\ten p\}$ is $\ten B$-invariant with minimal, lower Ritz value $L$ as shown by the proof of \Cref{thm:pertbound}. If $L$ is only an infinum, the bound is still sharp as the prior holds true arbitrarily closely.
\end{remark}%
For the non-trivial angles $\alpha \leq \mathrm{arccot} \left( \frac{\varrho(\ten v)}{\mu(\ten v) - L} \right)$, we can rewrite
\begin{equation*}
  \mu(\ten v) - \mu(\ten p) \leq g_{\varrho(\ten v)}(\alpha), \quad g_{\varrho(\ten v)}(\alpha) := \varrho(\ten v) \sin(2 \alpha) + \frac{(\mu(\ten v) - L)^2 - \varrho(\ten v)^2}{\mu(\ten v) - L} \sin(\alpha)^2.
\end{equation*}
It is the smooth transition between the usual quadratic bound for eigenvectors, with thus zero residual, and the general linear one that does not use the residual. That is,
\begin{equation*}
 g_{0}(\alpha) = (\mu(\ten v) - L) \sin(\alpha)^2 \leq g_{\varrho(\ten v)}(\alpha) \leq (\max_{\ten w \in V} \mu(\ten w) - L) \sin(\alpha).
\end{equation*}
The proof of \cref{thm:pertbound} follows the same steps as well-established approaches to these extremal cases and shifts the perspective to a two-dimensional subspace $\mathcal{X} \subset V$ via
\begin{equation}\label{eq:vptoX}
 \mu(\ten p) \geq \min_{\mathcal{X} \subseteq V :\, \ten v \in \mathcal{X},\, \dim(\mathcal{X}) = 2} \  \min_{\ten{\tilde p} \in \mathcal{X} :\, \angle_{\ten A}(\ten v,\ten{\tilde p}) \leq \alpha}\ \mu(\ten{\tilde p}).
\end{equation}
Since it is much simpler to express properties of the vectors in those of a general two-dimensional subspace than it would be vice versa, relevant terms can be related to each other in form of comparitively simple equalities (cf.~\cref{lem:temple,prop:X}).
Bounding the remaining variables not assumed to be known exactly in a derived, monotone function (cf. \cref{lem:gfun}) yields the initially sought bound on the Rayleigh quotient and thus the proof for \cref{thm:pertbound}. We discuss and separate the here required, slightly more elaborated steps in advance, also as we reuse some in the subsequent section. 
For the orthogonal projection $\ten P_\mathcal{X}: V \rightarrow \mathcal{X}$, we define the projected operator
\begin{align*}
 \ten B_\mathcal{X}: \mathcal{X} \rightarrow \mathcal{X}, \quad \ten B_\mathcal{X} := \ten P_\mathcal{X} \ten B \ten P_\mathcal{X}^\trp
\end{align*}
as well as the corresponding Rayleigh quotient and the residual
\begin{align*}
 \mu_\mathcal{X}(\ten x) := \frac{\langle \ten B_\mathcal{X} \ten x, \ten x \rangle_{\ten A}}{\langle \ten x, \ten x \rangle_{\ten A}}, \quad
 \varrho_\mathcal{X}(\ten x) := \frac{\|\ten B_\mathcal{X}\ten x - \mu_\mathcal{X}(\ten x) \ten x\|_{\ten A}}{\|\ten x\|_{\ten A}}.
\end{align*}
For $\ten x \in \mathcal{X}$, that is, $\ten P_\mathcal{X}\ten x = \ten x$, the Rayleigh quotient $\mu_\mathcal{X}(\ten x) = \mu(\ten x)$ remains unchanged but for the residual we have
\begin{align*}
 \|\ten B_\mathcal{X}\ten x - \mu_\mathcal{X}(\ten x) \ten x\|_{\ten A} = \|\ten {P}_\mathcal{X} (\ten B \ten x - \mu(\ten x) \ten x)\|_{\ten A}
 \leq \|\ten B \ten x - \mu(\ten x) \ten x\|_{\ten A},
\end{align*}
and thereby $\varrho_\mathcal{X}(\ten x) \leq \varrho(\ten x)$. The two Ritz values of $\ten B$ on $\mathcal X$, defined as the two eigenvalues $a,b$ with $a \geq b > 0$ of $\ten B_{\mathcal X}$, can be related to the Rayleigh quotient and residual value. The equality is reminicent of the Temple inequality, but we restate it for completeness.
\begin{lemma}\label{lem:temple}
 Let $\mathcal{X} \subset V$ be a two-dimensional subspace with corresponding Ritz values $a \geq b$. Then $\varrho|_\mathcal{X}(\ten x)^2 = (a - \mu(\ten x))(\mu(\ten x) - b)$ for all $\ten x \in \mathcal{X}$.
\end{lemma}
\begin{proof}
One first observes that  $(\ten B_{\mathcal X} - a \ten I_{\mathcal X})(\ten B_{\mathcal X} - b \ten I_{\mathcal X}) = 0$ for the identity $\ten I_{\mathcal X} : \mathcal X \rightarrow \mathcal X$. Assuming without loss of generality $\| \ten x \|_{\ten A} = 1$, and using the symmetry of $\ten B|_\mathcal{X}$, the assertion then follows with $\varrho|_\mathcal{X}(\ten x)^2 = \langle  (\ten B|_\mathcal{X} - \mu(\ten x) I_{\mathcal X})^2 \ten x, \ten x \rangle_{\ten A}
   = \langle  \ten B|_\mathcal{X}^2 \ten x, \ten x \rangle_{\ten A} - \mu(\ten x)^2$ and $0=
   \langle  (\ten B|_\mathcal{X} - a I_{\mathcal X})(\ten B|_\mathcal{X} - b I_{\mathcal X}) \ten x, \ten x \rangle_{\ten A} = \langle  \ten B|_\mathcal{X}^2 \ten x , \ten x \rangle_{\ten A} - (a + b) \mu(\ten x) + a b
  $.
  \end{proof}
Within the two dimensional subspace, the difference between Rayleigh quotients can be stated exactly in dependence of the angle and Ritz values.
  \begin{lemma}\label{prop:X}
Let $\mathcal{X} \subset V$ be a two-dimensional subspace with corresponding Ritz values $a > b$ and let $\ten v, \ten p \in \mathcal{X}$ with $\kappa := \mu(\ten v) > b$ and $\beta := \angle_{\ten A}(\ten v, \ten p)$. Then
   \begin{equation}\label{eq:prop:X}
     \kappa - \mu(\ten p) = \kappa - b - \bigl(\sqrt{\kappa-b} \cos \beta- \sqrt{a-\kappa} \sin\beta \bigr)^2.
\end{equation}
\end{lemma}%
The right-hand side attains its first maximum $\kappa - b$ at
\begin{equation}
 \beta = \arccot \left( \frac{\sqrt{a-\kappa}}{\sqrt{\kappa - b}} \right) \in \Bigl[0, \frac{\pi}{2}\Bigr].
\end{equation}
Ultimately, if $\beta$ is bounded by a value $\alpha$ equal or larger than this limit, only a trivial statement about the Rayleigh quotient can be made. In \cref{eq:thm:pertbound}, this is reflected by the maximum.
\begin{proof}

 Let  $\ten e_1, \ten e_2 \in \mathcal{X}$ be the two orthonormal eigenvectors corresponding to $a$ and $b$ of $\ten B|_\mathcal{X}: \mathcal{X} \rightarrow \mathcal{X}$. For a representation $\ten v = a_1 \ten e_1 + a_2 \ten e_2$ and $\ten p = a^\ast_1 \ten e_1 + a^\ast_2 \ten e_2$, where w.l.o.g. $\|\ten v\| = \|\ten p\| = 1$ and $a_1,a_2 \geq 0$, we can express $\ten p$ via a rotation matrix as
\begin{align*}
 \begin{pmatrix}
  a^\ast_1 \\ a^\ast_2
 \end{pmatrix}
 =
 \begin{pmatrix}
  \cos \beta & -\sin \beta \\
  \sin \beta & \cos \beta
 \end{pmatrix}
  \begin{pmatrix}
  a_1 \\ a_2
 \end{pmatrix}, \quad \beta = \angle_{\ten A}(\ten v,\ten p) \leq \alpha.
\end{align*}
Thus $a_1^\ast = a_1 \cos \beta - a_2 \sin \beta$.
With 
\begin{align}\label{eq:a12ast}
    a_1^\ast = \frac{\sqrt{\mu(\ten p)-b}}{\sqrt{a - b}}, \ a_1 = \frac{\sqrt{\kappa-b}}{\sqrt{a - b}}, \ a_2 = \frac{\sqrt{a -\kappa}}{\sqrt{a - b}},
\end{align}
this yields the equivalent identity
\begin{align}\label{func:f}
 \mu(\ten p) - b = \bigl(\sqrt{\kappa-b} \cos \beta- \sqrt{a-\kappa} \sin\beta \bigr)^2.
\end{align}
from which \cref{eq:prop:X} is directly obtained.
\end{proof}

We prove now that the function that underlies the right-hand side of \cref{eq:prop:X} fulfills the necessary monoticity.
\begin{lemma}\label{lem:gfun}
 The function 
 \begin{align*}
 g: \R_{> 0} \times \R_{\geq 0} \times \R_{\geq 0} \rightarrow \R_{\geq 0}, \quad g(s,t,\alpha) := \begin{cases}
                  f(s,t,\alpha), & \ \alpha < \alpha^\ast(s,t), \\
                  s, & \mbox{otherwise},
                 \end{cases}
\end{align*}
for $f(s,t,\alpha) :=  s - s ( \cos(\alpha) - \frac{t}{s} \sin(\alpha) )^2$ and $\alpha^\ast(s,t) := \arccot(\frac{t}{s})$ is continuous and monotonically non-decreasing in each of its variables.
\end{lemma}
\begin{proof}
 Firstly, $\frac{\partial}{\partial \alpha} f(s,t,\alpha) \geq 0$ for $\alpha \leq \alpha^\ast(s,t)$ and $\lim_{\alpha\rightarrow \alpha^\ast(s,t)} f(s,t,\alpha) = s$. Thus, the function $g(s,t,\alpha)$ is continuous and monotonically non-decreasing in $\alpha$.
 Secondly, since further $\frac{\partial}{\partial s} f(s,t,\alpha) = \frac{s^2 + t^2}{s^2} \sin(\alpha)^2 \geq 0$ (and $\frac{\partial}{\partial s} \alpha^\ast(s,t) \geq 0$), the function $g(s,t,\alpha)$ is monotonically non-decreasing in $s$. Thirdly, it is $\frac{\partial}{\partial t} f(s,t,\alpha) \geq 0$ for $\alpha \leq \alpha^\ast(s,t)$ (and $\frac{\partial}{\partial t} \alpha^\ast(s,t) \leq 0$). Thus, the function $g(s,t,\alpha)$ is also monotonically non-decreasing in $t$. 
\end{proof}
\begin{proof}[Proof of \Cref{thm:pertbound}]
Let $f(s,t,\alpha)$ and $g(s,t,\alpha)$ be as in \cref{lem:gfun} and $\mathcal{X} = \mathrm{span}\{\ten v,\ten p\}$ with Ritz values $a \geq b$, as well as $\beta := \angle_{\ten A}(\ten v, \ten p)$ and $\kappa := \mu(\ten v)$. If $\mu(\ten v) = b$ is already minimal within $\mathcal{X}$, then the inequality trivially holds true since the righthand side is always non-negative. Otherwise, with \cref{lem:temple}, we can replace
 \begin{equation}
  a = \frac{\varrho|_\mathcal{X}(\ten v)^2}{\mu(\ten v) - b} + \mu(\ten v),
 \end{equation}
which when inserted into \cref{eq:prop:X} yields 
\begin{equation}
 \mu(\ten v) - \mu(\ten p) = f(\mu(\ten v)-b,\varrho|_\mathcal{X}(\ten v),\beta) \leq g(\mu(\ten v)-b,\varrho|_\mathcal{X}(\ten v),\beta).
\end{equation}
Due to $-b \leq -L$ and $\varrho|_\mathcal{X}(\ten v) \leq \varrho(\ten v)$, as well as $\beta \leq \alpha$, we obtain
\begin{equation}
 g(\mu(\ten v)-b,\varrho|_\mathcal{X}(\ten v),\beta) \leq g(\mu(\ten v)-L,\varrho(\ten v),\alpha).
\end{equation}
Thus, for $\displaystyle \alpha \leq \mathrm{arccot} \left( \frac{\varrho(\ten v)}{\mu(\ten v) - L} \right)$, we have
\begin{equation*}
 (\mu(\ten v) - L) - (\mu(\ten p) - L) = \mu(\ten v) - \mu(\ten p)
 \leq (\mu(\ten p) - L) ( 1 - ( \cos(\alpha) - \frac{\varrho(\ten v)}{\mu(\ten v ) - L} \sin(\alpha) )^2 ).
\end{equation*}
For larger angles, it simply follows $(\mu(\ten v) - L) - (\mu(\ten p) - L) \geq 0$. Rearranging this result then directly yields \cref{eq:thm:pertbound}.
\end{proof}

\subsection{Angle bound}\label{sec:anglebound}
Using a similar approach and same notation as in \cref{sec:rayqup}, we derive the basis for \cref{sec:aposerrcon}. 
Throughout this section, we denote the two largest eigenvalues of $\ten B$ with $\mu_1$ and $\mu_2$ and assume that the eigenspace corresponding to $\mu_1$ is one-dimensional and spanned by $\ten{v}_1 \in V$. The statements of the following \cref{sinbound,qbound} further hold as equality if $\ten x$ is in the span of the first two eigenvectors and $\gamma = \mu_2 / \mu_1$.
\begin{lemma}\label{sinbound}
Let $\gamma \in [0,1)$ for which $\gamma \mu_1 \geq \mu_2$. Then 
\begin{align*}
 \sin^2 \angle_{\ten A}(\ten{v}_1,\ten x) \leq \frac{1-\frac{\mu(\ten x)}{\mu_1}}{1-\gamma}. 
\end{align*}
In particular, if ${\mu(\ten x)} \geq q \mu_1$, $q \in [0,1]$, then
\begin{align*}
 \sin^2 \angle_{\ten A}(\ten{v}_1,\ten x) \leq \frac{1-q}{1-\gamma}\,.
\end{align*}
\end{lemma}
\begin{proof}
 Let $\mathcal{X} = \mathrm{span}\{\ten v_1, \ten x\}$. If $\mathcal{X}$ is one-dimensional, then the statements become trivial. Otherwise, denote with $a = \mu_1$ and $b < a$ its corresponding Ritz values. Let further 
$q_0 := \mu(\ten x) / a < 1$ as well as $\gamma_0 := b/a$. Then, as $a$ equals the largest eigenvalue, it must $b \leq \mu_2$ and consequently $b /a = \gamma_0 \leq \mu_2 / \mu_1 < 1$. With regard to the relations stated in \Cref{eq:a12ast} for $ \kappa := \mu(\ten x)$, we have
\begin{align}\label{eq:sinqgamma}
 \sin^2 \angle_{\ten A}(\ten{v}_1,\ten x) = a_2^2 = \frac{a - \mu(\ten x)}{a - b} = \frac{1-q_0}{1-\gamma_0}.
\end{align}
The given bounds then follow immediately due to $1 \geq q_0 \geq q$ and $\gamma_0 \leq \gamma < 1$.
\end{proof}

\begin{lemma}\label{qbound}
 Let $\gamma \in [0,1)$ for which $\gamma \mu_1 \geq \mu_2$. Assume that 
 \begin{align}\label{qassumptions}
  \tau(\ten x) := \frac{\varrho(\ten x)}{\mu(\ten x)} \leq \frac{1 - \gamma}{2 \sqrt{\gamma}} 
 \quad \text{and}\quad  \sin^2 \angle_{\ten A}(\ten{v}_1,\ten x) \leq \frac{1}{1+\gamma}.
 \end{align}
 Then
 \begin{align*}
  \frac{\mu(\ten x)}{\mu_1} \geq q(\gamma, \tau(\ten x)), \quad q(\gamma, \tau(\ten x)) := \frac{\gamma + 1 + \sqrt{ (\gamma - 1)^2 -4\gamma \tau(\ten x)^2}}{2 (\tau(\ten x)^2 + 1)}.
 \end{align*}
\end{lemma}
\begin{proof}
The function $q(s,t)$ is well-defined for $0 \leq t \leq (1-s)/(2\sqrt{s})$, $s \in [0,1)$, as
the unique solution $z$ to the polynomial equation $(1-z)(z-s) = t^2 z^2$ subject to $z \geq 2s/(s+1)$. Note that the case where both solutions coincide at $z = 2s/(s+1)$ corresponds to $t = (1-s)/(2\sqrt{s})$.
One easily checks that $z = q(s,t)$ is monotonically non-increasing in both $s$ and in $t$. We now continue based on the notation and assertions in the proof of \Cref{sinbound}. 
For $\tau|_\mathcal{X} := \varrho|_\mathcal{X}(\ten x) / \mu(\ten x)$, \Cref{lem:temple} by division through $a^2 = \mu_1^2$  yields 
\begin{align*}
 (1-q_0)(q_0-\gamma_0) = \frac{a - \mu(\ten x)}{a} \frac{\mu(\ten x) - b}{a} = \frac{\varrho|_\mathcal{X}(\ten x)^2}{\mu(\ten x)^2} \frac{\mu(\ten x)^2}{a^2} = \tau|_\mathcal{X}^2 q_0^2. %
\end{align*}
With the second assumption of \cref{qassumptions}, \cref{eq:sinqgamma} and $\gamma \geq \gamma_0$, it further follows that 
\begin{align*}
 q_0 = 1 - \sin^2 \angle_{\ten A}(\ten{v}_1,\ten x) (1-\gamma_0) \geq 1 - (1-\gamma_0) / (1+\gamma_0) \geq 2\gamma_0 / (1+\gamma_0).
\end{align*}
Thus, by the properties of $q(s,t)$, we obtain $q_0 = q(\gamma_0, \tau|_{\mathcal{X}})$. 
As $\tau|_\mathcal{X} \leq \tau(\ten x) \leq (1-\gamma)/(2\sqrt{\gamma}) \leq (1-\gamma_0)/(2\sqrt{\gamma_0})$ by $\ten x \in \mathcal{X}$ and the first assumption of \cref{qassumptions}, respectively, as well as $\gamma_0 \leq \gamma$ by definition, we have $q_0 = \mu(\ten x)/\mu_1 \geq q(\gamma,\tau(\ten x))$. 
\end{proof}
Then second condition of \cref{qassumptions} in \cref{qbound} is implied by 
 \begin{align}\label{gammacoresp}
  \frac{\mu(\ten x)}{\mu_1} \geq \frac{2\gamma}{\gamma + 1}, %
 \end{align}
 as by \cref{sinbound}, it directly follows that $\sin^2 \angle_{\ten A}(\ten{v}_1,\ten x) \leq (1-2\gamma/(\gamma + 1))/(1-\gamma) = 1/(1+\gamma)$. This corresponds to the condition \cref{eq:imply2ndeq}.

\end{appendix}
\end{document}

\FloatBarrier

\section{All experiments}

\newpage  \FloatBarrier
\begin{itemize}
 \item $K = \exponeoneKlowDoneK$, $N = \exponeoneKlowDoneN$, $D = \exponeoneKlowDoneD$, $\gamma = \exponeoneKlowDonegamma$,  $\mathrm{maxwavelvl} = \exponeoneKlowDonemmax$
 \item $D_{\mathrm{aux}} = \diag(\exponeoneKlowDonetone(\ldots), \exponeoneKlowDonettwo(\ldots),\ldots,\exponeoneKlowDonetKmone(\ldots), \exponeoneKlowDonetK(\ldots))$
 \item $[t_{\min},t_{\max}] = [\exponeoneKlowDonetmin(\ldots),\exponeoneKlowDonetmax(\ldots)]$, $t_{\max} / t_{\min} = \exponeoneKlowDoneT(\ldots)$
 \item $\delta := \exponeoneKlowDonekappa(\ldots)$, $c = \eiter = \exponeoneKlowDoneepsiter(\ldots)$, $\enum = \exponeoneKlowDoneepsnum(\ldots)$, $\#_S = \exponeoneKlowDonelengthS$
 \item $\lambda_1^{\mathrm{full}} = \exponeoneKlowDonelambdaonetrue$, $\lambda_1^{\mathrm{inner/outer}} = \exponeoneKlowDonelambdaoneIO$, $\lambda_1^{\mathrm{inner}} = \exponeoneKlowDonelambdaoneIonly$
\end{itemize}
\nolabeltikzfigure{h}{1}{exp11_KlowD1_bounds}{Eigenvalue error bounds.}
\nolabeltikzfigure{htb}{1}{exp11_KlowD1_ranks}{Maximal ranks of iterates by residual}
\nolabeltikzfigure{h}{1}{exp11_KlowD1_ranksres}{Maximal ranks of iterates/residuals by iteration}
\nolabeltikzfigure{h}{1}{exp11_KlowD1_ranksheatmap}{Maximal ranks of iterates/residuals by iteration}
\nolabeltikzfigure{h}{1}{exp11_KlowD1_ranksheatmap_plain}{Maximal ranks of iterates/residuals by iteration}
\FloatBarrier

\newpage \mbox{}\newpage  \newpage \FloatBarrier
\input{numbers/exp12_KhighD1_data.txt}
\begin{itemize}
 \item $K = \exponetwoKhighDoneK$, $N = \exponetwoKhighDoneN$, $D = \exponetwoKhighDoneD$, $\gamma = \exponetwoKhighDonegamma$,  $\mathrm{maxwavelvl} = \exponetwoKhighDonemmax$
 \item $D_{\mathrm{aux}} = \diag(\exponetwoKhighDonetone(\ldots), \exponetwoKhighDonettwo(\ldots),\ldots,\exponetwoKhighDonetKmone(\ldots), \exponetwoKhighDonetK(\ldots))$
 \item $[t_{\min},t_{\max}] = [\exponetwoKhighDonetmin(\ldots),\exponetwoKhighDonetmax(\ldots)]$, $t_{\max} / t_{\min} = \exponetwoKhighDoneT(\ldots)$
 \item $\delta := \exponetwoKhighDonekappa(\ldots)$, $c = \eiter = \exponetwoKhighDoneepsiter(\ldots)$, $\enum = \exponetwoKhighDoneepsnum(\ldots)$, $\#_S = \exponetwoKhighDonelengthS$
 \item $\lambda_1^{\mathrm{inner/outer}} = \exponetwoKhighDonelambdaoneIO$, $\lambda_1^{\mathrm{inner}} = \exponetwoKhighDonelambdaoneIonly$
\end{itemize}
\input{numbers/exp121_KhighD1mminus_data.txt}\vspace{0.5cm}
\begin{itemize}
 \item $K = \exponetwooneKhighDonemminusK$, $N = \exponetwooneKhighDonemminusN$, $D = \exponetwooneKhighDonemminusD$, $\gamma = \exponetwooneKhighDonemminusgamma$,  $\mathbf{\mathrm{maxwavelvl} = \exponetwooneKhighDonemminusmmax}$
 \item $D_{\mathrm{aux}} = \diag(\exponetwooneKhighDonemminustone(\ldots), \exponetwooneKhighDonemminusttwo(\ldots),\ldots,\exponetwooneKhighDonemminustKmone(\ldots), \exponetwooneKhighDonemminustK(\ldots))$
 \item $[t_{\min},t_{\max}] = [\exponetwooneKhighDonemminustmin(\ldots),\exponetwooneKhighDonemminustmax(\ldots)]$, $t_{\max} / t_{\min} = \exponetwooneKhighDonemminusT(\ldots)$
 \item $\delta := \exponetwooneKhighDonemminuskappa(\ldots)$, $c = \eiter = \exponetwooneKhighDonemminusepsiter(\ldots)$, $\enum = \exponetwooneKhighDonemminusepsnum(\ldots)$, $\#_S = \exponetwooneKhighDonemminuslengthS$
 \item $\lambda_1^{\mathrm{inner/outer}} = \exponetwooneKhighDonemminuslambdaoneIO$, $\lambda_1^{\mathrm{inner}} = \exponetwooneKhighDonemminuslambdaoneIonly$
\end{itemize}
\nolabeltikzfigure{htb}{1}{exp12_KhighD1_ranks}{Maximal ranks of iterates by residual}
\nolabeltikzfigure{h}{1}{exp12_KhighD1_ranksres}{Maximal ranks of iterates/residuals by iteration}
\nolabeltikzfigure{h}{1}{exp12_KhighD1_ranksheatmap}{Maximal ranks of iterates/residuals by iteration}
\nolabeltikzfigure{h}{1}{exp12_KhighD1_ranksheatmap_plain}{Maximal ranks of iterates/residuals by iteration}
\FloatBarrier

\newpage \mbox{}\newpage  \FloatBarrier
\input{numbers/exp13_KhigherD1_data.txt}
\begin{itemize}
 \item $K = \exponethreeKhigherDoneK$, $N = \exponethreeKhigherDoneN$, $D = \exponethreeKhigherDoneD$, $\gamma = \exponethreeKhigherDonegamma$,  $\mathrm{maxwavelvl} = \exponethreeKhigherDonemmax$
 \item $D_{\mathrm{aux}} = \diag(\exponethreeKhigherDonetone(\ldots), \exponethreeKhigherDonettwo(\ldots),\ldots,\exponethreeKhigherDonetKmone(\ldots), \exponethreeKhigherDonetK(\ldots))$
 \item $[t_{\min},t_{\max}] = [\exponethreeKhigherDonetmin(\ldots),\exponethreeKhigherDonetmax(\ldots)]$, $t_{\max} / t_{\min} = \exponethreeKhigherDoneT(\ldots)$
 \item $\delta := \exponethreeKhigherDonekappa(\ldots)$, $c = \eiter = \exponethreeKhigherDoneepsiter(\ldots)$, $\enum = \exponethreeKhigherDoneepsnum(\ldots)$, $\#_S = \exponethreeKhigherDonelengthS$
 \item $\lambda_1^{\mathrm{inner/outer}} = \exponethreeKhigherDonelambdaoneIO$, $\lambda_1^{\mathrm{inner}} = \exponethreeKhigherDonelambdaoneIonly$
\end{itemize}
\nolabeltikzfigure{htb}{1}{exp13_KhigherD1_ranks}{Maximal ranks of iterates by residual}
\nolabeltikzfigure{h}{1}{exp13_KhigherD1_ranksres}{Maximal ranks of iterates/residuals by iteration}
\nolabeltikzfigure{h}{1}{exp13_KhigherD1_ranksheatmap}{Maximal ranks of iterates/residuals by iteration}
\FloatBarrier

\newpage \mbox{}\newpage  \FloatBarrier
\begin{itemize}
 \item $K = \exponesixKevenhigherDoneK$, $N = \exponesixKevenhigherDoneN$, $D = \exponesixKevenhigherDoneD$, $\gamma = \exponesixKevenhigherDonegamma$,  $\mathrm{maxwavelvl} = \exponesixKevenhigherDonemmax$
 \item $D_{\mathrm{aux}} = \diag(\exponesixKevenhigherDonetone(\ldots), \exponesixKevenhigherDonettwo(\ldots),\ldots,\exponesixKevenhigherDonetKmone(\ldots), \exponesixKevenhigherDonetK(\ldots))$
 \item $[t_{\min},t_{\max}] = [\exponesixKevenhigherDonetmin(\ldots),\exponesixKevenhigherDonetmax(\ldots)]$, $t_{\max} / t_{\min} = \exponesixKevenhigherDoneT(\ldots)$
 \item $\delta := \exponesixKevenhigherDonekappa(\ldots)$, $c = \eiter = \exponesixKevenhigherDoneepsiter(\ldots)$, $\enum = \exponesixKevenhigherDoneepsnum(\ldots)$, $\#_S = \exponesixKevenhigherDonelengthS$
 \item $\lambda_1^{\mathrm{inner/outer}} = \exponesixKevenhigherDonelambdaoneIO$, $\lambda_1^{\mathrm{inner}} = \exponesixKevenhigherDonelambdaoneIonly$
\end{itemize}
\nolabeltikzfigure{htb}{1}{exp16_KevenhigherD1_ranks}{Maximal ranks of iterates by residual}
\nolabeltikzfigure{h}{1}{exp16_KevenhigherD1_ranksres}{Maximal ranks of iterates/residuals by iteration}
\nolabeltikzfigure{h}{1}{exp16_KevenhigherD1_ranksheatmap}{Maximal ranks of iterates/residuals by iteration}
\nolabeltikzfigure{h}{1}{exp16_KevenhigherD1_ranksheatmap_plain}{Maximal ranks of iterates/residuals by iteration}
\FloatBarrier

 \newpage \mbox{}\newpage  \FloatBarrier
\input{numbers/exp14_KlowD1N6_data.txt}
\begin{itemize}
 \item $K = \exponefourKlowDoneNsixK$, $N = \exponefourKlowDoneNsixN$, $D = \exponefourKlowDoneNsixD$, $\gamma = \exponefourKlowDoneNsixgamma$,  $\mathrm{maxwavelvl} = \exponefourKlowDoneNsixmmax$
 \item $D_{\mathrm{aux}} = \diag(\exponefourKlowDoneNsixtone(\ldots), \exponefourKlowDoneNsixttwo(\ldots),\ldots,\exponefourKlowDoneNsixtKmone(\ldots), \exponefourKlowDoneNsixtK(\ldots))$
 \item $[t_{\min},t_{\max}] = [\exponefourKlowDoneNsixtmin(\ldots),\exponefourKlowDoneNsixtmax(\ldots)]$, $t_{\max} / t_{\min} = \exponefourKlowDoneNsixT(\ldots)$
 \item $\delta := \exponefourKlowDoneNsixkappa(\ldots)$, $c = \eiter = \exponefourKlowDoneNsixepsiter(\ldots)$, $\enum = \exponefourKlowDoneNsixepsnum(\ldots)$, $\#_S = \exponefourKlowDoneNsixlengthS$
 \item $\lambda_1^{\mathrm{inner/outer}} = \exponefourKlowDoneNsixlambdaoneIO$, $\lambda_1^{\mathrm{inner}} = \exponefourKlowDoneNsixlambdaoneIonly$
\end{itemize}
\nolabeltikzfigure{h}{1}{exp14_KlowD1N6_bounds}{Eigenvalue error bounds.}
\nolabeltikzfigure{h}{1}{exp14_KlowD1N6_ranks}{Maximal ranks of iterates by residual}
\nolabeltikzfigure{h}{1}{exp14_KlowD1N6_ranksres}{Maximal ranks of iterates/residuals by iteration}
\nolabeltikzfigure{h}{1}{exp14_KlowD1N6_ranksheatmap}{Maximal ranks of iterates/residuals by iteration}
\nolabeltikzfigure{h}{1}{exp14_KlowD1N6_ranksheatmap_plain}{Maximal ranks of iterates/residuals by iteration}
\FloatBarrier

 \newpage \mbox{}\newpage  \FloatBarrier
\input{numbers/exp21_KlowD2_data.txt}
\begin{itemize}
 \item $K = \exptwooneKlowDtwoK$, $N = \exptwooneKlowDtwoN$, $D = \exptwooneKlowDtwoD$, $\gamma = \exptwooneKlowDtwogamma$,  $\mathrm{maxwavelvl} = \exptwooneKlowDtwommax$
 \item $D_{\mathrm{aux}} = \diag(\exptwooneKlowDtwotone(\ldots), \exptwooneKlowDtwottwo(\ldots),\ldots,\exptwooneKlowDtwotKmone(\ldots), \exptwooneKlowDtwotK(\ldots))$
 \item $[t_{\min},t_{\max}] = [\exptwooneKlowDtwotmin(\ldots),\exptwooneKlowDtwotmax(\ldots)]$, $t_{\max} / t_{\min} = \exptwooneKlowDtwoT(\ldots)$
 \item $\delta := \exptwooneKlowDtwokappa(\ldots)$, $c = \eiter = \exptwooneKlowDtwoepsiter(\ldots)$, $\enum = \exptwooneKlowDtwoepsnum(\ldots)$, $\#_S = \exptwooneKlowDtwolengthS$
 \item 
 \begin{align*}
 (\lambda_i^{\mathrm{full}})_{i = 1}^{\exptwooneKlowDtwoD} & = \exptwooneKlowDtwolambdaDtrue \\
  (\lambda_i^{\mathrm{inner/outer}})_{i = 1}^{\exptwooneKlowDtwoD} & = \exptwooneKlowDtwolambdaDIO.
\end{align*}
\end{itemize}
\nolabeltikzfigure{htb}{1}{exp21_KlowD2_singleerrs}{Single Rayleigh Quotients}
\nolabeltikzfigure{h}{1}{exp21_KlowD2_ranksres}{Maximal ranks of iterates/residuals by iteration}
\nolabeltikzfigure{h}{1}{exp21_KlowD2_ranksheatmap}{Maximal ranks of iterates/residuals by iteration}
\FloatBarrier

 \newpage \mbox{}\newpage  \FloatBarrier
\input{numbers/exp22_KhighD2_data.txt}
\begin{itemize}
 \item $K = \exptwotwoKhighDtwoK$, $N = \exptwotwoKhighDtwoN$, $D = \exptwotwoKhighDtwoD$, $\gamma = \exptwotwoKhighDtwogamma$,  $\mathrm{maxwavelvl} = \exptwotwoKhighDtwommax$
 \item $D_{\mathrm{aux}} = \diag(\exptwotwoKhighDtwotone(\ldots), \exptwotwoKhighDtwottwo(\ldots),\ldots,\exptwotwoKhighDtwotKmone(\ldots), \exptwotwoKhighDtwotK(\ldots))$
 \item $[t_{\min},t_{\max}] = [\exptwotwoKhighDtwotmin(\ldots),\exptwotwoKhighDtwotmax(\ldots)]$, $t_{\max} / t_{\min} = \exptwotwoKhighDtwoT(\ldots)$
 \item $\delta := \exptwotwoKhighDtwokappa(\ldots)$, $c = \eiter = \exptwotwoKhighDtwoepsiter(\ldots)$, $\enum = \exptwotwoKhighDtwoepsnum(\ldots)$, $\#_S = \exptwotwoKhighDtwolengthS$
  \item 
 \begin{align*}
  (\lambda_i^{\mathrm{inner/outer}})_{i = 1}^{\exptwotwoKhighDtwoD} & = \exptwotwoKhighDtwolambdaDIO.
\end{align*}
\end{itemize}
\nolabeltikzfigure{htb}{1}{exp22_KhighD2_ranks}{Maximal ranks of iterates by residual}
\nolabeltikzfigure{h}{1}{exp22_KhighD2_ranksres}{Maximal ranks of iterates/residuals by iteration}
\nolabeltikzfigure{h}{1}{exp22_KhighD2_ranksheatmap}{Maximal ranks of iterates/residuals by iteration}
\FloatBarrier

 \newpage \mbox{}\newpage  \FloatBarrier
\begin{itemize}
 \item $K = \expfouroneKlowDfourK$, $N = \expfouroneKlowDfourN$, $D = \expfouroneKlowDfourD$, $\gamma = \expfouroneKlowDfourgamma$,  $\mathrm{maxwavelvl} = \expfouroneKlowDfourmmax$
 \item $D_{\mathrm{aux}} = \diag(\expfouroneKlowDfourtone(\ldots), \expfouroneKlowDfourttwo(\ldots),\ldots,\expfouroneKlowDfourtKmone(\ldots), \expfouroneKlowDfourtK(\ldots))$
 \item $[t_{\min},t_{\max}] = [\expfouroneKlowDfourtmin(\ldots),\expfouroneKlowDfourtmax(\ldots)]$, $t_{\max} / t_{\min} = \expfouroneKlowDfourT(\ldots)$
 \item $\delta := \expfouroneKlowDfourkappa(\ldots)$, $c = \eiter = \expfouroneKlowDfourepsiter(\ldots)$, $\enum = \expfouroneKlowDfourepsnum(\ldots)$, $\#_S = \expfouroneKlowDfourlengthS$
 \item 
 \begin{align*}
 (\lambda_i^{\mathrm{full}})_{i = 1}^{\expfouroneKlowDfourD} & = \expfouroneKlowDfourlambdaDtrue \\
  (\lambda_i^{\mathrm{inner/outer}})_{i = 1}^{\expfouroneKlowDfourD} & = \expfouroneKlowDfourlambdaDIO.
\end{align*}
\end{itemize}
\nolabeltikzfigure{htb}{1}{exp41_KlowD4_singleerrs}{Single Rayleigh Quotients}
\nolabeltikzfigure{h}{1}{exp41_KlowD4_ranksres}{Maximal ranks of iterates/residuals by iteration}
\nolabeltikzfigure{h}{1}{exp41_KlowD4_ranksheatmap}{Maximal ranks of iterates/residuals by iteration}
\FloatBarrier

 \newpage \mbox{}\newpage  \FloatBarrier
\input{numbers/exp42_KhighD4_data.txt}
\begin{itemize}
 \item $K = \expfourtwoKhighDfourK$, $N = \expfourtwoKhighDfourN$, $D = \expfourtwoKhighDfourD$, $\gamma = \expfourtwoKhighDfourgamma$,  $\mathrm{maxwavelvl} = \expfourtwoKhighDfourmmax$
 \item $D_{\mathrm{aux}} = \diag(\expfourtwoKhighDfourtone(\ldots), \expfourtwoKhighDfourttwo(\ldots),\ldots,\expfourtwoKhighDfourtKmone(\ldots), \expfourtwoKhighDfourtK(\ldots))$
 \item $[t_{\min},t_{\max}] = [\expfourtwoKhighDfourtmin(\ldots),\expfourtwoKhighDfourtmax(\ldots)]$, $t_{\max} / t_{\min} = \expfourtwoKhighDfourT(\ldots)$
 \item $\delta := \expfourtwoKhighDfourkappa(\ldots)$, $c = \eiter = \expfourtwoKhighDfourepsiter(\ldots)$, $\enum = \expfourtwoKhighDfourepsnum(\ldots)$, $\#_S = \expfourtwoKhighDfourlengthS$
  \item 
 \begin{align*}
  (\lambda_i^{\mathrm{inner/outer}})_{i = 1}^{\expfourtwoKhighDfourD} & = \expfourtwoKhighDfourlambdaDIO.
\end{align*}
\end{itemize}
\nolabeltikzfigure{htb}{1}{exp42_KhighD4_ranks}{Maximal ranks of iterates by residual}
\nolabeltikzfigure{h}{1}{exp42_KhighD4_ranksres}{Maximal ranks of iterates/residuals by iteration}
\nolabeltikzfigure{h}{1}{exp42_KhighD4_ranksheatmap}{Maximal ranks of iterates/residuals by iteration}
\FloatBarrier

 \newpage \mbox{}\newpage  \FloatBarrier
\input{numbers/exp43_KhigherD4_data.txt}
\begin{itemize}
 \item $K = \expfourthreeKhigherDfourK$, $N = \expfourthreeKhigherDfourN$, $D = \expfourthreeKhigherDfourD$, $\gamma = \expfourthreeKhigherDfourgamma$,  $\mathrm{maxwavelvl} = \expfourthreeKhigherDfourmmax$
 \item $D_{\mathrm{aux}} = \diag(\expfourthreeKhigherDfourtone(\ldots), \expfourthreeKhigherDfourttwo(\ldots),\ldots,\expfourthreeKhigherDfourtKmone(\ldots), \expfourthreeKhigherDfourtK(\ldots))$
 \item $[t_{\min},t_{\max}] = [\expfourthreeKhigherDfourtmin(\ldots),\expfourthreeKhigherDfourtmax(\ldots)]$, $t_{\max} / t_{\min} = \expfourthreeKhigherDfourT(\ldots)$
 \item $\delta := \expfourthreeKhigherDfourkappa(\ldots)$, $c = \eiter = \expfourthreeKhigherDfourepsiter(\ldots)$, $\enum = \expfourthreeKhigherDfourepsnum(\ldots)$, $\#_S = \expfourthreeKhigherDfourlengthS$
  \item 
 \begin{align*}
  (\lambda_i^{\mathrm{inner/outer}})_{i = 1}^{\expfourthreeKhigherDfourD} & = \expfourthreeKhigherDfourlambdaDIO.
\end{align*}
\end{itemize}
\nolabeltikzfigure{htb}{1}{exp43_KhigherD4_ranks}{Maximal ranks of iterates by residual}
\nolabeltikzfigure{h}{1}{exp43_KhigherD4_ranksres}{Maximal ranks of iterates/residuals by iteration}
\nolabeltikzfigure{h}{1}{exp43_KhigherD4_ranksheatmap}{Maximal ranks of iterates/residuals by iteration}
\FloatBarrier

 \newpage \mbox{}\newpage  \FloatBarrier
\begin{itemize}
 \item $K = \expfoursixKevenhigherDfourK$, $N = \expfoursixKevenhigherDfourN$, $D = \expfoursixKevenhigherDfourD$, $\gamma = \expfoursixKevenhigherDfourgamma$,  $\mathrm{maxwavelvl} = \expfoursixKevenhigherDfourmmax$
 \item $D_{\mathrm{aux}} = \diag(\expfoursixKevenhigherDfourtone(\ldots), \expfoursixKevenhigherDfourttwo(\ldots),\ldots,\expfoursixKevenhigherDfourtKmone(\ldots), \expfoursixKevenhigherDfourtK(\ldots))$
 \item $[t_{\min},t_{\max}] = [\expfoursixKevenhigherDfourtmin(\ldots),\expfoursixKevenhigherDfourtmax(\ldots)]$, $t_{\max} / t_{\min} = \expfoursixKevenhigherDfourT(\ldots)$
 \item $\delta := \expfoursixKevenhigherDfourkappa(\ldots)$, $c = \eiter = \expfoursixKevenhigherDfourepsiter(\ldots)$, $\enum = \expfoursixKevenhigherDfourepsnum(\ldots)$, $\#_S = \expfoursixKevenhigherDfourlengthS$
  \item 
 \begin{align*}
  (\lambda_i^{\mathrm{inner/outer}})_{i = 1}^{\expfoursixKevenhigherDfourD} & = \expfoursixKevenhigherDfourlambdaDIO.
\end{align*}
\end{itemize}
\nolabeltikzfigure{htb}{1}{exp46_KevenhigherD4_ranks}{Maximal ranks of iterates by residual}
\nolabeltikzfigure{h}{1}{exp46_KevenhigherD4_ranksres}{Maximal ranks of iterates/residuals by iteration}
\nolabeltikzfigure{h}{1}{exp46_KevenhigherD4_ranksheatmap}{Maximal ranks of iterates/residuals by iteration}
\FloatBarrier

 \newpage \mbox{}\newpage  \FloatBarrier
\input{numbers/exp44_KlowD4N6_data.txt}
\begin{itemize}
 \item $K = \expfourfourKlowDfourNsixK$, $N = \expfourfourKlowDfourNsixN$, $D = \expfourfourKlowDfourNsixD$, $\gamma = \expfourfourKlowDfourNsixgamma$,  $\mathrm{maxwavelvl} = \expfourfourKlowDfourNsixmmax$
 \item $D_{\mathrm{aux}} = \diag(\expfourfourKlowDfourNsixtone(\ldots), \expfourfourKlowDfourNsixttwo(\ldots),\ldots,\expfourfourKlowDfourNsixtKmone(\ldots), \expfourfourKlowDfourNsixtK(\ldots))$
 \item $[t_{\min},t_{\max}] = [\expfourfourKlowDfourNsixtmin(\ldots),\expfourfourKlowDfourNsixtmax(\ldots)]$, $t_{\max} / t_{\min} = \expfourfourKlowDfourNsixT(\ldots)$
 \item $\delta := \expfourfourKlowDfourNsixkappa(\ldots)$, $c = \eiter = \expfourfourKlowDfourNsixepsiter(\ldots)$, $\enum = \expfourfourKlowDfourNsixepsnum(\ldots)$, $\#_S = \expfourfourKlowDfourNsixlengthS$
 \item 
 \begin{align*}
 (\lambda_i^{\mathrm{full}})_{i = 1}^{\expfourfourKlowDfourNsixD} & = \expfourfourKlowDfourNsixlambdaDtrue \\
  (\lambda_i^{\mathrm{inner/outer}})_{i = 1}^{\expfourfourKlowDfourNsixD} & = \expfourfourKlowDfourNsixlambdaDIO.
\end{align*}
\end{itemize}
\nolabeltikzfigure{htb}{1}{exp44_KlowD4N6_singleerrs}{Single Rayleigh Quotients}
\nolabeltikzfigure{h}{1}{exp44_KlowD4N6_ranksres}{Maximal ranks of iterates/residuals by iteration}
\nolabeltikzfigure{h}{1}{exp44_KlowD4N6_ranksheatmap}{Maximal ranks of iterates/residuals by iteration}
\FloatBarrier

 \newpage \mbox{}\newpage  \FloatBarrier
\input{numbers/exp61_KlowD6_data.txt}
\begin{itemize}
 \item $K = \expsixoneKlowDsixK$, $N = \expsixoneKlowDsixN$, $D = \expsixoneKlowDsixD$, $\gamma = \expsixoneKlowDsixgamma$,  $\mathrm{maxwavelvl} = \expsixoneKlowDsixmmax$
 \item $D_{\mathrm{aux}} = \diag(\expsixoneKlowDsixtone(\ldots), \expsixoneKlowDsixttwo(\ldots),\ldots,\expsixoneKlowDsixtKmone(\ldots), \expsixoneKlowDsixtK(\ldots))$
 \item $[t_{\min},t_{\max}] = [\expsixoneKlowDsixtmin(\ldots),\expsixoneKlowDsixtmax(\ldots)]$, $t_{\max} / t_{\min} = \expsixoneKlowDsixT(\ldots)$
 \item $\delta := \expsixoneKlowDsixkappa(\ldots)$, $c = \eiter = \expsixoneKlowDsixepsiter(\ldots)$, $\enum = \expsixoneKlowDsixepsnum(\ldots)$, $\#_S = \expsixoneKlowDsixlengthS$
\item 
 \begin{align*}
 (\lambda_i^{\mathrm{full}})_{i = 1}^{\expsixoneKlowDsixD} & = \expsixoneKlowDsixlambdaDtrue \\
  (\lambda_i^{\mathrm{inner/outer}})_{i = 1}^{\expsixoneKlowDsixD} & = \expsixoneKlowDsixlambdaDIO.
\end{align*}
\end{itemize}
\nolabeltikzfigure{htb}{1}{exp61_KlowD6_singleerrs}{Single Rayleigh Quotients}
\nolabeltikzfigure{h}{1}{exp61_KlowD6_ranksres}{Maximal ranks of iterates/residuals by iteration}
\nolabeltikzfigure{h}{1}{exp61_KlowD6_ranksheatmap}{Maximal ranks of iterates/residuals by iteration}
\FloatBarrier

 \newpage \mbox{}\newpage  \FloatBarrier
\input{numbers/exp62_KhighD6_data.txt}
\begin{itemize}
 \item $K = \expsixtwoKhighDsixK$, $N = \expsixtwoKhighDsixN$, $D = \expsixtwoKhighDsixD$, $\gamma = \expsixtwoKhighDsixgamma$,  $\mathrm{maxwavelvl} = \expsixtwoKhighDsixmmax$
 \item $D_{\mathrm{aux}} = \diag(\expsixtwoKhighDsixtone(\ldots), \expsixtwoKhighDsixttwo(\ldots),\ldots,\expsixtwoKhighDsixtKmone(\ldots), \expsixtwoKhighDsixtK(\ldots))$
 \item $[t_{\min},t_{\max}] = [\expsixtwoKhighDsixtmin(\ldots),\expsixtwoKhighDsixtmax(\ldots)]$, $t_{\max} / t_{\min} = \expsixtwoKhighDsixT(\ldots)$
 \item $\delta := \expsixtwoKhighDsixkappa(\ldots)$, $c = \eiter = \expsixtwoKhighDsixepsiter(\ldots)$, $\enum = \expsixtwoKhighDsixepsnum(\ldots)$, $\#_S = \expsixtwoKhighDsixlengthS$
  \item 
 \begin{align*}
  (\lambda_i^{\mathrm{inner/outer}})_{i = 1}^{\expsixtwoKhighDsixD} & = \expsixtwoKhighDsixlambdaDIO.
\end{align*}
\end{itemize}
\nolabeltikzfigure{htb}{1}{exp62_KhighD6_ranks}{Maximal ranks of iterates by residual}
\nolabeltikzfigure{h}{1}{exp62_KhighD6_ranksres}{Maximal ranks of iterates/residuals by iteration}
\nolabeltikzfigure{h}{1}{exp62_KhighD6_ranksheatmap}{Maximal ranks of iterates/residuals by iteration}
\FloatBarrier

 \newpage \mbox{}\newpage  \FloatBarrier
\input{numbers/exp63_KhigherD6_data.txt}
\begin{itemize}
 \item $K = \expsixthreeKhigherDsixK$, $N = \expsixthreeKhigherDsixN$, $D = \expsixthreeKhigherDsixD$, $\gamma = \expsixthreeKhigherDsixgamma$,  $\mathrm{maxwavelvl} = \expsixthreeKhigherDsixmmax$
 \item $D_{\mathrm{aux}} = \diag(\expsixthreeKhigherDsixtone(\ldots), \expsixthreeKhigherDsixttwo(\ldots),\ldots,\expsixthreeKhigherDsixtKmone(\ldots), \expsixthreeKhigherDsixtK(\ldots))$
 \item $[t_{\min},t_{\max}] = [\expsixthreeKhigherDsixtmin(\ldots),\expsixthreeKhigherDsixtmax(\ldots)]$, $t_{\max} / t_{\min} = \expsixthreeKhigherDsixT(\ldots)$
 \item $\delta := \expsixthreeKhigherDsixkappa(\ldots)$, $c = \eiter = \expsixthreeKhigherDsixepsiter(\ldots)$, $\enum = \expsixthreeKhigherDsixepsnum(\ldots)$, $\#_S = \expsixthreeKhigherDsixlengthS$
  \item 
 \begin{align*}
  (\lambda_i^{\mathrm{inner/outer}})_{i = 1}^{\expsixthreeKhigherDsixD} & = \expsixthreeKhigherDsixlambdaDIO.
\end{align*}
\end{itemize}
\nolabeltikzfigure{htb}{1}{exp63_KhigherD6_ranks}{Maximal ranks of iterates by residual}
\nolabeltikzfigure{h}{1}{exp63_KhigherD6_ranksres}{Maximal ranks of iterates/residuals by iteration}
\nolabeltikzfigure{h}{1}{exp63_KhigherD6_ranksheatmap}{Maximal ranks of iterates/residuals by iteration}
\FloatBarrier

\end{appendix}

\end{document}

%% file: rwth-colors.tex
\definecolor{rwth}   {RGB}{  0  84 159}
\definecolor{rwth-75}{RGB}{ 64 127 183}
\definecolor{rwth-50}{RGB}{142 186 229}
\definecolor{rwth-25}{RGB}{199 221 242}
\definecolor{rwth-10}{RGB}{232 241 250}

\definecolor{black}   {RGB}{  0   0   0}
\definecolor{black-90}{RGB}{44   44  45}
\definecolor{black-80}{RGB}{83   84  86}
\definecolor{black-75}{RGB}{100 101 103}
\definecolor{black-50}{RGB}{156 158 159}
\definecolor{black-25}{RGB}{207 209 210}
\definecolor{black-10}{RGB}{236 237 237}

\definecolor{magenta}   {RGB}{227   0 102}
\definecolor{magenta-75}{RGB}{233  96 136}
\definecolor{magenta-50}{RGB}{241 158 177}
\definecolor{magenta-25}{RGB}{249 210 218}
\definecolor{magenta-10}{RGB}{253 238 240}

\definecolor{yellow}   {RGB}{255 237   0}
\definecolor{yellow-75}{RGB}{255 240  85}
\definecolor{yellow-50}{RGB}{255 245 155}
\definecolor{yellow-25}{RGB}{255 250 209}
\definecolor{yellow-10}{RGB}{255 253 238}

\definecolor{petrol}   {RGB}{  0  97 101}
\definecolor{petrol-75}{RGB}{ 45 127 131}
\definecolor{petrol-50}{RGB}{125 164 167}
\definecolor{petrol-25}{RGB}{191 208 209}
\definecolor{petrol-10}{RGB}{230 236 236}

\definecolor{turkis}   {RGB}{  0 152 161}
\definecolor{turkis-75}{RGB}{  0 177 183}
\definecolor{turkis-50}{RGB}{137 204 207}
\definecolor{turkis-25}{RGB}{202 231 231}
\definecolor{turkis-10}{RGB}{235 246 246}

\definecolor{grun}   {RGB}{ 87 171  39}
\definecolor{grun-75}{RGB}{141 192  96}
\definecolor{grun-50}{RGB}{184 214 152}
\definecolor{grun-25}{RGB}{221 235 206}
\definecolor{grun-10}{RGB}{242 247 236}

\definecolor{maigrun}   {RGB}{189 205   0}
\definecolor{maigrun-75}{RGB}{208 217  92}
\definecolor{maigrun-50}{RGB}{224 230 154}
\definecolor{maigrun-25}{RGB}{240 243 208}
\definecolor{maigrun-10}{RGB}{249 250 237}

\definecolor{orange}   {RGB}{246 168   0}
\definecolor{orange-75}{RGB}{250 190  80}
\definecolor{orange-50}{RGB}{253 212 143}
\definecolor{orange-25}{RGB}{254 234 201}
\definecolor{orange-10}{RGB}{255 247 234}

\definecolor{rot}   {RGB}{204   7  30}
\definecolor{rot-75}{RGB}{216  92  65}
\definecolor{rot-50}{RGB}{230 150 121}
\definecolor{rot-25}{RGB}{243 205 187}
\definecolor{rot-10}{RGB}{250 235 227}

\definecolor{bordeaux}   {RGB}{161  16  53}
\definecolor{bordeaux-75}{RGB}{182  82  86}
\definecolor{bordeaux-50}{RGB}{205 139 135}
\definecolor{bordeaux-25}{RGB}{229 197 192}
\definecolor{bordeaux-10}{RGB}{245 232 229}

\definecolor{violett}   {RGB}{ 97  33  88}
\definecolor{violett-75}{RGB}{131  78 117}
\definecolor{violett-50}{RGB}{168 133 158}
\definecolor{violett-25}{RGB}{210 192 205}
\definecolor{violett-10}{RGB}{237 229 234}

\definecolor{lila}   {RGB}{122 111 172}
\definecolor{lila-75}{RGB}{155 145 193}
\definecolor{lila-50}{RGB}{188 181 215}
\definecolor{lila-25}{RGB}{222 218 235}
\definecolor{lila-10}{RGB}{242 240 247}

%% file: numbers/exp11_KlowD1_data.txt
\newcommand{\exponeoneKlowDoneK}{\ensuremath{14}}
\newcommand{\exponeoneKlowDonemmax}{\ensuremath{18}}
\newcommand{\exponeoneKlowDonemmin}{\ensuremath{5}}
\newcommand{\exponeoneKlowDoneb}{\ensuremath{4}}
\newcommand{\exponeoneKlowDoneN}{\ensuremath{4}}
\newcommand{\exponeoneKlowDonegamma}{\ensuremath{4}}
\newcommand{\exponeoneKlowDoneD}{\ensuremath{1}}

\newcommand{\exponeoneKlowDonecfV}{\ensuremath{4}}

\newcommand{\exponeoneKlowDoneB}{\ensuremath{3594}}
\newcommand{\exponeoneKlowDonetone}{\ensuremath{-8.0}}
\newcommand{\exponeoneKlowDonettwo}{\ensuremath{0.31}}
\newcommand{\exponeoneKlowDonetKmone}{\ensuremath{12.18}}
\newcommand{\exponeoneKlowDonetK}{\ensuremath{15.11}}
\newcommand{\exponeoneKlowDonetmin}{\ensuremath{1.38}}
\newcommand{\exponeoneKlowDonetmax}{\ensuremath{54.5}}
\newcommand{\exponeoneKlowDoneT}{\ensuremath{39.61}}
\newcommand{\exponeoneKlowDonekappa}{\ensuremath{10.9096}}

\newcommand{\exponeoneKlowDonelambdaonetrue}{\ensuremath{0.93676803545383}}
\newcommand{\exponeoneKlowDonelambdaoneIO}{\ensuremath{0.93676803545431}}
\newcommand{\exponeoneKlowDonelambdaoneIonly}{\ensuremath{0.93676803545698}}
\newcommand{\exponeoneKlowDoneepsiter}{\ensuremath{0.57}}
\newcommand{\exponeoneKlowDoneepsnum}{\ensuremath{0.21}}

\newcommand{\exponeoneKlowDonelengthS}{\ensuremath{7}}

%% file: numbers/exp16_KevenhigherD1_data.txt
\newcommand{\exponesixKevenhigherDoneK}{\ensuremath{30}}
\newcommand{\exponesixKevenhigherDonemmax}{\ensuremath{18}}
\newcommand{\exponesixKevenhigherDonemmin}{\ensuremath{5}}

\newcommand{\exponesixKevenhigherDoneN}{\ensuremath{4}}
\newcommand{\exponesixKevenhigherDonegamma}{\ensuremath{4}}
\newcommand{\exponesixKevenhigherDoneD}{\ensuremath{1}}

\newcommand{\exponesixKevenhigherDonereltol}{\ensuremath{10^{-10}}}

\newcommand{\exponesixKevenhigherDonetone}{\ensuremath{-8.0}}
\newcommand{\exponesixKevenhigherDonettwo}{\ensuremath{0.31}}
\newcommand{\exponesixKevenhigherDonetKmone}{\ensuremath{63.88}}
\newcommand{\exponesixKevenhigherDonetK}{\ensuremath{69.4}}
\newcommand{\exponesixKevenhigherDonetmin}{\ensuremath{1.38}}
\newcommand{\exponesixKevenhigherDonetmax}{\ensuremath{256.64}}
\newcommand{\exponesixKevenhigherDoneT}{\ensuremath{186.52}}
\newcommand{\exponesixKevenhigherDonekappa}{\ensuremath{10.9096}}

\newcommand{\exponesixKevenhigherDonelambdaoneIO}{\ensuremath{0.93660387186093}}
\newcommand{\exponesixKevenhigherDonelambdaoneIonly}{\ensuremath{0.93660387186342}}
\newcommand{\exponesixKevenhigherDoneepsiter}{\ensuremath{0.57}}
\newcommand{\exponesixKevenhigherDoneepsnum}{\ensuremath{0.21}}

\newcommand{\exponesixKevenhigherDonelengthS}{\ensuremath{8}}

%% file: numbers/exp41_KlowD4_data.txt
\newcommand{\expfouroneKlowDfourK}{\ensuremath{14}}
\newcommand{\expfouroneKlowDfourmmax}{\ensuremath{18}}

\newcommand{\expfouroneKlowDfourN}{\ensuremath{4}}
\newcommand{\expfouroneKlowDfourgamma}{\ensuremath{4}}
\newcommand{\expfouroneKlowDfourD}{\ensuremath{4}}

\newcommand{\expfouroneKlowDfourtone}{\ensuremath{-8.0}}
\newcommand{\expfouroneKlowDfourttwo}{\ensuremath{0.31}}
\newcommand{\expfouroneKlowDfourtKmone}{\ensuremath{12.18}}
\newcommand{\expfouroneKlowDfourtK}{\ensuremath{15.11}}
\newcommand{\expfouroneKlowDfourtmin}{\ensuremath{1.38}}
\newcommand{\expfouroneKlowDfourtmax}{\ensuremath{54.5}}
\newcommand{\expfouroneKlowDfourT}{\ensuremath{39.61}}
\newcommand{\expfouroneKlowDfourkappa}{\ensuremath{10.9096}}

\newcommand{\expfouroneKlowDfourlambdaDtrue}{\ensuremath{(0.93676803545, 1.03129973333, 2.34915988893, 2.36420024713)}}

\newcommand{\expfouroneKlowDfourlambdaDIO}{\ensuremath{(0.9367680356, 1.03129973352, 2.34915990979, 2.36420024747)}}
\newcommand{\expfouroneKlowDfourepsiter}{\ensuremath{0.57}}
\newcommand{\expfouroneKlowDfourepsnum}{\ensuremath{0.21}}

\newcommand{\expfouroneKlowDfourlengthS}{\ensuremath{7}}

%% file: numbers/exp46_KevenhigherD4_data.txt
\newcommand{\expfoursixKevenhigherDfourK}{\ensuremath{30}}
\newcommand{\expfoursixKevenhigherDfourmmax}{\ensuremath{18}}

\newcommand{\expfoursixKevenhigherDfourN}{\ensuremath{4}}
\newcommand{\expfoursixKevenhigherDfourgamma}{\ensuremath{4}}
\newcommand{\expfoursixKevenhigherDfourD}{\ensuremath{4}}

\newcommand{\expfoursixKevenhigherDfourreltol}{\ensuremath{10^{-7}}}

\newcommand{\expfoursixKevenhigherDfourtone}{\ensuremath{-8.0}}
\newcommand{\expfoursixKevenhigherDfourttwo}{\ensuremath{0.31}}
\newcommand{\expfoursixKevenhigherDfourtKmone}{\ensuremath{63.88}}
\newcommand{\expfoursixKevenhigherDfourtK}{\ensuremath{69.4}}
\newcommand{\expfoursixKevenhigherDfourtmin}{\ensuremath{1.38}}
\newcommand{\expfoursixKevenhigherDfourtmax}{\ensuremath{256.64}}
\newcommand{\expfoursixKevenhigherDfourT}{\ensuremath{186.52}}
\newcommand{\expfoursixKevenhigherDfourkappa}{\ensuremath{10.9096}}

\newcommand{\expfoursixKevenhigherDfourlambdaDIO}{\ensuremath{(0.93660387187, 1.03109809184, 2.34897980674, 2.3640232121)}}
\newcommand{\expfoursixKevenhigherDfourepsiter}{\ensuremath{0.57}}
\newcommand{\expfoursixKevenhigherDfourepsnum}{\ensuremath{0.21}}

\newcommand{\expfoursixKevenhigherDfourlengthS}{\ensuremath{8}}

%% file: numbers/exp12_KhighD1_data.txt
\newcommand{\exponetwoKhighDoneK}{\ensuremath{24}}
\newcommand{\exponetwoKhighDonemmax}{\ensuremath{18}}
\newcommand{\exponetwoKhighDonemmin}{\ensuremath{5}}
\newcommand{\exponetwoKhighDoneb}{\ensuremath{4}}
\newcommand{\exponetwoKhighDoneN}{\ensuremath{4}}
\newcommand{\exponetwoKhighDonegamma}{\ensuremath{4}}
\newcommand{\exponetwoKhighDoneD}{\ensuremath{1}}
\newcommand{\exponetwoKhighDonecfT}{\ensuremath{4}}
\newcommand{\exponetwoKhighDonecfV}{\ensuremath{4}}
\newcommand{\exponetwoKhighDonereltol}{\ensuremath{10^{-10}}}
\newcommand{\exponetwoKhighDoneTVtol}{\ensuremath{0}}
\newcommand{\exponetwoKhighDoneB}{\ensuremath{3594}}
\newcommand{\exponetwoKhighDonetone}{\ensuremath{-8.0}}
\newcommand{\exponetwoKhighDonettwo}{\ensuremath{0.31}}
\newcommand{\exponetwoKhighDonetKmone}{\ensuremath{39.84}}
\newcommand{\exponetwoKhighDonetK}{\ensuremath{44.41}}
\newcommand{\exponetwoKhighDonetmin}{\ensuremath{1.38}}
\newcommand{\exponetwoKhighDonetmax}{\ensuremath{162.29}}
\newcommand{\exponetwoKhighDoneT}{\ensuremath{117.95}}
\newcommand{\exponetwoKhighDonekappa}{\ensuremath{10.9096}}
\newcommand{\exponetwoKhighDonecalg}{\ensuremath{0.57}}
\newcommand{\exponetwoKhighDoneupper}{\ensuremath{1.57}}
\newcommand{\exponetwoKhighDonelower}{\ensuremath{0.43}}
\newcommand{\exponetwoKhighDonelambdaoneIO}{\ensuremath{0.93661026524621}}
\newcommand{\exponetwoKhighDonelambdaoneIonly}{\ensuremath{0.9366102652482}}
\newcommand{\exponetwoKhighDoneepsiter}{\ensuremath{0.57}}
\newcommand{\exponetwoKhighDoneepsnum}{\ensuremath{0.21}}
\newcommand{\exponetwoKhighDonesq}{\ensuremath{0.5}}
\newcommand{\exponetwoKhighDonelengthS}{\ensuremath{8}}

%% file: numbers/exp121_KhighD1mminus_data.txt
\newcommand{\exponetwooneKhighDonemminusK}{\ensuremath{24}}
\newcommand{\exponetwooneKhighDonemminusmmax}{\ensuremath{17}}
\newcommand{\exponetwooneKhighDonemminusmmin}{\ensuremath{5}}
\newcommand{\exponetwooneKhighDonemminusb}{\ensuremath{4}}
\newcommand{\exponetwooneKhighDonemminusN}{\ensuremath{4}}
\newcommand{\exponetwooneKhighDonemminusgamma}{\ensuremath{4}}
\newcommand{\exponetwooneKhighDonemminusD}{\ensuremath{1}}
\newcommand{\exponetwooneKhighDonemminuscfT}{\ensuremath{4}}
\newcommand{\exponetwooneKhighDonemminuscfV}{\ensuremath{4}}
\newcommand{\exponetwooneKhighDonemminusreltol}{\ensuremath{10^{-10}}}
\newcommand{\exponetwooneKhighDonemminusTVtol}{\ensuremath{0}}
\newcommand{\exponetwooneKhighDonemminusB}{\ensuremath{3352}}
\newcommand{\exponetwooneKhighDonemminustone}{\ensuremath{-8.0}}
\newcommand{\exponetwooneKhighDonemminusttwo}{\ensuremath{0.31}}
\newcommand{\exponetwooneKhighDonemminustKmone}{\ensuremath{39.84}}
\newcommand{\exponetwooneKhighDonemminustK}{\ensuremath{44.41}}
\newcommand{\exponetwooneKhighDonemminustmin}{\ensuremath{1.38}}
\newcommand{\exponetwooneKhighDonemminustmax}{\ensuremath{162.29}}
\newcommand{\exponetwooneKhighDonemminusT}{\ensuremath{117.95}}
\newcommand{\exponetwooneKhighDonemminuskappa}{\ensuremath{10.9098}}
\newcommand{\exponetwooneKhighDonemminuscalg}{\ensuremath{0.57}}
\newcommand{\exponetwooneKhighDonemminusupper}{\ensuremath{1.57}}
\newcommand{\exponetwooneKhighDonemminuslower}{\ensuremath{0.43}}
\newcommand{\exponetwooneKhighDonemminuslambdaoneIO}{\ensuremath{0.93663533715566}}
\newcommand{\exponetwooneKhighDonemminuslambdaoneIonly}{\ensuremath{0.93663533715766}}
\newcommand{\exponetwooneKhighDonemminusepsiter}{\ensuremath{0.57}}
\newcommand{\exponetwooneKhighDonemminusepsnum}{\ensuremath{0.21}}
\newcommand{\exponetwooneKhighDonemminussq}{\ensuremath{0.5}}
\newcommand{\exponetwooneKhighDonemminuslengthS}{\ensuremath{8}}

%% file: numbers/exp13_KhigherD1_data.txt
\newcommand{\exponethreeKhigherDoneK}{\ensuremath{26}}
\newcommand{\exponethreeKhigherDonemmax}{\ensuremath{18}}
\newcommand{\exponethreeKhigherDonemmin}{\ensuremath{5}}
\newcommand{\exponethreeKhigherDoneb}{\ensuremath{4}}
\newcommand{\exponethreeKhigherDoneN}{\ensuremath{4}}
\newcommand{\exponethreeKhigherDonegamma}{\ensuremath{4}}
\newcommand{\exponethreeKhigherDoneD}{\ensuremath{1}}
\newcommand{\exponethreeKhigherDonecfT}{\ensuremath{4}}
\newcommand{\exponethreeKhigherDonecfV}{\ensuremath{4}}
\newcommand{\exponethreeKhigherDonereltol}{\ensuremath{10^{-10}}}
\newcommand{\exponethreeKhigherDoneTVtol}{\ensuremath{0}}
\newcommand{\exponethreeKhigherDoneB}{\ensuremath{3594}}
\newcommand{\exponethreeKhigherDonetone}{\ensuremath{-8.0}}
\newcommand{\exponethreeKhigherDonettwo}{\ensuremath{0.31}}
\newcommand{\exponethreeKhigherDonetKmone}{\ensuremath{47.24}}
\newcommand{\exponethreeKhigherDonetK}{\ensuremath{52.12}}
\newcommand{\exponethreeKhigherDonetmin}{\ensuremath{1.38}}
\newcommand{\exponethreeKhigherDonetmax}{\ensuremath{191.27}}
\newcommand{\exponethreeKhigherDoneT}{\ensuremath{139.01}}
\newcommand{\exponethreeKhigherDonekappa}{\ensuremath{10.9096}}
\newcommand{\exponethreeKhigherDonecalg}{\ensuremath{0.57}}
\newcommand{\exponethreeKhigherDoneupper}{\ensuremath{1.57}}
\newcommand{\exponethreeKhigherDonelower}{\ensuremath{0.43}}
\newcommand{\exponethreeKhigherDonelambdaoneIO}{\ensuremath{0.93660704620291}}
\newcommand{\exponethreeKhigherDonelambdaoneIonly}{\ensuremath{0.9366070462054}}
\newcommand{\exponethreeKhigherDoneepsiter}{\ensuremath{0.57}}
\newcommand{\exponethreeKhigherDoneepsnum}{\ensuremath{0.21}}
\newcommand{\exponethreeKhigherDonesq}{\ensuremath{0.5}}
\newcommand{\exponethreeKhigherDonelengthS}{\ensuremath{8}}

%% file: numbers/exp14_KlowD1N6_data.txt
\newcommand{\exponefourKlowDoneNsixK}{\ensuremath{14}}
\newcommand{\exponefourKlowDoneNsixmmax}{\ensuremath{18}}
\newcommand{\exponefourKlowDoneNsixmmin}{\ensuremath{5}}
\newcommand{\exponefourKlowDoneNsixb}{\ensuremath{4}}
\newcommand{\exponefourKlowDoneNsixN}{\ensuremath{6}}
\newcommand{\exponefourKlowDoneNsixgamma}{\ensuremath{0}}
\newcommand{\exponefourKlowDoneNsixD}{\ensuremath{1}}
\newcommand{\exponefourKlowDoneNsixcfT}{\ensuremath{4}}
\newcommand{\exponefourKlowDoneNsixcfV}{\ensuremath{4}}
\newcommand{\exponefourKlowDoneNsixreltol}{\ensuremath{10^{-10}}}
\newcommand{\exponefourKlowDoneNsixTVtol}{\ensuremath{0}}
\newcommand{\exponefourKlowDoneNsixB}{\ensuremath{3594}}
\newcommand{\exponefourKlowDoneNsixtone}{\ensuremath{-8.0}}
\newcommand{\exponefourKlowDoneNsixttwo}{\ensuremath{0.31}}
\newcommand{\exponefourKlowDoneNsixtKmone}{\ensuremath{12.18}}
\newcommand{\exponefourKlowDoneNsixtK}{\ensuremath{15.11}}
\newcommand{\exponefourKlowDoneNsixtmin}{\ensuremath{7.35}}
\newcommand{\exponefourKlowDoneNsixtmax}{\ensuremath{70.01}}
\newcommand{\exponefourKlowDoneNsixT}{\ensuremath{9.53}}
\newcommand{\exponefourKlowDoneNsixkappa}{\ensuremath{24.1883}}
\newcommand{\exponefourKlowDoneNsixcalg}{\ensuremath{0.29}}
\newcommand{\exponefourKlowDoneNsixupper}{\ensuremath{1.29}}
\newcommand{\exponefourKlowDoneNsixlower}{\ensuremath{0.71}}
\newcommand{\exponefourKlowDoneNsixlambdaonetrue}{\ensuremath{6.77932848602532}}
\newcommand{\exponefourKlowDoneNsixlambdaoneIO}{\ensuremath{6.77932848603416}}
\newcommand{\exponefourKlowDoneNsixlambdaoneIonly}{\ensuremath{6.77932848606342}}
\newcommand{\exponefourKlowDoneNsixepsiter}{\ensuremath{0.29}}
\newcommand{\exponefourKlowDoneNsixepsnum}{\ensuremath{0.36}}
\newcommand{\exponefourKlowDoneNsixsq}{\ensuremath{0.5}}
\newcommand{\exponefourKlowDoneNsixlengthS}{\ensuremath{7}}

%% file: numbers/exp21_KlowD2_data.txt
\newcommand{\exptwooneKlowDtwoK}{\ensuremath{14}}
\newcommand{\exptwooneKlowDtwommax}{\ensuremath{18}}
\newcommand{\exptwooneKlowDtwommin}{\ensuremath{5}}
\newcommand{\exptwooneKlowDtwob}{\ensuremath{4}}
\newcommand{\exptwooneKlowDtwoN}{\ensuremath{4}}
\newcommand{\exptwooneKlowDtwogamma}{\ensuremath{4}}
\newcommand{\exptwooneKlowDtwoD}{\ensuremath{2}}
\newcommand{\exptwooneKlowDtwocfT}{\ensuremath{4}}
\newcommand{\exptwooneKlowDtwocfV}{\ensuremath{4}}
\newcommand{\exptwooneKlowDtworeltol}{\ensuremath{10^{-7}}}
\newcommand{\exptwooneKlowDtwoTVtol}{\ensuremath{0}}
\newcommand{\exptwooneKlowDtwoB}{\ensuremath{3594}}
\newcommand{\exptwooneKlowDtwotone}{\ensuremath{-8.0}}
\newcommand{\exptwooneKlowDtwottwo}{\ensuremath{0.31}}
\newcommand{\exptwooneKlowDtwotKmone}{\ensuremath{12.18}}
\newcommand{\exptwooneKlowDtwotK}{\ensuremath{15.11}}
\newcommand{\exptwooneKlowDtwotmin}{\ensuremath{1.38}}
\newcommand{\exptwooneKlowDtwotmax}{\ensuremath{54.5}}
\newcommand{\exptwooneKlowDtwoT}{\ensuremath{39.61}}
\newcommand{\exptwooneKlowDtwokappa}{\ensuremath{10.9096}}
\newcommand{\exptwooneKlowDtwocalg}{\ensuremath{0.57}}
\newcommand{\exptwooneKlowDtwoupper}{\ensuremath{1.57}}
\newcommand{\exptwooneKlowDtwolower}{\ensuremath{0.43}}
\newcommand{\exptwooneKlowDtwolambdaonetrue}{\ensuremath{0.93676803545}}
\newcommand{\exptwooneKlowDtwolambdaDtrue}{\ensuremath{(0.93676803545, 1.03129973333)}}
\newcommand{\exptwooneKlowDtwolambdaoneIO}{\ensuremath{0.93676803569}}
\newcommand{\exptwooneKlowDtwolambdaDIO}{\ensuremath{(0.93676803569, 1.03129973404)}}
\newcommand{\exptwooneKlowDtwoepsiter}{\ensuremath{0.57}}
\newcommand{\exptwooneKlowDtwoepsnum}{\ensuremath{0.21}}
\newcommand{\exptwooneKlowDtwosq}{\ensuremath{0.5}}
\newcommand{\exptwooneKlowDtwolengthS}{\ensuremath{7}}

%% file: numbers/exp22_KhighD2_data.txt
\newcommand{\exptwotwoKhighDtwoK}{\ensuremath{24}}
\newcommand{\exptwotwoKhighDtwommax}{\ensuremath{18}}
\newcommand{\exptwotwoKhighDtwommin}{\ensuremath{5}}
\newcommand{\exptwotwoKhighDtwob}{\ensuremath{4}}
\newcommand{\exptwotwoKhighDtwoN}{\ensuremath{4}}
\newcommand{\exptwotwoKhighDtwogamma}{\ensuremath{4}}
\newcommand{\exptwotwoKhighDtwoD}{\ensuremath{2}}
\newcommand{\exptwotwoKhighDtwocfT}{\ensuremath{4}}
\newcommand{\exptwotwoKhighDtwocfV}{\ensuremath{4}}
\newcommand{\exptwotwoKhighDtworeltol}{\ensuremath{10^{-7}}}
\newcommand{\exptwotwoKhighDtwoTVtol}{\ensuremath{0}}
\newcommand{\exptwotwoKhighDtwoB}{\ensuremath{3594}}
\newcommand{\exptwotwoKhighDtwotone}{\ensuremath{-8.0}}
\newcommand{\exptwotwoKhighDtwottwo}{\ensuremath{0.31}}
\newcommand{\exptwotwoKhighDtwotKmone}{\ensuremath{39.84}}
\newcommand{\exptwotwoKhighDtwotK}{\ensuremath{44.41}}
\newcommand{\exptwotwoKhighDtwotmin}{\ensuremath{1.38}}
\newcommand{\exptwotwoKhighDtwotmax}{\ensuremath{162.29}}
\newcommand{\exptwotwoKhighDtwoT}{\ensuremath{117.95}}
\newcommand{\exptwotwoKhighDtwokappa}{\ensuremath{10.9096}}
\newcommand{\exptwotwoKhighDtwocalg}{\ensuremath{0.57}}
\newcommand{\exptwotwoKhighDtwoupper}{\ensuremath{1.57}}
\newcommand{\exptwotwoKhighDtwolower}{\ensuremath{0.43}}
\newcommand{\exptwotwoKhighDtwolambdaoneIO}{\ensuremath{0.9366102686}}
\newcommand{\exptwotwoKhighDtwolambdaDIO}{\ensuremath{(0.9366102686, 1.0311035381)}}
\newcommand{\exptwotwoKhighDtwoepsiter}{\ensuremath{0.57}}
\newcommand{\exptwotwoKhighDtwoepsnum}{\ensuremath{0.21}}
\newcommand{\exptwotwoKhighDtwosq}{\ensuremath{0.5}}
\newcommand{\exptwotwoKhighDtwolengthS}{\ensuremath{8}}

%% file: numbers/exp42_KhighD4_data.txt
\newcommand{\expfourtwoKhighDfourK}{\ensuremath{24}}
\newcommand{\expfourtwoKhighDfourmmax}{\ensuremath{18}}
\newcommand{\expfourtwoKhighDfourmmin}{\ensuremath{5}}
\newcommand{\expfourtwoKhighDfourb}{\ensuremath{4}}
\newcommand{\expfourtwoKhighDfourN}{\ensuremath{4}}
\newcommand{\expfourtwoKhighDfourgamma}{\ensuremath{4}}
\newcommand{\expfourtwoKhighDfourD}{\ensuremath{4}}
\newcommand{\expfourtwoKhighDfourcfT}{\ensuremath{4}}
\newcommand{\expfourtwoKhighDfourcfV}{\ensuremath{4}}
\newcommand{\expfourtwoKhighDfourreltol}{\ensuremath{10^{-7}}}
\newcommand{\expfourtwoKhighDfourTVtol}{\ensuremath{0}}
\newcommand{\expfourtwoKhighDfourB}{\ensuremath{3594}}
\newcommand{\expfourtwoKhighDfourtone}{\ensuremath{-8.0}}
\newcommand{\expfourtwoKhighDfourttwo}{\ensuremath{0.31}}
\newcommand{\expfourtwoKhighDfourtKmone}{\ensuremath{39.84}}
\newcommand{\expfourtwoKhighDfourtK}{\ensuremath{44.41}}
\newcommand{\expfourtwoKhighDfourtmin}{\ensuremath{1.38}}
\newcommand{\expfourtwoKhighDfourtmax}{\ensuremath{162.29}}
\newcommand{\expfourtwoKhighDfourT}{\ensuremath{117.95}}
\newcommand{\expfourtwoKhighDfourkappa}{\ensuremath{10.9096}}
\newcommand{\expfourtwoKhighDfourcalg}{\ensuremath{0.57}}
\newcommand{\expfourtwoKhighDfourupper}{\ensuremath{1.57}}
\newcommand{\expfourtwoKhighDfourlower}{\ensuremath{0.43}}
\newcommand{\expfourtwoKhighDfourlambdaoneIO}{\ensuremath{0.93661026532}}
\newcommand{\expfourtwoKhighDfourlambdaDIO}{\ensuremath{(0.93661026532, 1.03110353551, 2.34898509027, 2.36402837386)}}
\newcommand{\expfourtwoKhighDfourepsiter}{\ensuremath{0.57}}
\newcommand{\expfourtwoKhighDfourepsnum}{\ensuremath{0.21}}
\newcommand{\expfourtwoKhighDfoursq}{\ensuremath{0.5}}
\newcommand{\expfourtwoKhighDfourlengthS}{\ensuremath{8}}

%% file: numbers/exp43_KhigherD4_data.txt
\newcommand{\expfourthreeKhigherDfourK}{\ensuremath{26}}
\newcommand{\expfourthreeKhigherDfourmmax}{\ensuremath{18}}
\newcommand{\expfourthreeKhigherDfourmmin}{\ensuremath{5}}
\newcommand{\expfourthreeKhigherDfourb}{\ensuremath{4}}
\newcommand{\expfourthreeKhigherDfourN}{\ensuremath{4}}
\newcommand{\expfourthreeKhigherDfourgamma}{\ensuremath{4}}
\newcommand{\expfourthreeKhigherDfourD}{\ensuremath{4}}
\newcommand{\expfourthreeKhigherDfourcfT}{\ensuremath{4}}
\newcommand{\expfourthreeKhigherDfourcfV}{\ensuremath{4}}
\newcommand{\expfourthreeKhigherDfourreltol}{\ensuremath{10^{-7}}}
\newcommand{\expfourthreeKhigherDfourTVtol}{\ensuremath{0}}
\newcommand{\expfourthreeKhigherDfourB}{\ensuremath{3594}}
\newcommand{\expfourthreeKhigherDfourtone}{\ensuremath{-8.0}}
\newcommand{\expfourthreeKhigherDfourttwo}{\ensuremath{0.31}}
\newcommand{\expfourthreeKhigherDfourtKmone}{\ensuremath{47.24}}
\newcommand{\expfourthreeKhigherDfourtK}{\ensuremath{52.12}}
\newcommand{\expfourthreeKhigherDfourtmin}{\ensuremath{1.38}}
\newcommand{\expfourthreeKhigherDfourtmax}{\ensuremath{191.27}}
\newcommand{\expfourthreeKhigherDfourT}{\ensuremath{139.01}}
\newcommand{\expfourthreeKhigherDfourkappa}{\ensuremath{10.9096}}
\newcommand{\expfourthreeKhigherDfourcalg}{\ensuremath{0.57}}
\newcommand{\expfourthreeKhigherDfourupper}{\ensuremath{1.57}}
\newcommand{\expfourthreeKhigherDfourlower}{\ensuremath{0.43}}
\newcommand{\expfourthreeKhigherDfourlambdaoneIO}{\ensuremath{0.93660704628}}
\newcommand{\expfourthreeKhigherDfourlambdaDIO}{\ensuremath{(0.93660704628, 1.03110079261, 2.34898242886, 2.36402576124)}}
\newcommand{\expfourthreeKhigherDfourepsiter}{\ensuremath{0.57}}
\newcommand{\expfourthreeKhigherDfourepsnum}{\ensuremath{0.21}}
\newcommand{\expfourthreeKhigherDfoursq}{\ensuremath{0.5}}
\newcommand{\expfourthreeKhigherDfourlengthS}{\ensuremath{8}}

%% file: numbers/exp44_KlowD4N6_data.txt
\newcommand{\expfourfourKlowDfourNsixK}{\ensuremath{14}}
\newcommand{\expfourfourKlowDfourNsixmmax}{\ensuremath{18}}
\newcommand{\expfourfourKlowDfourNsixmmin}{\ensuremath{5}}
\newcommand{\expfourfourKlowDfourNsixb}{\ensuremath{4}}
\newcommand{\expfourfourKlowDfourNsixN}{\ensuremath{6}}
\newcommand{\expfourfourKlowDfourNsixgamma}{\ensuremath{0}}
\newcommand{\expfourfourKlowDfourNsixD}{\ensuremath{4}}
\newcommand{\expfourfourKlowDfourNsixcfT}{\ensuremath{4}}
\newcommand{\expfourfourKlowDfourNsixcfV}{\ensuremath{4}}
\newcommand{\expfourfourKlowDfourNsixreltol}{\ensuremath{10^{-7}}}
\newcommand{\expfourfourKlowDfourNsixTVtol}{\ensuremath{0}}
\newcommand{\expfourfourKlowDfourNsixB}{\ensuremath{3594}}
\newcommand{\expfourfourKlowDfourNsixtone}{\ensuremath{-8.0}}
\newcommand{\expfourfourKlowDfourNsixttwo}{\ensuremath{0.31}}
\newcommand{\expfourfourKlowDfourNsixtKmone}{\ensuremath{12.18}}
\newcommand{\expfourfourKlowDfourNsixtK}{\ensuremath{15.11}}
\newcommand{\expfourfourKlowDfourNsixtmin}{\ensuremath{7.35}}
\newcommand{\expfourfourKlowDfourNsixtmax}{\ensuremath{70.01}}
\newcommand{\expfourfourKlowDfourNsixT}{\ensuremath{9.53}}
\newcommand{\expfourfourKlowDfourNsixkappa}{\ensuremath{24.1883}}
\newcommand{\expfourfourKlowDfourNsixcalg}{\ensuremath{0.29}}
\newcommand{\expfourfourKlowDfourNsixupper}{\ensuremath{1.29}}
\newcommand{\expfourfourKlowDfourNsixlower}{\ensuremath{0.71}}
\newcommand{\expfourfourKlowDfourNsixlambdaonetrue}{\ensuremath{6.77932848603}}
\newcommand{\expfourfourKlowDfourNsixlambdaDtrue}{\ensuremath{(6.77932848603, 7.07168841064, 8.83204278272, 8.95810708913)}}
\newcommand{\expfourfourKlowDfourNsixlambdaoneIO}{\ensuremath{6.77932848603}}
\newcommand{\expfourfourKlowDfourNsixlambdaDIO}{\ensuremath{(6.77932848603, 7.07168841064, 8.83204300284, 8.95810708913)}}
\newcommand{\expfourfourKlowDfourNsixepsiter}{\ensuremath{0.29}}
\newcommand{\expfourfourKlowDfourNsixepsnum}{\ensuremath{0.36}}
\newcommand{\expfourfourKlowDfourNsixsq}{\ensuremath{0.5}}
\newcommand{\expfourfourKlowDfourNsixlengthS}{\ensuremath{7}}

%% file: numbers/exp61_KlowD6_data.txt
\newcommand{\expsixoneKlowDsixK}{\ensuremath{14}}
\newcommand{\expsixoneKlowDsixmmax}{\ensuremath{18}}
\newcommand{\expsixoneKlowDsixmmin}{\ensuremath{5}}
\newcommand{\expsixoneKlowDsixb}{\ensuremath{4}}
\newcommand{\expsixoneKlowDsixN}{\ensuremath{4}}
\newcommand{\expsixoneKlowDsixgamma}{\ensuremath{4}}
\newcommand{\expsixoneKlowDsixD}{\ensuremath{6}}
\newcommand{\expsixoneKlowDsixcfT}{\ensuremath{4}}
\newcommand{\expsixoneKlowDsixcfV}{\ensuremath{4}}
\newcommand{\expsixoneKlowDsixreltol}{\ensuremath{10^{-7}}}
\newcommand{\expsixoneKlowDsixTVtol}{\ensuremath{0}}
\newcommand{\expsixoneKlowDsixB}{\ensuremath{3594}}
\newcommand{\expsixoneKlowDsixtone}{\ensuremath{-8.0}}
\newcommand{\expsixoneKlowDsixttwo}{\ensuremath{0.31}}
\newcommand{\expsixoneKlowDsixtKmone}{\ensuremath{12.18}}
\newcommand{\expsixoneKlowDsixtK}{\ensuremath{15.11}}
\newcommand{\expsixoneKlowDsixtmin}{\ensuremath{1.38}}
\newcommand{\expsixoneKlowDsixtmax}{\ensuremath{54.5}}
\newcommand{\expsixoneKlowDsixT}{\ensuremath{39.61}}
\newcommand{\expsixoneKlowDsixkappa}{\ensuremath{10.9096}}
\newcommand{\expsixoneKlowDsixcalg}{\ensuremath{0.57}}
\newcommand{\expsixoneKlowDsixupper}{\ensuremath{1.57}}
\newcommand{\expsixoneKlowDsixlower}{\ensuremath{0.43}}
\newcommand{\expsixoneKlowDsixlambdaonetrue}{\ensuremath{0.93676803545}}
\newcommand{\expsixoneKlowDsixlambdaDtrue}{\ensuremath{(0.93676803545, 1.03129973333, 2.34915988893, 2.36420024713, 2.76449828512, 3.05384074982)}}
\newcommand{\expsixoneKlowDsixlambdaoneIO}{\ensuremath{0.93676803547}}
\newcommand{\expsixoneKlowDsixlambdaDIO}{\ensuremath{(0.93676803547, 1.03129973335, 2.34915988897, 2.36420024721, 2.76449828527, 3.05384075103)}}
\newcommand{\expsixoneKlowDsixepsiter}{\ensuremath{0.57}}
\newcommand{\expsixoneKlowDsixepsnum}{\ensuremath{0.21}}
\newcommand{\expsixoneKlowDsixsq}{\ensuremath{0.5}}
\newcommand{\expsixoneKlowDsixlengthS}{\ensuremath{7}}

%% file: numbers/exp62_KhighD6_data.txt
\newcommand{\expsixtwoKhighDsixK}{\ensuremath{24}}
\newcommand{\expsixtwoKhighDsixmmax}{\ensuremath{18}}
\newcommand{\expsixtwoKhighDsixmmin}{\ensuremath{5}}
\newcommand{\expsixtwoKhighDsixb}{\ensuremath{4}}
\newcommand{\expsixtwoKhighDsixN}{\ensuremath{4}}
\newcommand{\expsixtwoKhighDsixgamma}{\ensuremath{4}}
\newcommand{\expsixtwoKhighDsixD}{\ensuremath{6}}
\newcommand{\expsixtwoKhighDsixcfT}{\ensuremath{4}}
\newcommand{\expsixtwoKhighDsixcfV}{\ensuremath{4}}
\newcommand{\expsixtwoKhighDsixreltol}{\ensuremath{10^{-7}}}
\newcommand{\expsixtwoKhighDsixTVtol}{\ensuremath{0}}
\newcommand{\expsixtwoKhighDsixB}{\ensuremath{3594}}
\newcommand{\expsixtwoKhighDsixtone}{\ensuremath{-8.0}}
\newcommand{\expsixtwoKhighDsixttwo}{\ensuremath{0.31}}
\newcommand{\expsixtwoKhighDsixtKmone}{\ensuremath{39.84}}
\newcommand{\expsixtwoKhighDsixtK}{\ensuremath{44.41}}
\newcommand{\expsixtwoKhighDsixtmin}{\ensuremath{1.38}}
\newcommand{\expsixtwoKhighDsixtmax}{\ensuremath{162.29}}
\newcommand{\expsixtwoKhighDsixT}{\ensuremath{117.95}}
\newcommand{\expsixtwoKhighDsixkappa}{\ensuremath{10.9096}}
\newcommand{\expsixtwoKhighDsixcalg}{\ensuremath{0.57}}
\newcommand{\expsixtwoKhighDsixupper}{\ensuremath{1.57}}
\newcommand{\expsixtwoKhighDsixlower}{\ensuremath{0.43}}
\newcommand{\expsixtwoKhighDsixlambdaoneIO}{\ensuremath{0.93661026528}}
\newcommand{\expsixtwoKhighDsixlambdaDIO}{\ensuremath{(0.93661026528, 1.0311035355, 2.34898507071, 2.36402837387, 2.76410052757, 3.05335461469)}}
\newcommand{\expsixtwoKhighDsixepsiter}{\ensuremath{0.57}}
\newcommand{\expsixtwoKhighDsixepsnum}{\ensuremath{0.21}}
\newcommand{\expsixtwoKhighDsixsq}{\ensuremath{0.5}}
\newcommand{\expsixtwoKhighDsixlengthS}{\ensuremath{8}}

%% file: numbers/exp63_KhigherD6_data.txt
\newcommand{\expsixthreeKhigherDsixK}{\ensuremath{26}}
\newcommand{\expsixthreeKhigherDsixmmax}{\ensuremath{18}}
\newcommand{\expsixthreeKhigherDsixmmin}{\ensuremath{5}}
\newcommand{\expsixthreeKhigherDsixb}{\ensuremath{4}}
\newcommand{\expsixthreeKhigherDsixN}{\ensuremath{4}}
\newcommand{\expsixthreeKhigherDsixgamma}{\ensuremath{4}}
\newcommand{\expsixthreeKhigherDsixD}{\ensuremath{6}}
\newcommand{\expsixthreeKhigherDsixcfT}{\ensuremath{4}}
\newcommand{\expsixthreeKhigherDsixcfV}{\ensuremath{4}}
\newcommand{\expsixthreeKhigherDsixreltol}{\ensuremath{10^{-7}}}
\newcommand{\expsixthreeKhigherDsixTVtol}{\ensuremath{0}}
\newcommand{\expsixthreeKhigherDsixB}{\ensuremath{3594}}
\newcommand{\expsixthreeKhigherDsixtone}{\ensuremath{-8.0}}
\newcommand{\expsixthreeKhigherDsixttwo}{\ensuremath{0.31}}
\newcommand{\expsixthreeKhigherDsixtKmone}{\ensuremath{47.24}}
\newcommand{\expsixthreeKhigherDsixtK}{\ensuremath{52.12}}
\newcommand{\expsixthreeKhigherDsixtmin}{\ensuremath{1.38}}
\newcommand{\expsixthreeKhigherDsixtmax}{\ensuremath{191.27}}
\newcommand{\expsixthreeKhigherDsixT}{\ensuremath{139.01}}
\newcommand{\expsixthreeKhigherDsixkappa}{\ensuremath{10.9096}}
\newcommand{\expsixthreeKhigherDsixcalg}{\ensuremath{0.57}}
\newcommand{\expsixthreeKhigherDsixupper}{\ensuremath{1.57}}
\newcommand{\expsixthreeKhigherDsixlower}{\ensuremath{0.43}}
\newcommand{\expsixthreeKhigherDsixlambdaoneIO}{\ensuremath{0.9366070462}}
\newcommand{\expsixthreeKhigherDsixlambdaDIO}{\ensuremath{(0.9366070462, 1.03110079255, 2.34898241157, 2.36402576104, 2.7640942154, 3.0533480309)}}
\newcommand{\expsixthreeKhigherDsixepsiter}{\ensuremath{0.57}}
\newcommand{\expsixthreeKhigherDsixepsnum}{\ensuremath{0.21}}
\newcommand{\expsixthreeKhigherDsixsq}{\ensuremath{0.5}}
\newcommand{\expsixthreeKhigherDsixlengthS}{\ensuremath{8}}